\documentclass[11pt]{amsart}
\usepackage{soul}
\usepackage{mathtools}
\usepackage{amsmath}
\usepackage{amsthm}
\usepackage{amssymb}
\usepackage{dsfont}
\usepackage[english]{babel}
\usepackage[utf8x]{inputenc}
\usepackage{overpic}
\usepackage{enumerate}
\usepackage{pdflscape}
\usepackage{float}
\usepackage{hyperref}
\usepackage[all,cmtip]{xy}
\usepackage{todonotes}

\newtheoremstyle{pl}% ⟨name⟩
{3pt}% ⟨Space above⟩
{3pt}% ⟨Space below⟩
{\itshape}% ⟨Body font⟩
{}% ⟨Indent amount⟩
{\scshape}% ⟨Theorem head font⟩
{.}% ⟨Punctuation after theorem head⟩
{.5em}% ⟨Space after theorem head⟩
{}% ⟨Theorem head spec (can be left empty, meaning ‘normal’)⟩

\newtheoremstyle{pl*}% ⟨name⟩
{3pt}% ⟨Space above⟩
{3pt}% ⟨Space below⟩
{\itshape}% ⟨Body font⟩
{}% ⟨Indent amount⟩
{\bfseries}% ⟨Theorem head font⟩
{.}% ⟨Punctuation after theorem head⟩
{.5em}% ⟨Space after theorem head⟩
{}% ⟨Theorem head spec (can be left empty, meaning ‘normal’)⟩

\newtheoremstyle{pl**}% ⟨name⟩
{3pt}% ⟨Space above⟩
{3pt}% ⟨Space below⟩
{\itshape}% ⟨Body font⟩
{}% ⟨Indent amount⟩
{\bfseries}% ⟨Theorem head font⟩
{.}% ⟨Punctuation after theorem head⟩
{.5em}% ⟨Space after theorem head⟩
{}% ⟨Theorem head spec (can be left empty, meaning ‘normal’)⟩

\newtheoremstyle{mythm}% ⟨name⟩
{3pt}% ⟨Space above⟩
{3pt}% ⟨Space below⟩
{\itshape}% ⟨Body font⟩
{}% ⟨Indent amount⟩
{\bfseries}% ⟨Theorem head font⟩
{.}% ⟨Punctuation after theorem head⟩
{.5em}% ⟨Space after theorem head⟩
{\thmnote{#1 }#3}% ⟨Theorem head spec (can be left empty, meaning ‘normal’)⟩

\newtheoremstyle{myappendix}% ⟨name⟩
{3pt}% ⟨Space above⟩
{3pt}% ⟨Space below⟩
{\itshape}% ⟨Body font⟩
{}% ⟨Indent amount⟩
{\scshape}% ⟨Theorem head font⟩
{.}% ⟨Punctuation after theorem head⟩
{.5em}% ⟨Space after theorem head⟩
{\thmnote{#1 }#3}% ⟨Theorem head spec (can be left empty, meaning ‘normal’)⟩

\newtheoremstyle{df}% ⟨name⟩
{3pt}% ⟨Space above⟩
{3pt}% ⟨Space below⟩
{\normalfont}% ⟨Body font⟩
{}% ⟨Indent amount⟩
{\scshape}% ⟨Theorem head font⟩
{.}% ⟨Punctuation after theorem head⟩
{.5em}% ⟨Space after theorem head⟩
{}% ⟨Theorem head spec (can be left empty, meaning ‘normal’)⟩

\newtheoremstyle{rm}% ⟨name⟩
{3pt}% ⟨Space above⟩
{3pt}% ⟨Space below⟩
{\normalfont}% ⟨Body font⟩
{}% ⟨Indent amount⟩
{\scshape}% ⟨Theorem head font⟩
{.}% ⟨Punctuation after theorem head⟩
{.5em}% ⟨Space after theorem head⟩
{}% ⟨Theorem head spec (can be left empty, meaning ‘normal’)⟩

\theoremstyle{pl}
\newtheorem{thm}{Theorem}[section]
\newtheorem{lem}[thm]{Lemma}			
\newtheorem{cor}[thm]{Corollary}
\newtheorem{pro}[thm]{Proposition}

\newtheorem*{qn*}{Question}
\newtheorem*{lem*}{Lemma}

\theoremstyle{pl*}
\newtheorem{thm*}{Theorem}
\newtheorem{pro*}[thm*]{Proposition}
\newtheorem{cor*}[thm*]{Corollary}

\theoremstyle{pl**}
\newtheorem*{thm**}{Theorem}

\theoremstyle{mythm}
\newtheorem*{mythm}{Theorem}

\theoremstyle{myappendix}
\newtheorem*{lemma}{Lemma}
\newtheorem*{proposition}{Proposition}

\theoremstyle{df}
\newtheorem*{dfn}{Definition}

\theoremstyle{rm}

\newcommand{\ep}{
\epsilon
}

\newcommand{\qa}[1]{
\left[#1\right]
}
\newcommand{\ga}[1]{
\left\{#1\right\}
}
\newcommand{\mc}[1]{
\mathcal{#1}
}
\newcommand{\mb}[1]{
\mathbb{#1}
}
\newcommand{\T}{
\mc{T}
}

\def\epsilon{\varepsilon}

\begin{document}

\title{Uniform models and short curves for random 3-manifolds}

\author[Feller]{Peter Feller}
\address{Department of Mathematics, ETH Z\"{u}rich, Switzerland}
\email{peter.feller@math.ch}

\author[Sisto]{Alessandro Sisto}
	\address{Department of Mathematics, Heriot-Watt University, Edinburgh, UK}
	\email{a.sisto@hw.ac.uk}

\author[Viaggi]{Gabriele Viaggi}
\address{Department of Mathematics, Heidelberg University, Germany}
\email{gviaggi@mathi.uni-heidelberg.de}

\def\subjclassname{\textup{2020} Mathematics Subject Classification}
\expandafter\let\csname subjclassname@1991\endcsname=\subjclassname
\expandafter\let\csname subjclassname@2000\endcsname=\subjclassname
%\thanks{AMS 2020 subject classification: 57K32, 30F60, 20P05
%58C40, 30F60, 20P05
%}
%\date{October 30, 2020}

\begin{abstract}
We provide two constructions of hyperbolic metrics on 3-manifolds with Heegaard splittings that satisfy certain topological conditions, which both apply to random Heegaard splitting with asymptotic probability 1.

These constructions provide a lot of control on the resulting metric, allowing us to prove various results about the coarse growth rate of geometric invariants, such as diameter and injectivity radius, and about arithmeticity and commensurability in families of random 3-manifolds. For example, we show that the diameter of a random Heegaard splitting grows coarsely linearly in the length of the associated random walk.

The constructions only use tools from the deformation theory of Kleinian groups, that is, we do not rely on the solution of the Geometrization Conjecture by Perelman. In particular, we give a proof of Maher's result that random 3-manifolds are hyperbolic that bypasses Geometrization.
% Both constructions apply to random Heegaard splitting with asymptotic probability 1. The first one provides explicit uniform bilipschitz models for the hyperbolic metric. The second one gives a general criterion for a curve on a Heegaard surface in a 3-manifold to be a short geodesic for the hyperbolic metric on the 3-manifold. Such curves are abundant in a random setting. \pf{\st{As an applicationof the model metrics,} 
% 
% We use our constructions to recover Maher's result that random 3-manifolds are hyperbolic in a more explicit way (that allows to describe properties of the hyperbolic structure obtained)
% (TODO:maybe delete this braket or phrase it better)}.
% \pf{We illustrate the strength of this approach by proving new results about the coarse growth rate of geometric invariants, such as diameter and injectivity radius, and about arithmeticity and commensurability in families of random 3-manifolds. For example, we determine with asymptotic probability 1 the diameter of a random Heegaard splitting up to a multiplicative constant.}
\end{abstract}

\maketitle

\section{Introduction}
Every closed orientable 3-manifold $M$ can be presented as a {\em Heegaard splitting}. This means that $M$ is diffeomorphic to a 3-manifold $M_f$ obtained by gluing together two handlebodies (taking the second one with opposite orientation) of the same genus $H_g$ along an orientation preserving diffeomorphism $f\in  \mathrm{Diff}^+(\Sigma)$ of their boundary $\Sigma:=\partial H_g$ 
\[
M_f=H_g\cup_{f:\partial H_g\to\partial H_g}H_g.
\]
For a fixed genus $g$, however, not all 3-manifolds can arise. In this article, we restrict our attention to the family of those 3-manifolds that can be described as Heegaard splittings of a {fixed genus} $g\ge 2$.

The problem of finding hyperbolic {{metrics}} on ``most'' 3-manifolds with a splitting of a fixed genus $g\ge 2$ was originally raised by Thurston (as Problem 24 in~\cite{Thu82}) and made more precise by Dunfield and Thurston (see Conjecture 2.11 of \cite{DT06}) by introducing the notion of { random Heegaard splittings}. 

This notion is based on the observation that the diffeomorphism type of $M_f$ only depends on the isotopy class of the gluing map $f$, so it is well-defined for elements in the {\em mapping class group} 
\[
[f]\in{\rm Mod}(\Sigma):={\rm Diff}^+(\Sigma)/{\rm Diff}_0(\Sigma).
\] 
Therefore, Heegaard splittings of genus $g\ge 2$ are naturally parametrized by mapping classes $[f]\in{\rm Mod}(\Sigma)$.

A family $(M_n)_{n\in\mb{N}}$ of random Heegaard splittings of genus $g\ge 2$, or {\em random 3-manifolds}, is one of the form $M_n=M_{f_n}$ where $(f_n)_{n\in\mb{N}}$ is a random walk on the mapping class group ${\rm Mod}(\Sigma)$ driven by some initial probability measure $\mu$ whose finite support generates ${\rm Mod}(\Sigma)$. If $(f_n)_{n\in\mb{N}}$ is such a random walk, we will denote by $\mb{P}_n$ the distribution of the $n$-th step $f_n$ and by $\mb{P}$ the distribution of the path $(f_n)_{n\in\mb{N}}$. 

Exploiting work of Hempel~\cite{Hempel:Heegaard} and the solution of the Geometrization Conjecture by Perelman, Maher showed in \cite{Maher:Heegaard} that a random Heegaard splitting %$M_f$
of genus $g\ge 2$ admits a hyperbolic metric, thus answering Dunfield and Thurston's conjecture.

The main goal of this article is to provide a {constructive} and {effective} approach to the hyperbolization of random 3-manifolds. Before describing our main contributions (Theorems~\ref{main1} and~\ref{thm:main}),
we note that %this approach yields a constructive proof of Maher's result.
this approach yields a proof of Maher's result that does not rely on Perelman's work, which we can informally state as follows.
%The main goal of this article is to provide a {\em constructive} and {\em effective} approach to the hyperbolization of random 3-manifolds. Before describing this approach and the main results (Theorems~\ref{main1} and~\ref{thm:main}), we note that this approach yields a constructive proof of Maher's result.

\begin{thm*}
\label{main2}
There is a Ricci flow free hyperbolization for random 3-manifolds.
\end{thm*}  

By Ricci flow free hyperbolization we mean that we construct the hyperbolic metric only using tools from the deformation theory of Kleinian groups. More specifically, the main tools we use are the model manifold technology by Minsky~\cite{M10} and Brock, Canary and Minsky~\cite{BrockCanaryMinsky:ELC2}, as well as the effective version of Thurston's Hyperbolic Dehn Surgery by Hodgson and Kerckhoff~\cite{HK:universal}.
%and Brock and Bromberg's Drilling Theorem~\cite{BB04}. 

%\todo{I commented out this paragraph. The point is not to use Perelman. So saying in the same sentence that we still you Maher (which after all is not that scary) seems a bit weird. And we explain this later anyway, so why say it here?}%We remark that, even though we do not rely on Perelman's solution of the Geometrization Conjecture, we do use the main result from Maher~\cite{Maher:Heegaard}, namely, the fact that the {Hempel distance} of the Heegaard splittings (see Hempel~\cite{Hempel:Heegaard}) grows coarsely linearly along the random walk.

We develop two different approaches to Theorem~\ref{main2}, both bringing new and more refined information than the mere existence of a hyperbolic metric. 

One approach provides a so-called {model metric} that captures, up to {uniform bilipschitz distortion}, the geometry of 3-manifolds admitting suitable Heegaard splitting and allows the computation of its geometric invariants. The hyperbolic metric is constructed explicitly by gluing elementary building blocks. The other approach provides a purely topological criterion for a Heegaard splitting to admit a hyperbolic metric for which a specific simple closed curve on the Heegaard surface is short.
% This is a general criterion for hyperbolicity\pf{ from the data of a Heegaard splitting, rather than a statement for a random Heegaard splittings}, and 
Both criteria apply to a large class of Heegaard splittings, not just random splittings.
%\todo{I commented out ''The existence of the hyperbolic metric will be guaranteed by appealing to Thurston's Hyperbolization Theorem.'' since this seems confusing since really its first this then hyperbolic Dehn surgery}
%The existence of the hyperbolic metric will be guaranteed by appealing to Thurston's Hyperbolization Theorem.

We describe these two approaches and our main results (Theorem~\ref{main1} and Theorem~\ref{thm:main}) in the next two subsections. But first we list some applications to random 3-manifolds, which will be discussed in more detail later.

The first two results describe the behaviour of geometric invariants, namely diameter and injectivity radius.

\begin{thm*}
\label{main4}
There exists $c>0$ such that
\[
\mb{P}_n[M_f\text{ is hyperbolic and }{\rm diam}(M_f)\in[n/c,cn]]\stackrel{n\to\infty}{\longrightarrow}1.
\]
\end{thm*}

\begin{thm*}
\label{main6}
There exists $c>0$ such that
\[
\mb{P}_n\left[M_f\text{ is hyperbolic and }{\rm inj}(M_f)\le c/\log(n)^2\right]\stackrel{n\to\infty}{\longrightarrow}1.
\]
\end{thm*}

Our methods would also allow us to obtain a result similar to Theorem~\ref{main4} for the volume, but this (and more) is known already in that case, as we discuss below.

As described in~\cite{ST:large_proj}, $1/\log(n)^2$ is exactly the coarse decay rate for the length of the shortest curve in random mapping tori. Our methods, however, only give an upper bound in the case of random Heegaard splittings.

The third application has a more algebraic flavor, showing that random 3-manifolds are not arithmetic and belong to multiple commensurability classes.

\begin{thm*}
\label{main3}
With asymptotic probability $1$ the following holds
\begin{enumerate}
\item{$M_f$ is not arithmetic.}
\item{$M_f$ is not in a fixed commensurability class $\mc{R}$.}
\end{enumerate}
\end{thm*}

\subsection*{Uniform bilipschitz models for random 3-manifolds}

The notion of model manifold that we use is similar to the ones considered by Brock, Minsky, Namazi and Souto in~\cite{Na05}, \cite{NS09}, \cite{BMNS16}. A Riemannian metric $(M_f,\rho)$ is a {\em $\ep$-model metric} for $\ep<1/2$ if there is a decomposition into five pieces $M_f=H_1\cup\Omega_1\cup Q\cup\Omega_2\cup H_2$ satisfying the three requirements 
\begin{enumerate}
\item{%Topologically,
$H_1$ and $H_2$ are homeomorphic to genus $g$ handlebodies while $Q,\Omega_1$ and $\Omega_2$ are homeomorphic to $\Sigma\times[0,1]$.}
\item{%Geometrically,
$\rho$ has negative curvature ${\rm sec}\in(-1-\ep,-1+\ep)$, but outside the region $\Omega=\Omega_1\cup\Omega_2$ the metric is purely hyperbolic, i.e. ${\rm sec}=-1$.}
\item{The piece $Q$ is almost isometrically embeddable in a complete hyperbolic 3-manifold diffeomorphic to $\Sigma\times\mb{R}$.}
\end{enumerate} 

The importance of the last requirement is due to the fact that we understand explicitly hyperbolic 3-manifolds diffeomorphic to $\Sigma\times\mb{R}$ thanks to the work of Minsky~\cite{M10} and Brock, Canary and Minsky \cite{BrockCanaryMinsky:ELC2}. 

The following is a more precise version of Theorem~\ref{main2}.

\begin{thm*}
\label{main1}
For every $0<\ep<1/2$ and $K>1$ we have
\[
\mb{P}_n[\text{\rm $M_f$ has a hyperbolic metric $K$-bilipschitz to a $\ep$-model metric}]\stackrel{n\to\infty}{\longrightarrow}1.
\]
\end{thm*}

We remark that $\ep$-model metrics on random Heegaard splittings, similar to the ones that we build here, are constructed in~\cite{HV}. There, the existence of a underlying hyperbolic metric is guaranteed by Maher's result, and it is unclear whether the $\ep$-model metrics are uniformly bilipschitz to it. 

As an aside, we also mention that, using an unpublished result by Tian~\cite{Ti90}, the mere fact that a metric $\rho$ is a $\ep$-model metric and that the regions $\Omega_1,\Omega_2$ where it is not hyperbolic have uniformly bounded diameter (as follows from~\cite{HV}), implies, if $\ep>0$ is sufficiently small, that $\rho$ is uniformly close up to third derivatives to a hyperbolic metric. However, we do not rely on Tian's result. Instead, in order to provide a uniform bilipschitz control, we exploit ergodic properties of the random walk and drilling and filling theorems by Hodgson and Kerckhoff~\cite{HK:universal} and Brock and Bromberg~\cite{BB04}.   

Our methods follow closely~\cite{BMNS16} and~\cite{BD15} where uniform $\ep$-model metrics are constructed for special classes of 3-manifolds. 

The idea is the following. We can obtain a hyperbolic metric on $M_f$ by Hodgson and Kerckhoff's effective Hyperbolic Dehn Surgery~\cite{HK:universal}%\todo{%G: The difference between hyperbolic Dehn surgery and cone manifold deformation is that the latter keeps track of the isotopy class of the core geodesic while the former keeps track only of the homotopy class.
%Maybe we can just say that it is Hodgson and Kerckhoff version of HDS. P: I thought its ok to be vague here but I agree, so i changed it} %a hyperbolic cone manifold deformation
from a complete finite volume hyperbolic metric on a {\em drilled} manifold $\mb{M}$ which has the following form. Let $\Sigma\times[1,4]$ be a tubular neighborhood of $\Sigma\subset M_f$. We consider 3-manifolds diffeomorphic to
\[
\mb{M}=M_f-(P_1\times\{1\}\cup P_2\times\{2\}\cup P_3\times\{3\}\cup P_4\times\{4\})
\]
where $P_j$ is a pants decomposition of the surface $\Sigma\times\{j\}$. A finite volume hyperbolic metric on such a manifold can be constructed explicitly by gluing together the convex cores of two maximally cusped handlebodies $H_1,H_2$ and three maximally cusped I-bundles $\Omega_1,Q,\Omega_2$.
\[
\mb{M}=H_1\cup\Omega_1\cup Q\cup\Omega_2\cup H_2.
\]
Most of our work consists of finding suitable pants decompositions for which the Dehn surgery slopes needed to pass from $\mb{M}$ to $M_f$ satisfy the assumptions of the effective Hyperbolic Dehn Surgery Theorem~\cite{HK:universal}. In order to find them we crucially need two major tools. On one hand an explicit control on the geometry of hyperbolic handlebodies similar to the one obtained in~\cite{HV}, and on the other hand ergodic properties of the random walks proved by Baik, Gekhtman and Hamenstädt~\cite{BGH16}.

%\todo{G: This is a very important assumption that we do need. Maybe we can emphasize it at a different place. P: I agree it is important, but in the intro we said it already and it jsut seemed out of place here.}
%We stress the fact that for both Theorem~\ref{main2} and Theorem~\ref{main1} we assume that the support of $\mu$ is finite and generates the {entire} mapping class group.

\subsection*{Short curves via knots on Heegaard surfaces}
We now discuss short closed geodesics in hyperbolic Heegaard splittings and random 3-manifolds. We identify purely topological conditions on a simple closed curve on the Heegaard surface $\gamma\subset\Sigma$ that ensure that $M_f$ has a hyperbolic metric and $\gamma$ is a very short geodesic in it.

Let $\mc{D}$ and $f\mc{D}$ be the {\em disk sets} of the Heegaard surface $\Sigma\subseteq M_f$, i.e.~$\mc{D}$ and $f\mc{D}$ are the subsets of the {curve graph} $\mc{C}$ of the Heegaard surface $\Sigma$ given by the essential simple closed curves $\delta\subset\Sigma$ that {compress} to in the first and second handlebody of the Heegaard surface, respectively. For a subsurface $W\subset\Sigma$, we denote by
$d_W(\mc{D},f\mc{D})$ the distance in the curve graph of $W$ of the subsurface projections of $\mc{D}$ and $f\mc{D}$.

\begin{thm*}
\label{thm:main}
Let $\Sigma:=\partial H_g$ for a fixed $g\geq 2$. There exists a constant $C_\Sigma>0$ such that the following holds. Let $\gamma\subset\Sigma$ be a non-separating simple closed curve with complement $W:=\Sigma-\gamma$. Let $f\in{\rm Mod}(\Sigma)$ be a mapping class. Suppose that
\begin{enumerate}[(a)]
\item{Both $(H_g,\gamma)$ and $(H_g, f(\gamma))$ are pared acylindrical handlebodies.}
\item{We have a large subsurface projection $d_W(\mc{D},f\mc{D})\ge C_\Sigma$.} 
\end{enumerate}
Then $M_f$ has a hyperbolic metric. Moreover, the length of $\gamma$ in $M_f$ is bounded by
\[
\ell_{M_f}(\gamma)\le C_\Sigma/d_{W}(\mc{D},f\mc{D}).
\] 
\end{thm*}

We recall that, informally speaking, the pair $(H_g,\gamma)$ is called a {\em pared acylindrical handlebody} if $\Sigma-\gamma$ is incompressible and there are no non-trivial essential cylinders in $(H_g,\gamma)$ and $(H_g,\Sigma-\gamma)$. These objects arise naturally in the study of cusped hyperbolic {{metrics}} on $H_g$ (see Thurston~\cite{Thu86}). Many pairs $(H_g,\gamma)$ have this property. For example, if $\gamma\subset\Sigma$ satisfies $d_{\mc{C}}(\gamma,\mc{D})\ge 3$, then $(H_g,\gamma)$ is pared acylindrical. As the disk set $\mc{D}$ is a small quasi-convex subset of $\mc{C}$ by Masur and Minsky~\cite{MM95}, non-separating curves $\gamma\in\mc{C}$ that are far from $\mc{D}$ are abundant. 

Theorem~\ref{thm:main} builds upon the groundbreaking work of Minsky~\cite{M10} on hyperbolic {{metrics}} on $\Sigma\times\mb{R}$. In that setting, the collection of simple closed curves on $\Sigma$ that are isotopic to very short closed geodesics on a hyperbolic {{metrics}} on $\Sigma\times\mb{R}$ can be read off the list of {subsurface coefficients} associated to the {end invariants} of such a {{metric}}. 
To some extent, in a complicated Heegaard splitting $M_f$, the role of the end invariants can be, as a first approximation, replaced by the disk sets $\mc{D}$ and $f\mc{D}$ of the splitting. For sufficiently complicated hyperbolic Heegaard splittings we have the following conjectural description of length of curves. A curve $\gamma\subset\Sigma$ is isotopic to a short geodesic if and only if it lies on the boundary $\gamma\subset\partial W$ of a proper essential subsurface $W\subset\Sigma$ where the disk sets $\mc{D}$ and $f\mc{D}$ have large {subsurface projection} $d_W(\mc{D},f\mc{D})$. We understand Theorem~\ref{thm:main} as a first step in the direction of making precise this conjectural description.

The idea for the proof of Theorem~\ref{thm:main} is the following. We associate to $\gamma\subset\Sigma$ the 3-manifold $M_f-\gamma$. Notice that it decomposes as
\[
M_f-\gamma=(H_g-\gamma)\cup(H_g-f(\gamma)).
\]
If both $(H_g,\gamma)$ and $(H_g,f(\gamma))$ are pared acylindrical handlebodies, then the JSJ theory tells us that the complement $M_f-\gamma$ is irreducible and atoroidal. Moreover it is also Haken since we can choose $\Sigma-\gamma\subset M_f-\gamma$ as a Haken surface. Hence, by Thurston's Hyperbolization Theorem, the 3-manifold $M_f-\gamma$ admits a complete finite volume hyperbolic {{metric}}.

As in the proof of Theorem~\ref{main1}, we find a hyperbolic metric on $M_f$ for which $\gamma$ is a short curve provided that the filling slope in the complete finite volume hyperbolic metric on $M_f-\gamma$ satisfies the assumptions of the effective Hyperbolic Dehn Surgery Theorem \cite{HK:universal}. To this extent we argue that the size of a standard torus horosection of the cusp of $M_f-\gamma$ is comparable with the subsurface projection $d_W(\mc{D},f\mc{D})$. In order to check this we use tools from the model manifold technology by Minsky~\cite{M10} and Brock, Canary and Minsky \cite{BrockCanaryMinsky:ELC2}. 

%\todo{I commented out the sentence on how obvious the upper bound is, since we no longer do that and instead just cite other people}
%The upper bound on the length immediately follows from Hodgson and Kerckhoff's deformation theory of hyperbolic cone manifolds.

Ergodic properties of random walks on the mapping class group imply that the condition of large subsurface projection of $\mc{D}$ and $f_n\mc{D}$ on the complement $W_n$ of some non-separating curves $\gamma_n\subset\Sigma$ holds with asymptotic probability one, that is, with $\mb{P}_n\to1$. Thus, Theorem~\ref{thm:main} applies to random 3-manifolds and gives a proof of Theorem~\ref{main2} that does not rely on Perelman's work.

%Notice that the curve $\gamma$ will be short in $M_f$, namely, we have a uniform upper bound on the length $\ell_{M_f}(\gamma)<\ep_\Sigma$.
We briefly comment on the assumption (a) of Theorem~\ref{thm:main}. Its use is twofold: It implies that the complement $M_f-\gamma$ is hyperbolizable and it also implies that the inclusions $\Sigma-\gamma,\Sigma-f(\gamma)\subset H_g$ are {doubly incompressible} (as defined by Thurston~\cite{Thu86}). The latter plays a central role in the proof via Thurston's Uniform Injectivity~\cite{Thu86}.
As mentioned above, since $\mc{D}$ and $f\mc{D}$ are a small quasi-convex subsets of the curve graph, there are plenty of non-separating curves that satisfy $d_{\mc{C}}(\gamma,\mc{D}\cup f\mc{D})\ge 3$ and, hence, condition (a). %They all yield cusped hyperbolic manifolds $M_f-\gamma$.

In a direction opposite to Theorem~\ref{thm:main}, one can ask whether every very short curve on a complicated hyperbolic Heegaard splitting arises from a Hyperbolic Dehn Surgery on a complete hyperbolic manifold of the form $M_f-\gamma$ with $\gamma\subset\Sigma$ a simple closed curve. This is the case for {strongly irreducible} hyperbolic Heegaard splittings $M_f$ as proved by Souto~\cite{S08} and Breslin~\cite{Bre11}. They show that there is a constant $\ep_\Sigma>0$ such that, every closed geodesic of length at most $\ep_\Sigma$ in $M_f$ is isotopic to a simple closed curve $\gamma$ on the Heegaard surface $\Sigma\subset M_f$.

%\todo{This paragraph discusses the natural question about a converse of of Thorem 3. P: Agreed. I still think its a bit much in the intro, and breaks the flow. I keep it if you really think it adds to the story we are trying to tell. But I did comment out the first sentence, since I dont think its clear enough.}

%\todo{I commented out this paragraph since I find it hard for the reader to appreciated it at this point. Maybe it should be at the end of the proof of Thm 3.}%Concerning, instead, the restriction on $W$ and condition (b) we remark that the choice of $W$ as the complement of a non-separating simple closed curve $\gamma\subset\Sigma$ is not essential, but allows some simplifications and is enough for our application to random 3-manifolds. More important is the fact that $W$ is a {non-annular} subsurface because, in this case, a large subsurface projection is reflected in the size of a standard torus horosection of the cusp of $M_f-\gamma$. 

\subsection*{On the applications}

Now that we discussed our constructions, we can explain how we obtain the applications that we listed earlier, i.e.~Theorems~\ref{main4},~\ref{main6}, and~\ref{main3}.

We exploit the geometric control given by the $\ep$-model metric to compute the coarse growth or decay rate of the geometric invariants along the family $(M_{f_n})_{n\in\mb{N}}$ (see also \cite[Section 11.4]{Riv14}).
Our general strategy is the following. We use the model manifold technology \cite{M10}, \cite{BrockCanaryMinsky:ELC2} and compute the geometric invariants for the middle piece $Q$ of the $\ep$-model metric. Then, we argue that the invariants of $Q$ are uniformly comparable with those of $M_f$. 
In this way, Theorem~\ref{main1} allows for a uniform approach to several results.

For example, combined with a result of Brock \cite{Brock:pants}, Theorem~\ref{main1} allows the computation of the coarse growth rate of the volume, which is well-known to be linear as explained in~\cite{Maher:Heegaard} (see also \cite{HV}). In fact, there is even a law of large numbers for the volume \cite{V:volume}. Combined with results of Baik, Gekhtman and Hamenstädt~\cite{BGH16} Theorem~\ref{main1} shows that the smallest positive eigenvalue of the Laplacian behaves like $1/n^2$ as computed in~\cite{HV}. %We point out that Theorem~\ref{main1} allows a uniform approach to those results. 
We do not carry out those computations because they are already well established. 

To control the diameter of random 3-manifolds (Theorem~\ref{main4}), the ingredients of the proof are again Theorem~\ref{main1} and a result by White~\cite{W01}. 

Instead, to control the injectivity radius (Theorem~\ref{main6}), we use work of the second author and Taylor~\cite{ST:large_proj}. 

Finally, the proof of Theorem~\ref{main3} about arithmeticity and commensurability classes combines a study of geometric limits of random 3-manifolds, see Proposition~\ref{main5}, with arguments by Biringer and Souto~\cite{BS11}.

%%%%%%%%%%%%%%%%%%%%%%%%%%%%%%%%%%%%%%%%%%%%%%%%%%%%%%%%%%%%%%%%%%%%
\subsection*{Overview}
The paper is organized as follows.
Section \ref{preliminaries} introduces some ingredients and tools that we need in our constructions.
The rest of the article is divided into three parts as follows.
Part~\ref{part:shortcurves}: Short curves on Heegaard splittings (Sections~\ref{short curves} and~\ref{large horosection}).
Part~\ref{part:model}: The construction of the model metric (Sections~\ref{overview} and~\ref{examples}).
Part~\ref{part:random&appl}: The application to random 3-manifolds (Section~\ref{proof1}) and the computation of the coarse growth rate of geometric invariants (Section~\ref{applications}). Next, we briefly describe the content of each section.

In Section~\ref{short curves} we describe the topological part of the proof of Theorem~\ref{thm:main}, in particular, we check that $M_f-\gamma$ is hyperbolizable. The geometric part is developed, instead, in Section~\ref{large horosection} where we use the model manifold technology to check that the cusp of $M_f-\gamma$ has a large horosection.  

In Section~\ref{overview} we outline the construction of the $\ep$-model metric. In Section~\ref{examples} we build many examples to which the model metric construction applies. 

In Section~\ref{proof1} we prove Theorems~\ref{main2} and~\ref{main1} by showing that the examples of Section~\ref{examples} and the ones described by Theorem~\ref{thm:main} are generic from the point of view of a random walk. 
Lastly, in Section~\ref{applications} we prove Theorems~\ref{main4},~\ref{main6} and~\ref{main3}. 

%%%%%%%%%%%%%%%%%%%%%%%%%%%%%%%%%%%%%%%%%%%%%%%%%%%%%%%%%%%%%%%%%%%%
\subsection*{Acknowledgements}

Part of this work is contained in the PhD thesis of GV. He wishes to thank Ursula Hamenstädt for her help and support. He also acknowledges the financial support of the Max Planck Institute for Mathematics of Bonn and partial funding by the DFG 427903332.

PF and AS want to follow up on GV's acknowledgement above.
A previous arxiv version~\cite{Viaggi_V1} of this article contained GV's path to Theorem 1, a result PF and AS had independently announced in~\cite{OWR}.
PF and AS want to acknowledge that the GV had independently pursued and worked out a complete proof of Theorem~\ref{main2} via Theorem~\ref{main1} as part of his impressive PhD-work.
They are most thankful for GV's insight and flexibility in merging the two works into a coherent work, which exceeds the sum of its parts and which, in particular, includes further applications we would not have found separately.

We thank an anonymous referee for feedback that led to substantial reworking and improvement of Section~\ref{examples}.

PF and AS gratefully acknowledge support by the Swiss National Science Foundation (grants \#181199 and \#182186, respectively). 
%%%%%%%%%%%%%%%%%%%%%%%%%%%%%%%%%%%%%%%%%%%%%%%%%%%%%%%%%%%%%%

\section{Preliminaries}
\label{preliminaries}

In this section we review some of the main ingredients used in one or both of the two constructions that we are going to describe. 

In both constructions, geometrically finite hyperbolic structures on handlebodies $H_g$ and I-bundles $\Sigma\times[0,1]$ play an important role: They appear as building blocks (in Sections \ref{overview} and \ref{examples}) or coverings (in Sections \ref{short curves} and \ref{large horosection}) of our drilled Heegaard splittings $M_f$. We introduce in this section some terminology and facts from the deformation theory of hyperbolic {{metrics}} on 3-manifolds (as discussed, for example, in Chapter 7 of \cite{CMcC}). 

The curve graph $\mc{C}$ and the disk set $\mc{D}\subset\mc{C}$ are also objects that appear in both constructions. By fundamental work of Minsky \cite{M10} and Brock, Canary and Minsky \cite{BrockCanaryMinsky:ELC2}, the geometry and combinatorics of these graphs captures the internal geometry of hyperbolic manifolds diffeomorphic to $\Sigma\times\mb{R}$. The combinatorics of $\mc{C}$ also carries a lot of information on the topology of Heegaard splittings $M_f$: For example, by work of Hempel \cite{Hempel:Heegaard}, if the distance $d_{\mc{C}}(\mc{D},f\mc{D})$ is at least 3, then the splitting $M_f$ satisfies the topological assumptions of the Geometrization Conjecture. In this section we will briefly describe both aspects of these graphs.      

The last common ingredient that we present is the Hodgson and Kerckhoff~\cite{HK:universal} deformation theory of hyperbolic 3-manifolds $M$ with cone singularities along geodesic links $\Gamma\subset M$.
%\todo{I commented out this sentence, since why should it be here?}%In some circumstances, it is possible to deform the hyperbolic structure on $M$ and vary the cone angles along $\Gamma$. For example, in some cases it is possible to bring the cone angle from 0 to $2\pi$, a deformation called Hyperbolic Dehn Surgery that connects two complete non-singular hyperbolic metrics.
The Drilling Theorem of Brock and Bromberg \cite{BB04} allows to quantitatively keep under control the change in the hyperbolic {{metric}} along the deformation.  

\subsection{Hyperbolic 3-manifolds}
Let $M$ be a connected oriented 3-manifold (without boundary). A \emph{hyperbolic metric} on $M$ is a complete Riemannian metric of constant sectional curvature $-1$. A {\em hyperbolic structure} on $M$ is a pair $(M,\rho)$, where $\rho$ is a complete hyperbolic metric. Every hyperbolic structure is isometric to a quotient $%(M,\rho)\simeq
\mb{H}^3/\Gamma$ of the hyperbolic 3-space $\mb{H}^3$ by a discrete torsion free subgroup $\Gamma<{\rm Isom}^+(\mb{H}^3)$. We refer to a connected oriented 3-manifold that admits a hyperbolic structure as a \emph{hyperbolic manifold}. Hyperbolic structures of finite volume are unique, up to homotopy, by the Mostow–Prasad rigidity theorem. Hyperbolic manifolds with totally geodesic boundary can also be defined and will play a role in Part \ref{part:model}.

The group $\Gamma$ acts on the boundary $\mb{CP}^1=\partial\mb{H}^3$ by Möbius transformations. The action $\Gamma\curvearrowright\mb{CP}^1$ preserves the {\em limit set} $\Lambda\subset\mb{CP}^1$, which is the set of accumulation points of an orbit $\Gamma o\subset\mb{H}^3\cup\partial\mb{H}^3$, and its complement $\Omega:=\mb{CP}^1-\Lambda$, the {\em domain of discontinuity}. The action $\Gamma\curvearrowright\Omega$ is free and properly discontinuous. The quotient $\mb{H}^3\cup\Omega/\Gamma$ is a 3-manifold with boundary $\Omega/\Gamma$ and the boundary has a natural conformal structure. If $\Gamma$ is finitely generated and non-abelian, then, by Ahlfors Finiteness Theorem~\cite{Ah64}, the quotient $\Omega/\Gamma$ is a surface of finite type. In particular, $\Omega/\Gamma$ admits a {\em Poincaré metric} which is the unique complete finite area hyperbolic metric in the same conformal class. 

\subsection{Teichmüller space}
We denote by $\T(S)$ or $\T$ the Teichmüller space of conformal structures or, equivalently, hyperbolic metrics on a surface $S$ up to isotopy. For a given $\ep>0$, we denote by $\T_\ep\subset\T$ the subset of Teichmüller space consisting of those hyperbolic surfaces in which no closed geodesic is shorter than $\ep$. By the Mumford Compactness Theorem (see Theorem 12.6 of~\cite{primer}), the mapping class group ${\rm Mod}(\Sigma)$ acts cocompactly on $\T_\ep$. 

\subsection{Geometrically finite structures}
We describe some flexible classes of geometrically finite hyperbolic structures on a pair of fixed topological spaces, namely handlebodies $H_g$ and I-bundles $\Sigma\times[a,b]$. When dealing with such hyperbolic structures with cusps, the notion of pared (acylindrical) manifold naturally occurs (see Thurston \cite{Thu86}, \cite{Thu86(3)}). %For convenience, we briefly recall it.

\begin{dfn}[Pared Acylindrical]
Let $N$ be either $H_g$ or $\Sigma\times[a,b]$. Set $A:=S^1\times[0,1]$. Let $P\subset\partial N$ be a multicurve with regular neighborhood $U\subset\partial N$. We say that $(N,P)$ is a {\em pared 3-manifold} provided that the following holds.
\begin{enumerate}
\item{The inclusion $\partial N-P\subset N$ is $\pi_1$-injective.}
\item{Every $\pi_1$-injective homotopy $(A,\partial A)\to (N,U)$ is homotopic as a map of pairs to a map that sends $A$ to $U$.}
\end{enumerate}
We say that $(N,P)$ is also {\em acylindrical} if furthermore
\begin{enumerate}
\setcounter{enumi}{2}
\item{Every essential map $(A,\partial A)\to (N,\partial N-U)$ is homotopic as a map of pairs into $\partial N-U$.}
\end{enumerate}
\end{dfn}

We refer to Chapter 5 of \cite{CMcC} for the topological features of pared (acylindrical) manifolds. We can now give the following definition:

\begin{dfn}[Geometrically Finite]
Let $N$ be either $H_g$ or $\Sigma\times[a,b]$. Let $P\subset\partial N$ be a (possibly empty) multicurve such that $(N,P)$ is pared. A hyperbolic metric $\rho$ on $M={\rm int}(N)$ is {\em geometrically finite} with parabolic locus $P\subset\partial N$ if there is an orientation preserving diffeomorphism $f:N-P\to\mb{H}^3\cup\Omega/\Gamma$ that is isometric on $M={\rm int}(N)$. Such an identification induces a conformal structure on the boundary via $f:\partial N-P\to\Omega/\Gamma$.
\end{dfn}

If we have different identifications $f:N-P\to\mb{H}^3\cup\Omega/\Gamma$ and $f':N-P\to\mb{H}^3\cup\Omega'/\Gamma'$, the composition $g=f'f^{-1}:\mb{H}^3\cup\Omega/\Gamma\to\mb{H}^3\cup\Omega'/\Gamma'$ restricts to an isometry $g:\mb{H}^3\cup\Omega/\Gamma\to\mb{H}^3\cup\Omega'/\Gamma'$ and, hence, the restriction to the boundary $g:\Omega/\Gamma\to\Omega'/\Gamma'$ is conformal. Thus, every geometrically finite hyperbolic metric induces a well-defined Riemann surface structure on $\partial N-P$, called the {\em conformal boundary}, which we can think as a point in $\T(\partial N-P)$.

Two geometrically finite hyperbolic metrics $\rho,\rho'$ on $M={\rm int}(N)$ are said to be \emph{equivalent} if there is a diffeomorphism isotopic to the identity $\phi:N-P\to N-P$ such that $\rho=\phi^*\rho'$. The restriction of the diffeomorphism to the boundary $\phi:\partial N-P\to\partial N-P$ is isotopic to the identity and conformal with respect to the Riemann surface structures induced by $\rho$ and $\rho'$. Therefore, equivalent geometrically finite hyperbolic metrics have equivalent conformal boundaries.

\begin{dfn}[Convex Core]
 Every geometrically finite hyperbolic metric on $M={\rm int}(N)$ admits a {\em convex core} $\mc{CC}(N)\subset M$, which is the smallest convex closed subset whose inclusion in $N$ is a homotopy equivalence. 
\end{dfn}

\subsection{Maximally cusped structures}
Maximally cusped structures play a crucial role in Part \ref{part:model}. Let $N$ be either $H_g$ or $\Sigma\times[a,b]$.

\begin{dfn}[Maximally Cusped]
A hyperbolic structure on $M={\rm int}(N)$ is {\em maximally cusped} or {\em maximally cusped at $P$} if it is geometrically finite and has parabolic locus $P$ which is a pants decomposition of $\partial N$.
\end{dfn}

% Maximally cusped hyperbolic structure can be thought as lying on the boundary of the deformation spaces of convex cocompact ones (see~\cite[Theorem III]{KLMBS93} and~\cite[Theorem 5.1]{Oh98}). 

Work of Maskit~\cite{Maskit83}, Keen, Maskit and Series~\cite{KLMBS93}, and Ohshika~\cite{Oh98} gives the following.

\begin{thm}[see Theorem 7.2.9 in \cite{CMcC}]
\label{maximal cusps}
The following holds.
\begin{itemize}
\item{For every pants decomposition $P$ of $\Sigma=\partial H_g$ such that $(H_g,P)$ is a {\rm pared acylindrical 3-manifold} there exists a unique (up to isotopy) hyperbolic metric on ${\rm int}(H_g)$ which is maximally cusped at $P$. We denote such a metric by $H(P)$.}
\item{For every pants decompositions $P,R$ of $\Sigma$ such that $(\Sigma\times[a,b],P\times\{a\}\sqcup R\times\{b\})$ is a {\rm pared acylindrical 3-manifold} there exists a unique (up to isotopy) hyperbolic metric on $\Sigma\times[a,b]$ which is maximally cusped at $P\times\{a\}\sqcup R\times\{b\}$. We denote such a metric by $Q(P,R)$.}
\end{itemize}
\end{thm}

% \subsection{Geometry and topology of convex cores}
% Let $N$ be either $H_g$ or $\Sigma\times[a,b]$.

Regarding convex cores, work of Keen, Maskit and Series~\cite{KLMBS93} provides the following description.

\begin{thm}
\label{kms}
The following holds.
\begin{itemize}
\item{Let $P$ be a pants decomposition of $\Sigma=\partial H_g$ such that $(H_g,P)$ is {\rm pared acylindrical}. Let $H(P)$ be the corresponding maximally cusped structure. The convex core $\mc{CC}(H(P))\subset H_g$ is isotopic to $H_g-P$ and has totally geodesic boundary.}
\item{Let $P,R$ be pants decompositions of $\Sigma$ such that $(\Sigma\times[a,b],P\times\{a\}\sqcup R\times\{b\})$ is {\rm pared acylindrical}. Let $Q(P,R)$ be the corresponding maximally cusped structure. The convex core $\mc{CC}(Q(P,R))\subset\Sigma\times[a,b]$ is isotopic to $\Sigma\times[a,b]-(P\times\{a\}\cup R\times\{b\})$ and has totally geodesic boundary.}
\end{itemize}
\end{thm}

\subsection{Curve graph and disk set}
We now introduce the curve graph and the disk sets and relate them to the topology of Heegaard splittings and handlebodies. We also describe some of their geometric properties.

\begin{dfn}[Curve Graph and Disk Set]
The {\em curve graph} $\mc{C}:=\mc{C}(\Sigma)$ of a compact orientable surface $\Sigma$ is the graph whose vertices are the isotopy classes of essential non-peripheral simple closed curves. Two vertices are joined by an edge of length 1 if the corresponding curves can be realized disjointly on $\Sigma$. We endow the graph $\mc{C}$ with the intrinsic path metric $d_{\mc{C}}$. In fact, we will only consider the vertex set of $\mc{C}$ endowed with its induced metric, which we also denoted by $\mc{C}$ by abuse of notation.

If $\Sigma=\partial H_g$ is the boundary of a handlebody $H_g$, we have an associated \emph{disk set} $\mc{D}\subset\mc{C}$ consisting of those essential simple closed curves of $\Sigma$ that bound a properly embedded disk in $H_g$.  

If $f:\Sigma\to\Sigma$ is a gluing map defining the Heegaard splitting $M_f=H_g\cup_f H_g$, then the {\em Hempel distance} of $f$ is the quantity
\[
d_{\mc{C}}(\mc{D},f\mc{D}).
\]
\end{dfn}

Splittings with large Hempel distance have strong topological properties: 

\begin{thm}[Hempel \cite{Hempel:Heegaard}]
\label{hempel}
Let $f$ be a gluing map. If the Hempel distance of $f$ is at least 3, then the Heegaard splitting $M_f$ is irreducible, does not contain any essential torus or Klein bottle, and is not Seifert fibered.
\end{thm}

Combined with the solution of the Geometrization Conjecture by Perelman, Theorem \ref{hempel} implies that if $d_{\mc{C}}(\mc{D},f\mc{D})\ge 3$, then the splitting $M_f$ is hyperbolic. In \cite{Maher:Heegaard}, Maher proved that the Hempel distance of a random mapping class grows coarsely linearly and, hence, that random Heegaard splittings are hyperbolic.

In our random 3-manifolds setup, we will also exploit the curve graph to control the topology of our building blocks. For example, we will use it to check that the blocks satisfy the assumptions of Theorems \ref{maximal cusps} and \ref{uniforminj}: From a curve graph point of view, we have the following useful criterion.

\begin{lem}
\label{lem:d>=3impliesparedacyl}
Let $\gamma,\gamma'\subset\Sigma$ be essential multicurves. We have
\begin{enumerate}[{\rm (i)}]
\item{If $d_{\mc{C}}(\gamma,\mc{D})\ge 2$ then $(H_g,\gamma)$ is pared.}
\item{If $d_{\mc{C}}(\gamma,\mc{D})\ge 3$ then $(H_g,\gamma)$ is pared acylindrical.}
\item{If $d_{\mc{C}}(\gamma,\gamma')\ge 1$ then $(\Sigma\times[0,1],\gamma\times\{0\}\sqcup\gamma'\times\{1\})$ is pared.}
\item{If $d_{\mc{C}}(\gamma,\gamma')\ge 3$ then $(\Sigma\times[0,1],\gamma\times\{0\}\sqcup\gamma'\times\{1\})$ is pared acylindrical.}
\end{enumerate}
\end{lem}

The proof can be found in Appendix~\ref{appendix c}.

From a geometric point of view, Masur and Minsky proved that the curve graph $\mc{C}$ is a Gromov hyperbolic space (see \cite{MasurMinsky:I}) and the disk set $\mc{D}\subset\mc{C}$ is a uniformly quasi-convex subspace (see \cite{MM95}). 

As discovered by Minsky \cite{M00}, \cite{M01}, \cite{M10}, and Brock, Canary and Minsky \cite{BrockCanaryMinsky:ELC2} in groundbreaking work that led to the solution of the Ending Lamination Conjecture, the geometry of the curve graph is related to the internal geometry of hyperbolic {metrics} on $\Sigma\times[0,1]$: The relation is established via {\em end invariants} and the devices of {\em subsurface projections} and {\em hierarchies of tight geodesics} introduced in \cite{MasurMinsky:II}.

We briefly recall the definition of subsurface projection in the case that is most relevant for us.

\begin{dfn}[Non-annular Subsurface Projection]
Let $W\subset\Sigma$ be a proper compact connected non-annular subsurface of $\Sigma$ that is not a three-holed sphere. Let $\alpha\in\mc{C}$ be any curve. The {\em subsurface projection} of $\alpha$ to $W$ is the (possibly empty) subset $\pi_W(\alpha)$ of the curve graph $\mc{C}(W)$ of all possible essential surgeries of $\alpha\cap W$ (see Section 2 of \cite{MasurMinsky:II} for more details).

If both $\alpha$ and $\beta$ intersect $W$ essentially, we define
\[
d_W(\alpha,\beta)={\rm diam}_{\mc{C}(W)}(\pi_W(\alpha)\cup\pi_W(\beta)).
\] 
\end{dfn}

The geometry of the curve graph is also tied to the geometry of Teichmüller space via the \emph{shortest curves projection} $\Upsilon:\T\to\mc{C}$, a coarsely defined map that associates to every marked hyperbolic surface $X\in\T$ a shortest geodesic on it $\Upsilon(X)$.

We recall a few properties of $\Upsilon$ due to Masur and Minsky \cite{MasurMinsky:I}. It follows from~\cite{MasurMinsky:I} that $\Upsilon$ is Lipschitz and sends Teichmüller geodesics to uniform {unparametrized} quasi-geodesics. The latter means that there is a constant $B$ only depending on $\Sigma$ such that 
\[
d_{\mc{C}}(\Upsilon(X),\Upsilon(Z))\ge d_{\mc{C}}(\Upsilon(X),\Upsilon(Y))+d_{\mc{C}}(\Upsilon(Y),\Upsilon(Z))-B
\]
for every $X<Y<Z$ aligned on a Teichmüller geodesic. The situation improves when we restrict $\Upsilon$ to a Teichmüller geodesic entirely contained in $\T_\ep$: By work of Hamenstädt~\cite{H10}, there exists $L_\ep$ such that if $l:I\to\T$ is a Teichmüller geodesic entirely contained in $\T_\ep$, then $\Upsilon l:I\to\mc{C}$ is a $L_\ep$-quasi-geodesic.

We will need an observation, which says roughly the following. Consider a Teichmüller segment  $\gamma$ that starts at a point projecting in the curve graph to the disk set $\mc{D}$. Then the projection of $\gamma$ to the curve graph has an initial segment that fellow-travels $\mc{D}$, and then the projection starts moving linearly away from $\mc{D}$.

\begin{lem}
\label{far from disks}
There exists $h_0$ such that the following holds. For every Teichmüller segment $\gamma:[0,T]\to\T$ with $\gamma(0)=o$ and $\Upsilon(o)\in\mc{D}$ there exist $0\le t_0\le t_1\le T$ with the following properties
\begin{itemize}
\item{$d_{\mc{C}}(\Upsilon\gamma(t),\mc{D})\le h_0$ for every $t\in[0,t_0]$ and $d_{\mc{C}}(\Upsilon\gamma(t),\mc{D})\ge 2h_0$ for every $t\in[t_1,T]$.}
%\item{If $d_{\mc{C}}(\Upsilon\gamma(t),\mc{D})\ge 2h_0$ then $t\in[t_1,T]$.}
\item{If $t_1<s<s'<T$ then 
\[
d_{\mc{C}}(\Upsilon\gamma(s),\mc{D})\ge d_{\mc{C}}(\Upsilon\gamma(s),\Upsilon\gamma(s'))+d_{\mc{C}}(\Upsilon\gamma(s),\mc{D})-h_0.
\]
}
\end{itemize}
\end{lem}

\begin{proof} 
Recall that $\Upsilon(o)\in\mc{D}$. If the composition $\Upsilon\gamma$ was a geodesic, then, the lemma would follow from the quasi-convexity of $\mc{D}\subset\mc{C}$ and hyperbolicity of $\mc{C}$: Initially, on a subsegment $[0,t_0]$, $\Upsilon\gamma(t)$ stays close to $\mc{D}$. As soon as $\Upsilon\gamma(t)$ gets sufficiently far from $\mc{D}$ (if ever), say at time $t_1\ge t_0$, it never comes back, meaning that the distance between $\mc{D}$ and $\Upsilon\gamma(t)$ starts to increase linearly. All the properties follow from this picture. 

Now, $\Upsilon\gamma$ is not exactly a geodesic. However, the restriction is a uniform unparametrized quasi-geodesic, meaning that there is a constant $L$ only depending on $\Sigma$ and a homeomorphism $\rho:[0,T']\to[0,T]$ such that the composition $\Upsilon\gamma\rho$ is a $L$-quasi-geodesic. Hence, as $\mc{C}$ is hyperbolic, $\Upsilon\gamma[0,T]$ lies uniformly close to a geodesic $[\Upsilon\gamma(0),\Upsilon\gamma(T)]$ for which the above description holds. A careful bookkeeping yields the desired properties.
\end{proof}

\subsection{Filling and drilling}
One of the most important tools needed in our constructions is the Universal Dehn Surgery Theorem by Hodgson and Kerckhoff~\cite{HK:universal}. We briefly recall the basic setup and statements.

Let $M$ be a compact oriented $3$-manifold with boundary. For a component of the boundary $T$ that is a torus, a choice $\mu$ of a homotopy class of a non-null homotopic simple closed curve in $T$, called a {\em slope}. Gluing a solid torus $S^1\times D^2$ to $M$ via a gluing homeomorphism
$\partial (S^1\times D^2)=S^1\times S^1\to T$ that sends the homotopy class of $\{1\}\times S^1$ to the slope $\mu$ is called {\em Dehn filling of $M$ along the slope $\mu$} and yields an compact oriented $3$-manifold that contains $M$ as a submanifold. We call the slope along which we perform a Dehn filling the {\em filling slope}.

Let $\eta_3>0$ be a Margulis constant which we fix once and for all. 

Let $\mb{M}$ be a complete finite volume hyperbolic 3-manifold with cusps. Each cusp $\gamma$ of $\mb{M}$ has a standard $\eta_3$-Margulis neighborhood $\mb{T}_{\eta_3}(\gamma)\subset\mb{M}$ isometric to a quotient $\mc{O}/\mb{Z}^2$ where $\mc{O}\subset\mb{H}^3$ is a horoball and $\mb{Z}^2<{\rm Isom}^+(\mb{H}^3)$ consists of parabolic isometries stabilizing $\mc{O}$. The boundary $T_\gamma:=\partial\mb{T}_{\eta_3}(\gamma)\subset\mb{M}$ is a {\em standard torus horosection} of $\gamma$. The intrinsic metric is of $T_\gamma$ is flat.

\begin{dfn}[Normalized Length]
Let $\mu\subset T_\gamma$ be a simple closed geodesic. The {\em normalized length} of $\mu$, is the quantity defined by 
\[
{\rm nl}(\mu):=\ell_{T_\gamma}(\mu)/\sqrt{{\rm Area}(T_\gamma)}.
\]
\end{dfn}

By Dehn filling $\mb{M}$ along a slope $\mu$ of $T_\gamma$, we mean Dehn filling the underlying manifold of $\mb{M}\setminus{\rm int}(\mb{T}_{\eta_3}(\gamma))$ along a slope $\mu$ of the boundary component $T_\gamma$ of~$M$.
Hodgson and Kerckhoff proved the following:

\begin{thm}[Hodgson-Kerckhoff \cite{HK:universal}]
\label{hk surgery}
Let $\eta_3>0$ be a Margulis constant. Let $\mb{M}$ be a complete finite volume hyperbolic 3-manifold with cusps. For each cusp $\gamma$ of $\mb{M}$, let $\mu$ be a simple closed geodesic on the standard torus horosection $T_\gamma=\partial\mb{T}_{\eta_3}(\gamma)$ of $\gamma$. Suppose that the normalized length of each $\mu$ is at least $10.6273$. Then there exists a hyperbolic metric on the result of Dehn filling $M$ along the slopes defined by $\mu$ such that core corves of the added solid tori are geodesics.
\end{thm}

If we have a geodesic link $\Gamma\subset M$ in a complete finite volume hyperbolic 3-manifold $M$ and the complement $M-\Gamma$ admits a complete finite volume hyperbolic {{metric}} as well, then, by the techniques from~\cite{HK:universal,HK:shapeDSspace}, there is a relation between the length of the link $\Gamma\subset M$ and the normalized length of the standard meridians of $\Gamma$ in $M-\Gamma$ made fully explicit in~\cite{FPS19}.
% and add reference to Futer, Purcell, Schleimer. 

\begin{thm}[{\cite[Corollary~6.13]{FPS19}}]
\label{deformation}
Let $\Gamma\subset M$ be a geodesic link in a complete finite volume hyperbolic 3-manifold. Suppose that one of the following holds
\begin{enumerate}
\item{$M-\Gamma$ has a complete finite volume hyperbolic metric for which the total normalized length $L$ of the meridians of $\Gamma$ is at least $L\ge 7.823$.}
\item{In %the complete {{metric}} on
$M$ each component of $\Gamma$ has length at most 0.0996 and the total length of $\Gamma$ is at most $\ell\le 0.1396$.}
\end{enumerate}
Then
\[
\frac{2\pi}{L^2+16.17}<\ell<\frac{2\pi}{L^2-28.78}
\]
\end{thm}

Brock and Bromberg studied the change in the geometry of a family $M_t$ of hyperbolic cone manifold structures on a topological model $M$ in a very general setup and proved that away from the singular locus $\Gamma$ one can get uniform bilipschitz control only depending on the length of $\Gamma$. This is the content of the following Drilling Theorem.

\begin{thm}[{\cite[Theorem 6.2]{BB04}}]
\label{drilling thm}
For every $K\in(1,2)$ there exists $\eta_{{\rm drill}}\in(0,\eta_3)$ such that the following holds. Let $M$ be a complete finite volume hyperbolic 3-manifold. Let $\Gamma\subset M$ be a geodesic link and let $M'$ be a complete hyperbolic structure on $M-\Gamma$.
Suppose that $\ell_M(\Gamma)<\eta_{{\rm drill}}$. Then, there exists a $K$-bilipschitz diffeomorphism of pairs
%isotopic to the inclusion $M'\subset M$ 
\[
\left(M'-\bigsqcup_{\alpha\subset\Gamma}{\mb{T}_{\eta_3}(\alpha)},\bigsqcup_{\alpha\subset\Gamma}{T_\alpha}\right)\to\left(M-\bigsqcup_{\alpha\subset\Gamma}{\mb{T}_{\eta_3}(\alpha)},\bigsqcup_{\alpha\subset\Gamma}{T_\alpha}\right)
\]
where $\mb{T}_{\eta_3}(\alpha)$ denotes a standard $\eta_3$-Margulis neighborhood for $\alpha$ and $T_\alpha=\partial\mb{T}_{\eta_3}(\alpha)$ its boundary.
%\todo{G: I replaced isotopic to the identity with diffeomorphism of pairs. It is probably enough for the way we use it. P: I think it is better clearer that way. If you prefer, we could just after wards put a sentence concerning isotopy. But since we dont need it, it might not be necessary} 
\end{thm}

\part{Short curves on Heegaard splittings}\label{part:shortcurves}
\section{Outline and the topology of Heegard splittings}
\label{short curves}

In this and the next section we %move to the second construction and
discuss short curves on Heegaard splittings. The goal is to prove Theorem~\ref{thm:main}. The family of examples that arise from Theorem~\ref{thm:main} is shown to be generic from the point of view of the random walk in Section~\ref{proof1}. 

\subsection{Outline}
We now outline the strategy to produce a hyperbolic metric on $M_f$ for which $\gamma\subset\Sigma$ is a short geodesic.

We start by associating to $\gamma$ the 3-manifold   
\[
M_f-\gamma=(H_g-\gamma)\cup_f(H_g-f(\gamma)).
\]
The Heegaard splitting $M_f$ is obtained from $M_f-\gamma$ by Dehn filling along a filling slope which is completely determined by the topology. 

According to Thurston's Hyperbolization Theorem, the manifold $M_f-\gamma$ admits a complete hyperbolic metric provided that it is {irreducible}, {atoroidal} and {Haken}; see, for example, Theorem 1.42 in~\cite{Kapovich:hyperbolization}. We take~\cite{Kapovich:hyperbolization} as our general reference for basic 3-manifold topology. Notice, however, that, since it is the interior of a compact manifold with non-empty boundary, by basic 3-manifold topology, if $M_f-\gamma$ is irreducible, then it will also be automatically Haken (see Corollary 1.24 in~\cite{Kapovich:hyperbolization}). 

Both irreducibility and atoroidality will follow from our assumption that both $(H_g,\gamma)$ and $(H_g,f(\gamma))$ are so-called {pared acylindrical handlebodies}. This is condition (a) of Theorem~\ref{thm:main} and is borrowed from the theory of Jaco, Shalen and Johannson (see Chapter 5 of~\cite{CMcC}).

Once we have a hyperbolic {{metric}} on $M_f-\gamma$ we use it to find a hyperbolic metric on $M_f$. %via a cone manifold deformation.
The tool for such an operation is Hodgson and Kerckhoff effective version of the Hyperbolic Dehn Surgery Theorem (see Theorem~\ref{hk surgery}). We need to certify that the canonical filling slope has large normalized length on a standard torus horosection $T:=\partial\mb{T}_{\eta_3}(\gamma)$ of the cusp of $M_f-\gamma$. 

Actually, we will check something stronger, namely that, if condition (b) of Theorem~\ref{thm:main} holds, then the flat torus $T$ itself will be long and skinny so that {\em every} filling slope different from the one corresponding to $(\Sigma-\gamma)\cap T$ will have large normalized length.

The central point of the proof is to show that $T$ is long and skinny provided that the subsurface projection $d_{\Sigma-\gamma}(\mc{D},f\mc{D})$ is large. The main ingredient for the argument is the model manifold technology by Minsky~\cite{M10} and Brock, Canary and Minsky~\cite{BrockCanaryMinsky:ELC2}. 

At this point, before going on, we need a further consequence of the condition (a): If $(H_g,\gamma)$ is pared acylindrical, then the inclusion $\Sigma-\gamma\subset H_g$ is {\em doubly incompressible}. Double incompressibility allows to use Thurston's Uniform Injectivity for Pleated Surfaces \cite{Thu86} and its consequences. In particular, we show that there are two simple closed curves $\alpha\subset\Sigma-\gamma$ and $\beta\subset\Sigma-f(\gamma)$ which are represented by geodesics in $M_f-\gamma$ with moderate length and which are combinatorially close to the disk set projections as follows:
$d_{\Sigma-\gamma}(\mc{D},\alpha)\le 2$ and $d_{\Sigma-\gamma}(f\mc{D},\beta)\le 2$. 
Notice that, since $d_{\Sigma-\gamma}(\mc{D},f\mc{D})$ is large, also $d_{\Sigma-\gamma}(\alpha,\beta)$ will be large. 
In order to produce the moderate length curves $\alpha$ and $\beta$ it turns out to be convenient to work with the %geometrically finite
coverings of $M_f-\gamma$ corresponding to $\pi_1(H_g-\gamma)$ and $\pi_1(H_g-f(\gamma))$ instead of $M_f-\gamma$ itself.

As a last step, incompressibility of $\Sigma-\gamma\subset M_f-\gamma$ and the fact that the curves $\alpha,\beta\subset\Sigma-\gamma$ are represented by moderate length curves in $M_f-\gamma$ and are combinatorially far apart in the curve graph $\mc{C}(\Sigma-\gamma)$ implies, via the model manifold technology, that $T=\partial\mb{T}_{\eta_3}(\gamma)$ is long and skinny.  

This concludes the outline. The remainder of this section is dedicated to the topological part of the proof. Establishing the topological properties that imply that $M_f-\gamma$ is hyperbolic and checking that $\Sigma-\gamma$ is doubly incompressible. The geometric part of the argument will be discussed in the next section.

\subsection{Hyperbolizable drilled Heegaard splittings}
Gluing together two pared acylindrical handlebodies produces an irreducible and atoroidal 3-manifold.
% We give a proof of this fact

\begin{pro}
\label{hyperbolizable}
If $(H_g,\gamma)$ and $(H_g,f(\gamma))$ are both pared acylindrical then $M_f-\gamma$ is irreducible, atoroidal and Haken.
\end{pro}

In particular, by Thurston's Hyperbolization Theorem for Haken manifolds, the drilled splitting $M_f-\gamma$ admits a hyperbolic structure, which is also unique by Mostow-Prasaad rigidity.

We split the proof of Proposition~\ref{hyperbolizable} into small steps, Lemmas~\ref{incompressible},~\ref{irreducible} and~\ref{atoroidal}. First we observe that $\Sigma-\gamma\subset M_f-\gamma$ is incompressible.

\begin{lem}
\label{incompressible}
The surface $\Sigma-\gamma\subset M_f-\gamma$ is $\pi_1$-injective. The fundamental group $\pi_1(M_f-\gamma)$ decomposes as $\pi_1(H_g-\gamma)*_{\pi_1(\Sigma-\gamma)}\pi_1(H_g-f(\gamma))$.
\end{lem}

\begin{proof}
This is a consequence of Dehn's Lemma and Seifert--van Kampen's Theorem.
\end{proof}

We now use transversality arguments together with acylindricity of the handlebodies to establish irreducibility and atoroidality of $M_f-\gamma$.

\begin{lem}
\label{irreducible}
$M_f-\gamma$ is irreducible.
\end{lem}

\begin{proof}
Assume towards a contradiction that $M_f-\gamma$ is reducible, and let $S$ be an embedded 2-sphere that does not bound a ball. We may and do assume that $S$ intersects $\Sigma-\gamma$ transversally and such that any component of $(\Sigma-\gamma)\cap S$ that is innermost in $S$ is not homotopically trivial in $\Sigma-\gamma$. Indeed, otherwise this innermost component bounds a disc in $\Sigma-\gamma$ and a disk in $S$ such that their union is contained in one of the two handlebodies. However, this sphere bounds a ball $B$ in that handlebody by irreducibility of handlebodies, and thus $B$ can be used to guide an isotopy of $S$ that removes that component of $(\Sigma-\gamma)\cap S$. Note that $\Sigma\cap S$ is nonempty since otherwise $S$ would be contained in one handlebody and bound a ball, again by irreducibility of handlebodies. Hence, any component $\alpha$ of $\Sigma\cap S$ that is innermost in $S$ compresses in one of the two handlebodies via the disc $D$ in $S$ with $D\cap \Sigma=\alpha$. This contradicts $\pi_1$-injectivity of $\Sigma-\gamma$ into both handlebodies. 
\end{proof}

\begin{lem}
\label{atoroidal}
$M_f-\gamma$ is atoroidal. 
\end{lem}

\begin{proof}
We show that $M_f-\gamma$ is topologically atoroidal (i.e.~every incompressible torus is boundary parallel), which implies atoroidal for Haken manifolds (which $M_f-\gamma$ is); see Section~1.2 in~\cite{Kapovich:hyperbolization}.
 Let $T$ be an incompressible torus in $M_f-\gamma$. We show that $T$ is boundary parallel.
 Arrange that $T$ is in general position with respect to $\Sigma-\gamma$, and take $\{\alpha_i\}_{i=1,\dots,n}$ to be the simple closed curves that are the components of the intersection $(\Sigma-\gamma)\cap T$.
By an innermost argument, if some $\alpha_i$ was null-homotopic in $T$, then some $\alpha_j$ would bound a disk in one of the handlebodies $H$ bounded by $\Sigma$. However, this implies that $\alpha_j$ is trivial in $\Sigma-\gamma$ since $\Sigma-\gamma$ maps $\pi_1$-injectively to both handlebodies. Then, the union of the disks in $S_K$ and $T$ with boundary $\alpha_j$ bounds a ball in $H$ by irreducibility of $H$, and we can use this ball to reduce the number of components of $(\Sigma-\gamma)\cap T$. In view of this argument, we can assume that each $\alpha_i$ is an essential curve on $T$. Moreover, notice that each $\alpha_i$ is then automatically an essential curve on $\Sigma-\gamma$, whence on $\Sigma$, too, for otherwise some $\alpha_j$ would bound a compressing disk.
We reindex the $\alpha_i$ to make sure that consecutive ones (modulo $n$) bound an annulus in $T$.

Consider now some $\alpha_i$, and let $H$ be the handlebody containing the annulus $A\subseteq T$ bounded by $\alpha_i$ and $\alpha_{i+1}$ (notice that if there is one $\alpha_i$, then in fact there are at least two because $\Sigma-\gamma$ separates $M_f-\gamma$, so that $(\Sigma-\gamma)\cap T$ needs to have at least two connected components).

{\bf Claim:} $A$ is boundary parallel.

{\em Proof of the Claim.}
Using that $(H,\gamma)$ is pared acylindrical, we note %also
that $\alpha_i$ and $\alpha_{i+1}$ are two boundaries of an embedded annulus $A'$ in $\Sigma$ that is homotopic rel boundary to $A$ in $H$. Since $\partial H\setminus (\partial A)$ is disconnected, also $H\setminus A$ is disconnected. %, which follows from $(H,\gamma)$ being pared acylindrical.
Let $N$ be the 3-manifold given as the closure of the component with boundary $A\cup A'$. To establish that $A$ is boundary parallel, we show that $N$ is a solid torus.

Let $I\subset A$ be a properly embedded interval connecting the two boundary components of $A$. And let $I'\subset A'$ be an arc that is homotopic rel boundary to $I$ in $H$. An innermost circle argument gives that $I'\cup I$ is also null-homotopic in $N$, hence by Dehn's Lemma there exists a disc $D\subset N$ with boundary $I'\cup I$.
Let $S$ be the sphere obtained as the union of the disc given by boundary compression of $A$ along $D$ and the disc in $\Sigma$ with the same boundary.
This sphere $S$ bounds a ball in $H$, and thus $N$, by irreducibility of handlebodies. Thus $N$ compresses to a ball; hence, $N$ is a solid torus as desired. \qed

Let $A'$ be the annulus in $\Sigma$ that union $A$ forms a torus that bounds a solid torus in $H$.
If $A'$ does not contain $\gamma$, then we can reduce the number of components of $(\Sigma-\gamma)\cap T$ by isotoping $A$ in $M_f-\gamma$ to the other handlebody. Otherwise, $A'$ is (up to isotopy in $\Sigma$) $N(\gamma)$, and we can apply an isotopy in $M_f$ to move $A$ inside a regular neighborhood $N$ of $\gamma$ in $M_f$. Notice that there is at least one $\gamma_i$ for otherwise $T$ would be contained in one of the handlebodies; however, handlebodies do not contain incompressible tori since incompressible tori are $\pi_1$-injective and the fundamental groups of handlebodies do not contain $\mathbb Z^2$ subgroups. In particular, $T$ is a union of annuli as above, and hence, after applying finitely many isotopies, we reduce to the case that every annulus as $T$ is entirely contained in $N$ (we can assume that the isotopies we found above move $N$ inside itself). Hence, $T$ can be thought of as an incompressible surface in $N-\gamma$. There is a classification of incompressible surfaces in $S^1$--bundles; see \cite[Satz 2.8]{Waldhausen_incomp_surf_bundles}, which in our case implies that $T$ is boundary parallel, as required.
\end{proof}

Hence $M_f-\gamma$ is irreducible and atoroidal. As mentioned before, irreducibility is already enough to ensure that it is also Haken. Incidentally, we observe that $\Sigma-\gamma\subset M_f-\gamma$ is a Haken surface. 

Combining with Lemma~\ref{lem:d>=3impliesparedacyl}, we get the following.

\begin{cor}
Let $\gamma\subset\Sigma$ be a non-separating simple closed curve. If $d_{\mc{C}}(\gamma,\mc{D}\cup f\mc{D})\ge 3$, then $M_f-\gamma$ is hyperbolizable.
\end{cor}

\subsection{Double incompressibility}
The second crucial topological property of a pared acylindrical handlebody $(H_g,\gamma)$ that we need is the fact that the inclusion of the boundary $\Sigma-\gamma\subset H_g$ is {doubly incompressible}. The following definition is due to Thurston \cite{Thu86}.

\begin{dfn}[Doubly Incompressible]
Let $\gamma\subset\Sigma$ be an essential simple closed curve with a tubular neighborhood $N(\gamma)$ in $\Sigma$ and $U(\gamma)$ in $H_g$. The inclusion $\Sigma-\gamma\subset H_g$ is {\em doubly incompressible} if it satisfies
\begin{enumerate}[(a)]
\item{The inclusion $\Sigma-\gamma\subset H_g$ is $\pi_1$-injective.}
\item{Essential relative homotopy classes of maps $(I,\partial I)\to (\Sigma-N(\gamma),\partial N(\gamma))$ are mapped injectively to relative homotopy classes of maps $(I,\partial I)\to (H_g-U(\gamma),\partial U(\gamma))$.}
\item{There are no essential cylinders in $\Sigma-N(\gamma)$: This means that every essential map $f:(A,\partial A)\to (H_g,\Sigma-N(\gamma))$ is either homotopic into $U(\gamma)$ or the restriction of $f$ to $\partial A$ extends to a map into $\Sigma-N(\gamma)$.}
\item{Each maximal abelian subgroup of $\pi_1(\Sigma-\gamma)$ is mapped to a maximal abelian subgroup of $\pi_1(H_g)$.}
\end{enumerate}
Since maximal abelian subgroups of a free group are infinite cyclic group, the last condition is equivalent to 
\begin{enumerate}[(a)]
\setcounter{enumi}{4}
\item{each maximal cyclic subgroup of $\pi_1(\Sigma-\gamma)$ is mapped to a maximal cyclic subgroup of $\pi_1(H_g)$.}
\end{enumerate}
\end{dfn}

As Thurston observes (see Section 7 of \cite{Thu86}) we have the following

\begin{pro}
\label{doubleincomp}
If $(H_g,\gamma)$ is pared acylindrical, then the inclusion $\Sigma-\gamma\subset H_g$ is doubly incompressible.
\end{pro}

The Proposition~\ref{doubleincomp} is probably well-known to experts and follows from JSJ theory. However, it might not be easy to extract from the literature. For this reason, and for the sake of being more self-contained, we include a proof in Appendix~\ref{appendix b}.

\section{Long skinny cusp horosection}
\label{large horosection}

In this section we show that a standard torus horosection $T$ of the cusp of $M_f-\gamma$ is {long} and {skinny}, so that {every} filling slope different from the one coming from $(\Sigma-\gamma)\cap T$ will have large normalized length, provided that $d_W(\mc{D},f\mc{D})$ is sufficiently large. This is the geometric part of the proof of Theorem~\ref{thm:main} and rests on the model manifold technology of Minsky~\cite{M10} and Brock, Canary and Minsky~\cite{BrockCanaryMinsky:ELC2}. 

Here, the standard torus horosection $T$ is $\partial\mb{T}_{\eta_3}(\gamma)$, where $\eta_3>0$ is a fixed Margulis constant and $\mb{T}_{\eta_3}\subset M_f-\gamma$ is the cusp of $M_f-\gamma$ that forms a connected component of the $\eta_3$-thin part of $M_f-\gamma$.

The proof is divided into two steps. The first one consists of finding simple closed curves $\alpha,\beta\subset W:=\Sigma-\gamma$ that are represented by closed geodesics in $M_f-\gamma$ with moderate length and such that $d_W(\alpha,\mc{D})$, $d_W(\beta,f\mc{D})\le 2$. This is the content of Proposition~\ref{shortsurgeries} and Corollary~\ref{shortcurves}. As a second step, once we have such curves $\alpha$ and $\beta$, we argue that $d_W(\alpha,\beta)$ gives a coarse lower bound for the length of any slope on $T$ that does not come from the Heegaard surface. We prove this in Proposition~\ref{width}.

\subsection{Handlebody covering}
In order to find the curves $\alpha$ and $\beta$, we work with {\em handlebody coverings}, which we now describe. 

Consider first the pared handlebody $(H_g,\gamma)$. The fundamental group of $H_g-\gamma\subset M_f-\gamma$ injects into $\pi_1(M_f-\gamma)$ (see Lemma~\ref{incompressible}) and determines a covering $N$ of $M_f-\gamma$ to which $H_g-\gamma$ lifts homeomorphically. By slight abuse of notation we will not distinguish between $H_g-\gamma$ and its homeomorphic lift to $N$. Recall from the outline, that we want to find a simple closed curve $\alpha\subset\Sigma-\gamma$ that has moderate length representative in $N$ and satisfies $d_W(\alpha,\mc{D})\le 2$.

As both $H_g-\gamma$ and $N$ are aspherical, the inclusion $H_g-\gamma\subset N$ is a homotopy equivalence. Since the pair $(N,H_g-\gamma)$ has also the homotopy extension property (e.g.~since it can be give the structure of a CW-pair), it follows that the manifold $N$ deformation retracts to $H_g-\gamma$ (this is a general fact; see, for example, \cite[Corollary~0.20]{Hatcher_AT}). Therefore, according to Proposition~\ref{doubleincomp}, the inclusion $W=\Sigma-\gamma\subset N$ is doubly incompressible.

%and observe that $(N,H_g-\gamma)$ has the homotopy extension property (since, e.g., it can be given the structure of a CW-pair (which have the homotopy extension property by \cite[Proposition 0.16]{Hatcher_AT}); such CW structure exists since that was also the case for $(M_f-\gamma,H_g-\gamma)$ and CW structures can be lifted) deformation retracts to $H_g-U(\gamma)$) 
%\todo{P: Added the referece to bibtech and added even moredetailed explanation on why this has the homotopy extension property, but commented it out, since it was too much.}

Even though we will not need it, for the sake of clarity, we provide a more complete description of the covering. The hyperbolic structure on $N$ is isometric to a geometrically finite structure on $\textrm{int}(H_g)$ with a rank one cusp at $\gamma$ and the submanifold $H_g-\gamma\subset N$ is a so-called relative Scott core for $N$~\cite{Sc73,McC86,KS89}. The fact that $N$ is homeomorphic to the interior of $H_g$ follows from Bonahon's Tameness Theorem~\cite{Bo86}. Geometric finiteness is, instead, a consequence of Canary and Thurston's Covering Theorem~\cite{C96}.
 
We now come back to double incompressibility. Crucially, it allows to use Thurston's Uniform Injectivity for Pleated Surfaces~\cite{Thu86}.

\begin{thm}[Uniform Injectivity, {\cite[Theorem 5.7]{Thu86}}]
\label{uniforminj}
Fix $\eta>0$, a Margulis constant. For every $\ep>0$ there exists $\delta>0$ such that for any type preserving doubly incompressible pleated surface $g:W\to N$ with pleating locus $\lambda$ and induced metric $\sigma$, if $x,y\in\lambda$ lie in the $\eta$-thick part of $(W,\sigma)$, then
\[
d_{{\bf P}(N)}({\bf p}_g(x),{\bf p}_g(y))\le\delta\Longrightarrow d_\sigma(x,y)\le\ep.
\]
Here ${\bf p}_g:\lambda\to {\bf P}(N)$ denotes the map induced by $g$ from the lamination $\lambda$ to the projective unit tangent bundle of $N$.
\end{thm} 

We use Theorem~\ref{uniforminj} to prove the following

\begin{pro}
\label{shortsurgeries}
There exists $L>0$ such that the following holds. For every $\delta\in\mc{D}$ there exists a pleated surface $g:(W,\sigma)\to N$ in the proper homotopy class of the inclusion $W=\Sigma-\gamma\subset N$ that realizes $\lambda:=\pi_W(\delta)$ as a sublamination of its pleating locus and such that one of the arcs $\delta_0$ of $\lambda\cap W_0$ satisfies $\ell_\sigma(\delta_0)\le L$. Here $W_0$ is the $\eta_0$-non cuspidal part of $(W,\sigma)$ where $\eta_0>0$ is a universal constant.
\end{pro}

Finally, the moderate length segments provided by Proposition~\ref{shortsurgeries} can be promoted to moderate length curves $\alpha$ and $\beta$ in a simple way

\begin{cor}
\label{shortcurves}
There exists $L>0$, only depending on $\Sigma$ such that there are simple closed curves $\alpha,\beta\subset W$ with $d_W(\alpha,\pi_W(\mc{D})),d_W(\beta,\pi_W(f\mc{D}))\le 2$ and satisfying $\ell_{M_f-\gamma}(\alpha^*),\ell_{M_f-\gamma}(\beta^*)\le L$, where $\alpha^*,\beta^*$ are the geodesic representatives of $\alpha,\beta$ in $M_f-\gamma$.
\end{cor}

\begin{proof}
We only provide the argument for $\alpha$ since the one for $\beta$ is completely analogous. Choose $\delta\in\mc{D}$ arbitrarily. Let $g:(W,\sigma)\to N$ and $\delta_0$ be the pleated surfaces and the moderate length arc in $\lambda\cap W_0$ provided by Proposition~\ref{shortsurgeries}. The length of $\partial W_0$ is bounded by $2\eta_0$. One of the boundary components of a regular neighborhood of $\delta_0\cup\partial W_0$ is an essential simple closed curve $\alpha$ of length at most $2L+2\eta_0$  with the property $d_W(\alpha,\pi_W(\delta))\le 2$. 
\end{proof}

The proof of Proposition~\ref{shortsurgeries} is where we fully exploit the assumption that $W$ is the complement of a non-separating simple closed curve. % to make the arguments elementary.
The main ingredients of the proof are Thurston's Uniform Injectivity and a technical lemma about quasi geodesic concatenations in $\mb{H}^3$. We also use the following general property of pleated surfaces observed by Thurston in~\cite{Thu86}.

\begin{lem}[{\cite[Lemma~3.1]{M00}}]
\label{thick-to-thick}
For every Margulis constant $\eta_0>0$ there exists $\eta_3$, only depending on $\eta_0$ and the topological type of $W$, such that, if $g:(W,\sigma)\to N$ is a $\pi_1$-injective pleated surface, then only the $\eta_0$-thin part of $(W,\sigma)$ can enter the $\eta_3$-thin part of $N$.
\end{lem}

\subsection{Quasi geodesic concatenations}
Beside the use of Uniform Injectivity, our proof of Proposition~\ref{shortsurgeries} is elementary and rests on the following fact about piecewise broken geodesics in $\mb{H}^3$, which we state without proof.

\begin{lem}
\label{broken geo}
There exists $L>0$ and $A=A(L)>0$ such that the following holds. Let $\gamma:I\to\mb{H}^3$ be a broken piecewise geodesic $\gamma=\gamma_1*\dots*\gamma_m$ with breaking angles $\angle\gamma_{i-1}\gamma_i$ contained in $(0,\pi/2)$ and geodesic segments of length at least $L$. Then $\gamma$ is a $A$-quasi geodesic. 
\end{lem}

In particular, if the length of the geodesic segments is large enough, only depending on $A$, then $\gamma$ cannot be a loop. Lemma~\ref{broken geo} follows from the fact that, in a hyperbolic space, local quasi geodesics, such as $\gamma$ (as is not difficult to check), are global quasi geodesics.

Using Lemma~\ref{broken geo}, we find the following. 

\begin{lem}
\label{concatenation}
For each $\ep>0$ there exists $L_0>0$ such that the following holds. Let $X=\mb{H}-\sqcup_{1\le j\le n}{\mc{O}_j}$ be the complement in $\mathbb H^3$ of a family of pairwise disjoint open horoballs. Let $\gamma=\gamma_1*\dots*\gamma_{2n}$ be a concatenation of paths in $X$ such that
\begin{itemize}
 \item{$\gamma_{2j+1}$ is a geodesic of length at least $\ep$ on the horosphere $\mc{H}_j=\partial\mc{O}_j$.}
 \item{$\gamma_{2j}$ is the orthogonal segment connecting $\mc{H}_{j-1}$ and $\mc{H}_j$. We require that $\gamma_{2j}$ has length at least $L_0$.}
\end{itemize}
Then $\gamma$ is not a loop.
\end{lem}

\begin{proof}
We use the following fact that can be easily checked in the upper half space model of $\mb{H}^3$. If $x,y$ lie on the same horosphere $\mc{H}$, then we have 
\[
\sinh(d_{\mb{H}^3}(x,y))=d_{\mc{H}}(x,y).
\]
Denote by $\mc{H}_j=\partial\mc{O}_j$ the boundary horosphere of the horoball $\mc{O}_j$. Observe that, by assumption, the flat geodesic $\gamma_{2j}\subset\mc{H}_j$ has length $\ell(\gamma_{2j})\ge\ep$. 

If we expand $\mc{H}_j$ radially from its center at infinity, the intrinsic geometry of the horosphere expands exponentially with exponent equal to the increase in the radius. Hence, if we inflate all the horoballs $\mc{O}_j$ by $r$, then each $\gamma_{2j+1}$ is shortened to an arc $\gamma_{2j+1}^r$ of length 
\[
\ell(\gamma_{2j+1}^r)=\ell(\gamma_{2j+1})-2r\ge L-2r,
\]
while the length of the inflated $\gamma_{2j}$, denoted by $\gamma_{2j}^r$, becomes
\[
\ell(\gamma_{2j}^r)=e^r\ell(\gamma_{2j}).
\]
We now straighten all the $\gamma_{2j}^r$ relative to the endpoints and obtain geodesic arcs $\alpha_{2j}^r$ of length $\ell(\alpha_{2j}^r)=\sinh^{-1}(e^r\ell(\gamma_{2j}))\ge\sinh^{-1}(e^r\ep)$. 

Let us now consider the angles between the segments $\gamma_{2j-1}$ and $\alpha_{2j}$. Observe that, by assumption, $\gamma_{2j-1}^r$ is orthogonal to $\mc{H}_j^r$, the inflated horosphere. Therefore, as, by convexity of horoballs, $\alpha_{2j}$ is contained in $\mc{O}_j^r$, the angle $\angle\gamma_{2j-1}^r\alpha_{2j}^r$ between the two geodesics $\gamma_{2j-1}$ and $\alpha_{2j}$ is in $(0,\pi/2)$. 

In conclusion, if the length of each $\gamma_{2j+1}$ is much larger than $L+2r$ and $\sinh^{-1}(e^r\ep)\ge L$, where $L$ is the constant of Lemma~\ref{broken geo}, then the broken piecewise geodesic $\gamma=\gamma_1^r*\alpha_2^r*\dots*\gamma_{2n-1}^r*\alpha_{2n}^r$ is a $A$-quasi geodesic. If it is also sufficiently long compared to $A$, which can be again achieved by assuming that each $\gamma_{2j+1}$ is long enough, it cannot be a loop.  
\end{proof}

\subsection{Moderate length surgeries}
We now prove Proposition~\ref{shortsurgeries}.

\begin{proof}[Proof of Proposition~\ref{shortsurgeries}]
Pick $\delta\in\mc{D}$ arbitrarily. Since $(H_g,\gamma)$ is pared acylindrical, $\delta$ intersects $\gamma$ essentially. Consider $\lambda=\delta\cap W$, it is a multi-arc in $W=\Sigma-\gamma$.

Now, after collapsing parallel components to a single one, $\lambda$ can be realized as a sub-lamination of the pleating locus of some pleated surface $g:(\Sigma-\gamma,\sigma)\to N$ in the proper (relative to cusps) homotopy class of the inclusion $\Sigma-\gamma\subset N$ (see for example Theorems I.5.3.6 and I.5.3.9 in~\cite{CEG:notes_on_notes}). 
 
When regarding $\delta$ as a simple closed curve on $\Sigma$, we can think of it as a concatenation $\delta=\beta_1*\dots*\beta_{2n}$ of arcs $\beta_{2j}$ in $W$ and arcs $\beta_{2j+1}$ crossing a regular annular neighborhood of $\gamma$. Here we are using the fact that $W$ is the complement of a non-separating simple closed curve.

Using the proper homotopy between the inclusion $W=\Sigma-\gamma\subset N$ and $g$, we simultaneously straighten all the $\beta_{2j}$'s to subarcs $\alpha_{2j}$ of the geodesic leaves of $g(\lambda)$ that start and end in the standard horosection $T=\partial\mb{T}_{\eta_3}(\gamma)$. Then, again, using the proper homotopy between $\Sigma-\gamma\subset N$ and the deformation retraction of $N$ to $N-\mb{T}_{\eta_3}(\gamma)$, we replace the arcs $\beta_{2j+1}$ with arcs $\alpha_{2j+1}$ on the horosection $T$ that are geodesic with respect to the intrinsic flat metric and join the endpoints of the previously obtained $\alpha_{2j}$ and $\alpha_{2(j+1)}$.

We have that $\delta\subset\Sigma$ is homotopic in $N$ to a concatenation $\alpha_1*\dots*\alpha_{2n}$ of closed arcs $\alpha_j$, where each $\alpha_{2j+1}$ is a geodesic in the boundary of the standard horosection $T$, while each $\alpha_{2j}$ is a proper geodesic arc in $N-\mb{T}_{\eta_3}(\gamma)$ contained in the pleating locus $g(\lambda)$. 

Notice that, by Lemma~\ref{thick-to-thick}, if $\eta_3$ is sufficiently small, then, when we look at $\alpha_{2j}$ in the intrinsic metric of the pleated surface, it will join two points of the $\eta_0$-cuspidal part for some universal $\eta_0>0$.

Now, by Theorem~\ref{uniforminj}, which applies to $g$ because it is doubly incompressible, there is a uniform $\ep>0$ such that the length of each $\alpha_{2j+1}$, is at least $\ep$. Let $L_0$ be as in Lemma~\ref{concatenation}, for the given $\ep$. Notice that this constant only depends on $\Sigma$. Being homotopically trivial in $N$, the loop $\alpha_1*\dots*\alpha_{2n}$ lifts to a closed loop in $\mb{H}^3$. Hence, by Lemma~\ref{concatenation}, some $\alpha_{2i}$ must have length less than $L_0$.
\end{proof}

\subsection{Size of the standard horosection}
Consider $T=\partial\mb{T}_{\eta_3}$, the boundary of the standard horosection. We show that $T$ is long and skinny, meaning that the length of any slope $\mu\subset T$ different from the one coming from $(\Sigma-\gamma)\cap T$ is very long, provided that $d_W(\alpha,\beta)$ is sufficiently large. 

\begin{pro}
\label{width}
There exists $c=c(\Sigma)>0$ such that
\[
d_W(\alpha,\beta)\le c\ell_T(\mu)+c
\] 
for every slope $\mu$ in $T$ that is not homotopic to a component of $(\Sigma-\gamma)\cap T$.
\end{pro}

\begin{proof}
 Notice that there is a bound $D'$, depending only on $\Sigma$, on the length of a component of $(\Sigma-\gamma)\cap T$. In fact, the boundary of the cusp of a complete hyperbolic surface of finite area can be bounded in terms of the topological type of the surface only, as the area of the cusp is an increasing function of the length of its boundary, while the total area of the surface only depends on the topological type (e.g. by the Gauss-Bonnet theorem for non-compact surfaces).

Observe now that the intrinsic diameter of $T$ satisfies 
\[
\text{\rm diam}(T)\le D:=(\ell_T(\mu)+D')/2;
\]
hence, any two points on $T$ can be joined by a flat geodesic of length at most $D$.

Consider the covering $p:Q\to M_f-\gamma$ corresponding to $\pi_1(W=\Sigma-\gamma)<\pi_1(M_f-\gamma)$. By Bonahon's Tameness~\cite{Bo86} $Q$ is geometrically and topologically tame. Moreover, since $\Sigma-\gamma$ is not a virtual fiber (for example, because it separates $M_f-\gamma$), the covering $Q$ is a geometrically finite hyperbolic {{metric}} on $W\times\mb{R}$ (without accidental parabolics) by Thurston-Canary Covering Theorem~\cite{C96}. Denote by $X\sqcup Y=\partial\mc{CC}(Q)$ the boundary of the convex core.

By work of Minsky~\cite{M10} (see also Theorem 2.1.3 of Bowditch~\cite{Bowditch:ELT}), there exists a uniform quasi-geodesic $l:=\alpha_0,\alpha_1,\dots,\alpha_n$ in $\mc{C}(W)$, that is
\[ 
\frac{1}{C}|i-j|-C\le d_W(\alpha_i,\alpha_j)\le C|i-j|+C
\]
for some uniform constant $C>0$, with the following properties
\begin{enumerate}[(i)]
\item{Every $\alpha_j$ has a geodesic representative $\alpha_j^*$ in $Q$ of moderate length, that is $\ell(\alpha_j^*)\le L$ for some uniform $L>0$.}
\item{The initial and terminal curves have moderate length on the $X$ and $Y$ boundary components, that is $\ell_X(\alpha_0),\ell_Y(\alpha_n)\le L$, with $L$ as before.}
\item{Every curve $\beta\in\mc{C}(W)$ such that $\ell_Q(\beta)\le L$ lies uniformly close to the quasi geodesic $l$, that is $d_W(\beta,l)\le R$ for some uniform $R>0$.}
\end{enumerate} 

We notice that 
\[
n\ge c_0d_W(\alpha,\beta)-c_0
\]
for some uniform $c_0>0$. In fact, on one hand, by property (iii), we have $d_W(\alpha,l),d_W(\beta,l)\le R$. On the other hand, the uniform quasi geodesic property of $l:=\alpha_0,\dots,\alpha_n$ and hyperbolicity of the curve graph $\mc{C}(W)$ gives $d_W(\alpha_0,\alpha_n)\ge d_W(\alpha,\beta)-R_1$ for some uniform $R_1>0$. Combined together the two properties give the estimate above. 

The moderate length geodesics $\alpha_j^*$ are well spaced in $Q$. This is a consequence of the model manifold technology, which we use in the form of the following result of Bowditch~\cite{Bowditch:ELT} and Brock and Bromberg~\cite{BB11}.

\begin{thm}[see Theorem 2.1.4 of~\cite{Bowditch:ELT} and Theorem 7.16 of~\cite{BB11}]
\label{blocksep}
For every $L>0$ there exists $A>1$ such that the following holds. Let $Q$ a hyperbolic structure on $W\times\mb{R}$ for which the boundary $\partial W$ is parabolic. Let $Q_0=Q-\mb{T}_{\eta_3}(\partial W)$ be the complement of the standard cusp neighborhoods. Suppose that $\alpha,\beta\in\mc{C}(W)$ are simple closed curves represented in $Q$ by closed geodesics of length at most $L$. Then 
\[
\frac{1}{A}d_{\mc{C}(W)}(\alpha,\beta)-A\le\rho_{Q_0}\left(\mb{T}_{\eta_3}(\alpha),\mb{T}_{\eta_3}(\beta)\right)\le Ad_{\mc{C}(W)}(\alpha,\beta)+A
\]
Where $\rho_{Q_0}$ denotes the {$\eta_3$-electric distance} in $Q_0$.
\end{thm}

The {\em $\eta_3$-electric distance} $\rho_{Q_0}(x,y)$ between two points $x,y\in Q_0$ is defined to be the infimum of the {electric lengths} of all paths joining them in $Q_0$. The \emph{electric length} of a path $\delta$ in $Q_0$ is the length of the portion of $\delta$ that lies in the $\eta_3$-thick part of $Q_0$. Observe that, by definition, the electric distance satisfies $\rho_{Q_0}\le d_{Q_0}$.

The electric distance is also defined on hyperbolic surfaces $(W,\sigma)$, and it is a fact (bounded diameter lemma; see Lemma 1.10 of~\cite{Bo86}) that for any fixed Margulis constant $\eta>0$, if $(W,\sigma)$ has finite area, then the $\eta$-electric diameter of $(W,\sigma)$ is uniformly bounded only in terms of $\eta$ and $\chi(W)$. This fact applies in our setting: The $\eta_3$-electric diameter of any pleated surface is uniformly bounded. 

It follows from Theorem~\ref{blocksep} and the fact that the sequence $\alpha_0,\cdots,\alpha_n$ is a uniform quasi-geodesic that 
\[
\frac{1}{B}|i-j|-B\le\rho_{Q_0}(\alpha_i^*,\alpha_j^*)\le B|i-j|+B.
\]
for some uniform $B>0$. Also notice that $X$ and $Y$ are uniformly close to $\alpha_0^*$ and $\alpha_n^*$ respectively, meaning that 
\[
\rho_{Q_0}(\alpha_0^*,X),\rho_{Q_0}(\alpha_n^*,Y)\le\log(2L/\eta_3).
\]
This follows from the following standard fact, which we state without proof.

\begin{lem}
Let $\alpha\subset Q$ be a closed curve homotopic to a geodesic $\alpha^*$. Then 
\[
d_Q(\alpha,\mb{T}_{\eta_3}(\alpha^*))\le\log(2\ell(\alpha)/\eta_3).
\]
\end{lem}

%\begin{proof}
%Let $P$ be the covering of $Q$ corresponding to $\gamma\in\pi_1(Q)$. We can simultaneously lift $\gamma$ and $\gamma^*$ to $P$ since they are homotopic. The lift of $\gamma^*$ is the core geodesic of $P$. In Fermi coordinates $(r,l,\theta)\in(0,\infty)\times[0,\ell(\gamma^*))\times[0,2\pi)$ around $\gamma^*\subset P$ (where $l$ is the length along $\gamma^*$, $r$ is the distance from $\gamma^*$ and $\theta$ is the angular paramenter), the metric of $P$ can be written as
%\[
%ds^2=dr^2+\cosh(r)^2dl^2+\sinh(r)^2d\theta^2.
%\]
%Notice that the region $P_{\eta_3}:=\{r\le\cosh^{-1}(\eta_3/\ell(\gamma^*))\}\subset P$ projects to the Margulis tube $\mb{T}_{\eta_3}(\gamma^*)\subset Q$. Also notice that, if $\gamma$ is entirely contained in $\{r\ge R\}$, then $\ell(\gamma)\ge\cosh(R)\ell(\gamma^*)$. Therefore, $\gamma$ must intersect the solid torus $\{r\le\cosh^{-1}(\ell(\gamma)/\ell(\gamma^*))\}$. In particular, the distance between $\gamma$ and $P_{\eta_3}$, and, hence, also the distance between $\gamma$ and $\mb{T}_{\eta_3}$ in $Q$, is at most
%\[
%d_P(\gamma,P_{\eta_3})\le\cosh^{-1}(\ell(\gamma)/\ell(\gamma^*))-\cosh^{-1}(\eta_3/\ell(\gamma^*)).
%\]
%A little explicit computation with inverse hyperbolic cosines shows that the last term is bounded by $\log(2\ell(\gamma)/\eta_3)$.
%\end{proof}

For simplicity, from now on, we assume $n=2m$ since it does not affect the argument and simplifies the notation. Consider a pleated surface $G:(W,\sigma)\to Q$ realizing the middle curve $\alpha_{m=n/2}$ in $Q$. Under the covering projection $G$ descends to the pleated surface $g:=pG:(W,\sigma)\to M_f-\gamma$ realizing $\alpha_m$ in $M_f-\gamma$. 

The manifold $Q$ has two cusps that cover $\mb{T}_{\eta_3}(\gamma)$; we fix one of those and denote it by $\mb{T}^1_Q$, with boundary $T^1_Q=\partial\mb{T}^1_Q$. Then, we choose points $x\in T^1_Q\cap X$ and $w\in W$ such that $G(w)\in T^1_Q$. Observe that $\text{\rm inj}_x(X),\text{\rm inj}_w(W,\sigma)\ge\eta_3$. As a consequence, since the electric diameter of pleated surfaces is uniformly bounded, we have 
\[
\rho_{Q_0}(G(w),\alpha_m^*),\rho_{Q_0}(x,\alpha_0^*)\le K
\] 
for some uniform $K>0$. Now connect $p(x)$ to $g(w)$ via a shortest flat geodesic $\xi$ between them on $T$. Denote by $\delta$ the lift of $\xi$ to $Q$ with basepoint $G(w)$. We have $\ell(\delta)\le D=(\ell_T(\mu)+D')/2$.

In order to conclude the proof of Proposition~\ref{width} it remains to show that

{\bf Claim}: We have $\ell(\delta)\ge cn-c$ for some uniform constant $c>0$. 

We divide the proof of this claim into two cases.

{\bf Case I}. The endpoint $z$ of $\delta$ different from $G(w)$ coincides with $x$. 

In this case, we have
\[
\ell(\delta)\ge\rho_{Q_0}(x,G(w))\ge\rho_{Q_0}(\alpha_0^*,\alpha_m^*)-2K\ge\frac{m}{B}-B-2K,
\]
as desired.

{\bf Case II}. The endpoint $z$ of $\delta$ different from $G(w)$ differs from $x$.

In this case, let $\tau$ be a non peripheral loop on $X$ based at $x$ that has moderate length, say $\ell_X(\tau)\le L_1$ for some uniform $L_1>0$ only depending on $W$. 

We now observe that $p(\tau)$ does not lift to a loop based at $z$ in $Q$. In fact, we claim that $p(\tau)$ admits a unique lift which is a loop based at a point on the chosen cusp of $Q$ and such a lift is $\tau$, which is based at $x\neq z$. In general, lifts of $p(\tau)$ based at a point on $p^{-1}(p(x))\cap T^1_Q$ correspond to elements $\kappa\in\pi_1(T,p(x))$ such that $\kappa p(\tau)\kappa^{-1}\in p_*\pi_1(X,x)$.

\begin{lem}
\label{conjugate}
If $c\in\pi_1(\Sigma-\gamma)$ is not peripheral and $\kappa\in\pi_1(M_f-\gamma)- \pi_1(\Sigma-\gamma)$, then
\[
\kappa c\kappa^{-1}\not\in\pi_1(\Sigma-\gamma).
\]
\end{lem}

\begin{proof}
Recall that $\pi_1(M_f-\gamma)$ is a free product with amalgamation
\[
\pi_1(M_f-\gamma)=\pi_1(H_g-\gamma)*_{\pi_1(\Sigma-\gamma)}\pi_1(H_g-f(\gamma)).
\]
Write $\kappa$ as a reduced word $\kappa=a_1b_1\dots a_nb_n$ with $a_j\in\pi_1(H_g-\gamma)-\pi_1(\Sigma-f(\gamma))$ for $j>1$ and $b_j\in\pi_1(H_g-\gamma)-\pi_1(\Sigma-\gamma)$ for $j<n$. Since $\kappa\notin\pi_1(\Sigma-\gamma)$, either $b_n\not\in\pi_1(\Sigma-\gamma)$, or we can take $b_n$ to be the identity and $a_n\not\in\pi_1(\Sigma-\gamma)$.

The two cases can be dealt with in the same way, so we only consider the first case, that is, we assume that $b_n\not\in\pi_1(\Sigma-\gamma)$. We have
\[
\kappa c\kappa^{-1}=a_1b_1\dots a_nb_ncb_n^{-1}a_n^{-1}\dots b_1^{-1}a_1^{-1}.
\]
We claim that $b_ncb_n^{-1}\in\pi_1(H_g-\gamma)$ is not in $\pi_1(\Sigma-\gamma)$ provided that $c$ is not a peripheral element: This follows from the fact that $(H_g,\gamma)$ is pared acylindrical. In fact, suppose that $b_ncb_n^{-1}\in\pi_1(\Sigma-\gamma)$ and consider the homotopy between $c$ and $b_ncb_n^{-1}$ which takes place in $H_g-\gamma$. We have the following possibilities: If the homotopy is deformable into the cusp, then $c$ would be peripheral, which is ruled out by our initial assumption. As $(H_g,\gamma)$ is pared acylindrical, if the homotopy is not deformable to the cusp, then it is deformable to the boundary $\Sigma-\gamma$. In this case $b_n$ would be contained in $\pi_1(\Sigma-\gamma)$, which is again a contradiction. Therefore, if $c$ is not peripheral, the word $\kappa c\kappa^{-1}=a_1b_1\dots a_n(b_ncb_n^{-1})a_n^{-1}\dots b_1^{-1}a_1^{-1}$ is still reduced, and it contains a term not in $\pi_1(\Sigma-\gamma)$. Hence it cannot represent an element in $\pi_1(\Sigma-\gamma)$.
\end{proof}

We now return to the main argument for Proposition~\ref{width}. By Lemma~\ref{conjugate}, the loop $g_n(\tau)$ lifts to an arc $\eta$ with basepoint $z$ on the preferred cusp and another endpoint $u$ on a different component of $p^{-1}(\mb{T}_{\eta_3}(\gamma))$. We now observe that, if $\eta_3$ has been chosen sufficiently short in the beginning, no component of $p^{-1}(\mb{T}_{\eta_3}(\gamma))$ different from the cusps of $Q$ intersects the convex core $\mc{CC}(Q)$. 

\begin{lem}
\label{cusp lift}
If $\eta_3$ is sufficiently small, only depending on the topological type of $W$, then
\[
p^{-1}\mb{T}_{\eta_3}(\gamma)\cap\mc{CC}(Q)=\text{\rm cusp}(Q).
\]
\end{lem}

\begin{proof}
Let $p^{-1}\mb{T}_{\eta_3}(\gamma)=\bigsqcup_{j\in I}\mc{O}_j$ be the full preimage of the cusp under the covering projection $p:Q\to M_f-\gamma$. Suppose that a component $X\subset\partial\mc{CC}(Q)$ of the boundary of the convex core intersects one of the components $\mc{O}_j$ of the lift of the Margulis tube. Note that $p:X\to M_f-\gamma$ is a type preserving pleated surface in the homotopy class of the inclusion $\Sigma-\gamma\subset M_f-\gamma$ and that $p(\mc{O}_j)=\mb{T}_{\eta_3}(\gamma)$.
By Lemma~\ref{thick-to-thick}, the pleated surface $p(X)$ can only intersect $\mb{T}_{\eta_3}(\gamma)$ in its $\eta_0$-cuspidal part, for some uniform $\eta_0$, if $\eta_3$ has been chosen sufficiently small in the beginning. This means that $\mc{O}_j$ intersects the cuspidal part on $X$ and, hence, $\mc{O}_j$ is one of the cusps of~$Q$.
\end{proof}

We are now able to conclude: By Lemmas~\ref{conjugate} and~\ref{cusp lift}, the arc $\delta*\eta$ has an endpoint $G(w)\in T^1_Q\cap\mc{CC}(Q)$ and another one $u\in p^{-1}(\mb{T}_{\eta_3}(\gamma))-\mc{CC}(Q)$ outside the convex core. Therefore, it must intersect $\partial\mc{CC}(Q)=X\sqcup Y$. Say it intersects $X$. In particular $\rho_{Q_0}(G(w),X)\le d_{Q_0}(G(w),X)\le\ell(\delta*\eta)$, which, combined with the previously established inequalities gives us
\begin{flalign*}
\ell(\delta)+L_1 &\ge\ell(\delta*\eta)\\
 &\ge\rho_{Q_0}(G(w),X)\\
 &\ge\rho_{Q_0}(\alpha_m^*,\alpha_0^*)-\rho_{Q_0}(X,\alpha_0^*)-2K\\
 &\ge m/B-B-2K-\log(2L/\eta_3).
\end{flalign*}
This concludes the proof of the claim.
\end{proof}

\subsection{The proof of Theorem~\ref{thm:main}}
Combining Proposition~\ref{width} and Theorem~\ref{hk surgery}, we complete the proof of Theorem~\ref{thm:main} as follows.

\begin{proof}[Proof of Theorem~\ref{thm:main}]
We endow $M_f-\gamma$ with a complete finite volume hyperbolic structure, which exists by the assumption of $(H_g,\gamma)$ and $(H_g,f(\gamma))$ being pared acylindrical.

Let $\mu$ be the flat geodesic on $T=\partial\mb{T}_{\eta_3}(\gamma)$ that represents the filling slope needed to pass from $M_f-\gamma$ to $M_f$ (also known as the meridian of $\gamma$). By Theorem~\ref{hk surgery} and Theorem~\ref{deformation}, if $\text{\rm nl}(\mu)\ge 11$, then there exists a hyperbolic metric on $M_f$ for which $\gamma$ is a geodesic of length $\ell_{M_f}(\gamma)\le a/\text{\rm nl}(\mu)^2$ for some universal constant $a>0$.

By Proposition~\ref{width}, we have
\[
\ell(\mu)\ge cd_W(\alpha,\beta)-c,
\]
which is larger than $\frac{1}{2}cd_W(\mc{D},f\mc{D})$ provided that $d_W(\mc{D},f\mc{D})$ is sufficiently large.
Notice that $\text{\rm Area}(T)\le\eta_3\ell(\mu)$ whenever $\eta_3<\ell(\mu)$ (which we have given that $d_W(\mc{D},f\mc{D}))$ is sufficiently large); hence, 
\[
\text{\rm nl}(\mu)=\ell(\mu)/\sqrt{\text{\rm Area}(T)}\ge\sqrt{\ell(\mu)/\eta_3}.
\]

Thus $\text{\rm nl}(\mu)\ge 11$ if the subsurface projection of the disk sets to $W$ is sufficiently large. This shows that $M_f$ is hyperbolic.

In order to conclude, it remains to bound the length of $\gamma$ in $M_f$. This follows again from Theorem~\ref{hk surgery}: 
\[
\ell(\gamma)\le\frac{a}{\text{\rm nl}(\mu)^2}\le\frac{a}{\ell(\mu)/\eta_3}\le\frac{2a\eta_3/c}{d_W(\mc{D},f\mc{D})}.\qedhere
\] 
\end{proof}
\part{Model manifolds}\label{part:model}
\section{A gluing scheme}
\label{overview}

Here we outline a construction for the $\ep$-model metric which follows closely ideas of Brock and Dunfield~\cite{BD15} and Brock, Minsky, Namazi and Souto~\cite{BMNS16}. At the end of the discussion we formulate a criterion of applicability. In the entire section we fix a gluing map $f\in{\rm Diff}^+(\Sigma)$.

\subsection{Assembling simple pieces}
We identify a tubular neighborhood of the Heegaard surface $\Sigma\subset M_f=H_g\cup_fH_g$ with $\Sigma\times[1,4]$ such that $\Sigma$ is identified with $\Sigma\times\{2.5\}$. 

\begin{figure}[h]
\begin{center}
\begin{overpic}{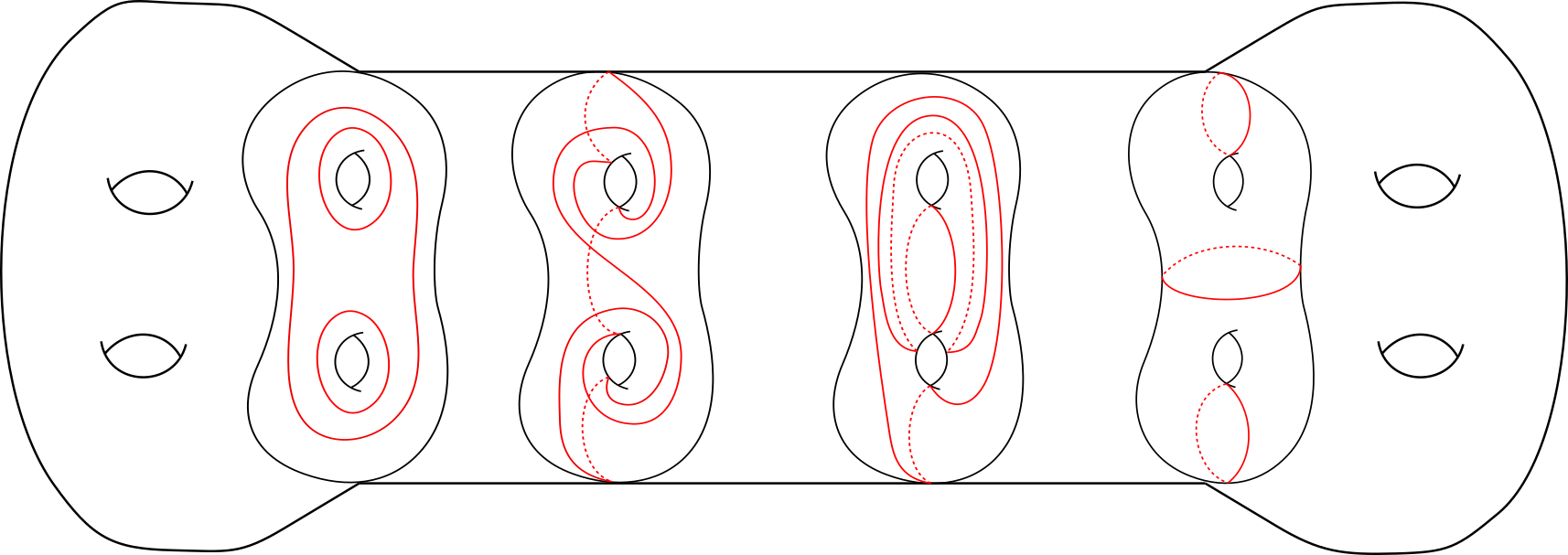}
\put (15,1) {$P_1\times\{1\}$}
\put (33,1) {$P_2\times\{2\}$}
\put (52,1) {$P_3\times\{3\}$}
\put (71,1) {$P_4\times\{4\}$}
\put (10,16) {$H_1$}
\put (30,16) {$\Omega_1$}
\put (48.5,16) {$Q$}
\put (67,16) {$\Omega_2$}
\put (85,16) {$H_2$}
\end{overpic}
\end{center}
\caption{Gluing.}
\label{fig:figure1}
\end{figure}

Given pants decompositions $P_1$, $P_2$, $P_3$, and $P_4$ in $\Sigma$, we consider
\[
\mb{M}:=M_f-(P_1\times\{1\}\sqcup P_2\times\{2\}\sqcup P_3\times\{3\}\sqcup P_4\times\{4\}).
\]

We isolate the five pieces (see Figure \ref{fig:figure1})
\begin{align*}
&H_1:=H_g-(\Sigma\times(1,2.5]\sqcup P_1\times\{1\}),\\
&\Omega_1:=\Sigma\times[1,2]-(P_1\times\{1\}\sqcup P_2\times\{2\}),\\
&Q:=\Sigma\times[2,3]-(P_2\times\{2\}\sqcup P_3\times\{3\}),\\
&\Omega_2:=\Sigma\times[3,4]-(P_3\times\{3\}\sqcup P_4\times\{4\})\text{, and}\\
&H_2:=H_g-(\Sigma\times[2.5,4)\sqcup P_4\times\{4\}).
\end{align*}

\begin{lem}
\label{five pieces}
Provided that
\begin{align}
&(H_g-\Sigma\times(1,2.5],P_1\times\{1\}),\nonumber\\
&(\Sigma\times[1,2],P_1\times\{1\}\sqcup P_2\times\{2\}),\nonumber\\
&(\Sigma\times[2,3],P_2\times\{2\}\sqcup P_3\times\{3\}),\label{5pieces}\tag{A}\\
&(\Sigma\times[3,4],(P_3\times\{3\}\sqcup P_4\times\{4\})\text{, and}\nonumber\\
&(H_g-\Sigma\times[2.5,4),P_4\times\{4\})\nonumber 
\end{align}
are all pared acylindrical manifolds, there exists a hyperbolic metric on $\mb{M}$ such that the restriction of the metric to each of the pieces is isometric to the convex core of the unique maximally cusped hyperbolic structure on the piece.
\end{lem}

\begin{proof}
As all the pieces are all pared acylindrical manifolds, we can endow each of them with a complete hyperbolic metric with totally geodesic boundary and rank one cusps at the pants decompositions. Indeed, we can take the convex cores of the (up to isotopy unique) maximally cusped hyperbolic structure on the pieces, which exist by Theorem~\ref{kms}. 

Since there is only one hyperbolic metric up to isotopy on a triply punctured sphere (such as the connected components of the complement of the pared loci in the boundaries of the pieces), we can arrange the metrics on the pieces such that they agree on the intersections and match together to give a complete finite volume hyperbolic metric on $\mb{M}$. 
\end{proof}

In order to pass from $\mb{M}$ to the closed 3-manifold $M_f$, we have to perform Dehn fillings on each cusp. This is the second step of the construction. The filling slopes
are completely determined by the identification of $\mb{M}$ with the drilled $M_f$: They are the meridians $\gamma$ of small tubular neighborhoods of the curves in $\alpha\times\{j\}\subset P_j\times\{j\}$ inside $\Sigma\times[1,4]$. Theorem~\ref{hk surgery} gives us sufficient conditions to guarantee that $M_f$ has a hyperbolic metric. Theorem \ref{drilling thm} gives sufficient conditions to ensure that $\mb{M}$ is $K$-bilipschitz to $M_f$ away from its cusps. We make this precise in the next subsection.

\subsection{A filling criterion}
To be able to apply Theorems~\ref{hk surgery}, \ref{drilling thm}, and~\ref{deformation}, we have to certify that the filling slopes have large normalized length. % is the main point that we have to address. 
We now discuss a criterion to check this condition. 

We define
\begin{align*}
&\mb{N}_1:=H_1\cup\Omega_1=H_g-P_2\times\{2\},\\
&\mb{N}_2:=H_2\cup\Omega_2=H_g-P_3\times\{3\}\text{, and}\\
&\mb{Q}:=\Omega_1\cup Q\cup\Omega_2=\Sigma\times[1,4]-(P_1\times\{1\}\sqcup P_2\times\{2\}\sqcup P_3\times\{3\}\sqcup P_4\times\{4\}). 
\end{align*}

The curves in $P_1$ and $P_4$ represent rank one cusps on $\partial\mb{Q}$ while the curves in $P_2$ and $P_3$ represent rank two cusps. Similarly, the curves $P_2$ and $P_3$ represent rank one cusps of $\mb{N}_1$ and $\mb{N}_2$ while the curves in $P_1$ and $P_4$ are cusps of rank two. 

We now try to understand what happens when we Dehn fill the {rank two} cusps in $\mb{Q},\mb{N}_1,\mb{N}_2$. Choosing the filling slopes to be the canonical meridians of $P_1,P_2,P_3,P_4$ in $M_f$, the Dehn filled manifolds are
\begin{align*}
&\mb{Q}^{{\rm fill}}=\Sigma\times[1,4]-(P_1\times\{1\}\sqcup P_4\times\{4\}),\\
&\mb{N}_1^{{\rm fill}}=H_g-(\Sigma\times(2,2.5]\sqcup P_2\times\{2\})\text{, and}\\
&\mb{N}_2^{{\rm fill}}=H_g-(\Sigma\times[2.5,3)\sqcup P_3\times\{3\}).
\end{align*}

We endow $\mb{Q}^{{\rm fill}},\mb{N}_1^{{\rm fill}},\mb{N}_2^{{\rm fill}}$ with a complete hyperbolic metric with totally geodesic boundary and rank one cusps at the curves $P_1\times\{1\}\sqcup P_4\times\{4\},P_2\times\{2\},P_3\times\{3\}$. Again, this is possible provided that
\begin{align}
&(\Sigma\times[1,4],P_1\times\{1\}\sqcup P_4\times\{4\}),\nonumber\\
&(H_g-\Sigma\times(2,2.5],P_2\times\{2\})\text{, and}\label{3pieces}\tag{B}\\
&(H_g-\Sigma\times[2.5,3),P_3\times\{3\}) \nonumber
\end{align}
are all pared acylindrical.

%We are now ready to explain the main idea. 

Recall that our goal is to show that the filling slopes we singled out on the rank two cusps of $\mb{Q},\mb{N}_1,\mb{N}_2$ have very large normalized length. First, observe that $\mb{Q},\mb{N}_1,\mb{N}_2$ isometrically embed in their doubles $D\mb{Q},D\mb{N}_1,D\mb{N}_2$ which are finite volume hyperbolic 3-manifolds. Checking that the filling slopes of the rank two cusps of $\mb{Q},\mb{N}_1,\mb{N}_2$ have large normalized length is the same as checking that they have large normalized length in $D\mb{Q},D\mb{N}_1,D\mb{N}_2$.

Also notice that 
\[
\mb{Q},\mb{N}_1,\mb{N}_2\subset\mb{Q}^{\text{\rm fill}},\mb{N}_1^{{\rm fill}},\mb{N}_2^{{\rm fill}}\subset D\mb{Q}^{\text{\rm fill}},D\mb{N}_1^{{\rm fill}},D\mb{N}_2^{{\rm fill}}.
\]
Again, the doubles are finite volume hyperbolic 3-manifolds.
  
The idea is as follows. Suppose that we can find hyperbolic metrics on $\mb{Q}^{\text{\rm fill}},\mb{N}_1^{{\rm fill}},\mb{N}_2^{{\rm fill}}$ such that the curves 
\[
P_2\times\{2\}\sqcup P_3\times\{3\},P_1\times\{1\},P_4\times\{4\}\subset\mb{Q}^{\text{\rm fill}},\mb{N}_1^{{\rm fill}},\mb{N}_2^{{\rm fill}}
\]
are very short geodesics. Then, $D\mb{Q},D\mb{N}_1,D\mb{N}_2$ are exactly the complete finite volume hyperbolic structures on the complement of those geodesic links $D\mb{Q}^{\text{\rm fill}}-P_1\times\{1\}\sqcup P_4\times\{4\},D\mb{N}_1^{{\rm fill}}-P_2\times\{2\},D\mb{N}_2^{{\rm fill}}-P_3\times\{3\}$.

By Theorem \ref{deformation}, the condition of having large normalized length for the meridians of $P_1\times\{1\}\sqcup P_4\times\{4\},P_2\times\{2\},P_3\times\{3\}$ is equivalent to the condition of being very short for those geodesics. 

The previous discussion leads us to make the following definition.
%\todo{I rewrote a bit the filling criterion. I found that we later say ''pants decomps satisfy filling criterion with parameter $\eta$'', but we had $\eta$ defined depending on some $L$. So I changed the order of things to be closer to the way it gets talked about below and in Sec 4. Hope it also seems better (or a t least not worse) for you.}

%The pants decompositions $P_1,P_2,P_3,P_4$ satisfy the filling criterion with parameter $\eta$ if the following holds: 

%The 8 blocks in \eqref{5pieces} and \eqref{3pieces} admit hyperbolic metrics with totally geodesic boundary provided that the pants decompositions $P_1\times\{1\},P_2\times\{2\},P_3\times\{3\},P_4\times\{4\}$ make the blocks pared acylindrical.
\begin{dfn}[Filling Criterion]
We say that pants decompositions $P_1,P_2,P_3,P_4$ \emph{satisfy the filling criterion with parameter $\eta>0$}, if~\eqref{5pieces} and~\eqref{3pieces} are all pared acylindrical and the following hold:
\begin{enumerate}
\item{$P_1\times\{1\}$ is isotopic to its geodesic realization in $\mb{N}_1^{\rm fill}$ and each component of this geodesic link has length at most $\eta$,}
\item{$P_4\times\{4\}$ is isotopic to its geodesic realization in $\mb{N}_2^{\rm fill}$ and each component of this geodesic link has length at most $\eta$, and}
\item{$P_2\times\{2\}\sqcup P_3\times\{3\}$ is isotopic to its geodesic realization in $\mb{Q}^{\rm fill}$ and each component of this geodesic link has length at most $\eta$.}
\end{enumerate}
\end{dfn}
By the above discussion we have the following.
For $L>0$, if $P_1,P_2,P_3,P_4$ satisfy the filling criterion with parameter
\[
\eta:=\min\left\{\frac{2\pi}{L^2+16.17},\frac{0.1396}{3g-3}\right\},
\]
then the normalized length of the filling slopes of $\mb{M}$ corresponding to $P_1\times\{1\},P_2\times\{2\},P_3\times\{3\},P_4\times\{4\}\subset M_f$ is at least $L$.
% provided that
%$(P_1,P_2,P_3,P_4)$ satisfies the filling criterion with parameter $\eta$.
Hence, by Theorems~\ref{hk surgery} and~\ref{deformation} (combined with Theorem \ref{drilling thm} for the furthermore-part), we established the following.

\begin{pro}
\label{model}
Fix $K\in(1,2)$. There exists $\eta>0$ such that the following holds. Suppose that there are four pants decompositions $P_1,P_2,P_3,P_4$ such that the {filling criterion} with parameter $\eta$ is satisfied. Then, $M_f$ admits a hyperbolic metric such that $\Gamma:=\bigsqcup_{j=1,2,3,4}{P_j\times\{j\}}$ is a geodesic link and each component has length at most $\eta$.
Furthermore, if $\mb{M}$ denotes the unique finite volume hyperbolic structure on $M_f-\Gamma$, then we have a $K$-bilipschitz diffeomorphism of pairs 
\[
\left(\mb{M}-\bigsqcup_{\alpha\in P_1\cup P_2\cup P_3\cup P_4}{\mb{T}_{\eta_3}(\alpha)}\right)\cong\left(M_f-\bigsqcup_{\alpha\in P_1\cup P_2\cup P_3\cup P_4}{\mb{T}_{\eta_3}(\alpha)}\right).
\]
%in the isotopy class of the inclusion $\mb{M}\subset M_f$.
\end{pro}

We conclude with a small remark. The model manifold technology of Minsky \cite{M10} and Brock, Canary and Minsky \cite{BrockCanaryMinsky:ELC2}, provides several tools to locate and measure the length of the geodesic representatives of $P_2\times\{2\}$ and $P_3\times\{3\}$ in $\mb{Q}^{{\rm fill}}$. However, the same technology is not available for handlebodies. This is the place where the difficulties arise. We get around the lack of a model for handlebodies by restricting ourselves to the geometry of the collars of $\mb{N}_1^{{\rm fill}},\mb{N}_2^{{\rm fill}}$ where, in some cases, we have a good amount of control (as in \cite{HV}).

\section{A family of examples}
\label{examples}

In this section we construct many examples satisfying Proposition \ref{model} (see Proposition \ref{pa examples}). Later, in Section~\ref{proof1}, we will show that this family is {generic} from the point of view of random walks.

\subsection{A family of hyperbolic mapping tori} 
For every pants decomposition $P\subset\Sigma$, we first construct hyperbolic mapping tori 
\[
T_\psi=\Sigma\times[0,1]/(x,0)\sim(\psi(x),1)
\]
with the property that the multicurve $P\times\{0\}\subset\Sigma\times\{0\}$ is a very short geodesic link. In particular, the infinite cyclic covering of $T_\psi$ (see Figure \ref{fig:figure2}) is a hyperbolic structure ${\hat T}_\psi$ on $\Sigma\times\mb{R}$ where the multicurves $\psi^nP\times\{n\}\subset\Sigma_n:=\Sigma\times\{n\}$ are very short geodesic links for every $n\in\mb{Z}$. 

\begin{figure}[h]
\begin{center}
\begin{overpic}{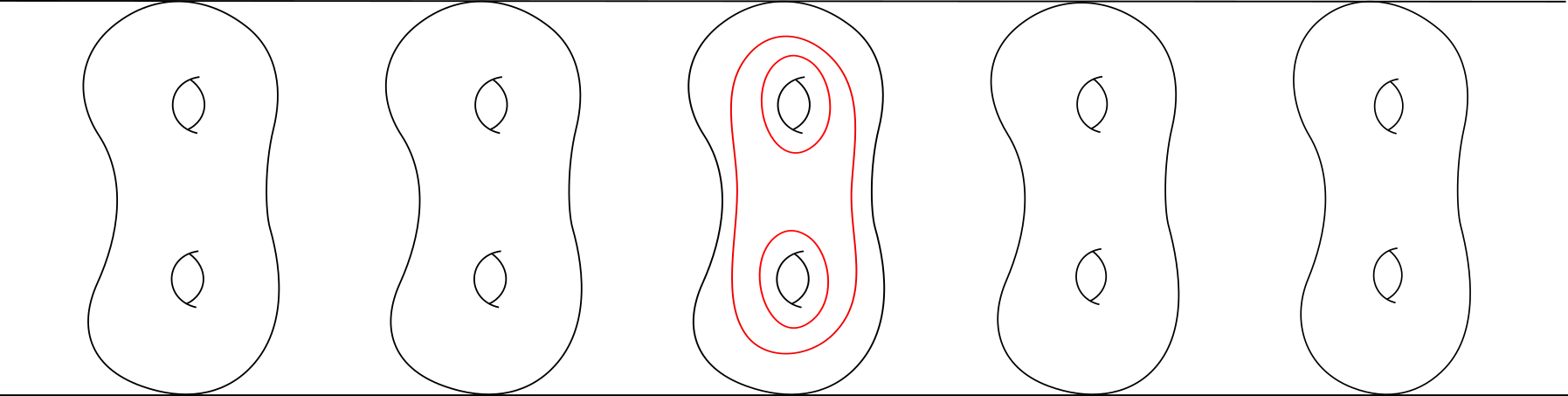}
\put (-1,12) {${\hat T}_\psi=\Sigma\times\mb{R}$}
\put (42,-3) {$\Sigma_j=\Sigma\times\{j\}$}
\put (42,12) {$P_j=\psi^j\times\{j\}$}
\put (73,12) {$[\Sigma_i,\Sigma_{i+1}]$}
%\put (38,27) {$\Psi(x,t)=(\psi(x),t+1)\curvearrowright$}
%\put (8,5) {$\mb{N}_1^{\rm fill}$}
%\put (88,5) {$\mb{N}_2^{\rm fill}$}
%\put (47,60) {$\mb{Q}^{\rm fill}$}
%\put (20,35) {$[\Sigma_{-n},\Sigma_1]\subset{\hat T}_\psi$}
%\put (28,32) {$\downarrow$}
%\put (58,35) {$[\Sigma_{-1},\Sigma_n]\subset{\hat T}_\phi$}
%\put (66,32) {$\downarrow$}
%\put (20,41) {$[\Sigma_0,\Sigma_n]\subset{\hat T}_\psi$}
%\put (28,45) {$\uparrow$}
%\put (58,41) {$[\Sigma_{-n},\Sigma_0]\subset{\hat T}_\phi$}
%\put (66,45) {$\uparrow$}
\end{overpic}
\end{center}
\caption{Infinite cyclic covering.}
\label{fig:figure2}
\end{figure}

%We denote by $[\Sigma_i,\Sigma_j]$ the region bounded by $\Sigma_i,\Sigma_j$ for every $i<j$. We call such an object a {\em mapping torus with a short pants decomposition $P$} and we call the gluing map $\psi$ a {\em pseudo-Anosov mapping class with a short pants decomposition $P$}.

\begin{dfn}[Mapping Class with a Short Pants Decomposition]
Let $\eta>0$ and let $P$ be a pants decomposition of $\Sigma$. If $\psi\in\text{\rm Mod}(\Sigma)$ is a mapping class such that its mapping torus 
\[
T_\psi\coloneqq\left(\Sigma\times[0,1]/(x,0)\sim(\psi x,1)\right)
\]
has a hyperbolic {{metric}} with respect to which $P\times\{0\}\subset T_\psi$ is a geodesic link of length at most $\eta$, then we call $\psi$ a {\em pseudo-Anosov mapping class with a short pants decomposition $P$ of length at most $\eta$}. 
\end{dfn}

These objects are abundant. Concretely, we have the following.

\begin{lem}
\label{pA with short pants}
For every pants decomposition $P\subset\Sigma$ and $\ep>0$, there exists a pseudo-Anosov mapping class $\psi$ with a short pants decomposition $P$ of length at most $\ep$.
%hyperbolic mapping torus $T_\psi$ for which $P\times\{0\}\subset T_\psi$ is a geodesic link of length at most $\ep$. 
\end{lem}
 
\begin{proof}
Let $\phi\in\text{\rm Mod}(\Sigma)$ be a mapping class such that  
\[
(\Sigma\times[0,1],P\times\{0\}\sqcup\phi P\times\{1\})
\]
is pared acylindrical. For example, one can choose $\phi$ to be a large power of any pseudo-Anosov. Consider the convex core $Q$ of the maximally cusped structure $Q(P,\phi P)$. The boundary $\partial Q$ consists of totally geodesic hyperbolic three punctured spheres that are paired according to $\phi$. We glue them together isometrically as prescribed by the pairing. The glued manifold is a finite volume hyperbolic 3-manifold diffeomorphic to 
\[
T_\phi-P\times\{0\}=\left(\Sigma\times[0,1]/(x,0)\sim(\phi x,1)\right)- P\times\{0\}.
\]
The curves in $P\times\{0\}$ represent rank two cusps. Observe that for each $\alpha\in P$ any choice of corresponding boundary torus $\partial\mb{T}(\alpha)\subseteq T_\phi-P\times\{0\}$ has a preferred meridian $m_\alpha$ (given as the boundary of an essential disc in $\mb{T}(\alpha)\subseteq T_\phi$) and a preferred longitude $l_\alpha$ (given as a curve in the fiber $\Sigma\times\{0\}\subseteq T_\phi$ that is parallel to $\alpha$). 

If we perform Dehn surgeries with slopes $m_\alpha+kl_\alpha$, the resulting closed manifold will be diffeomorphic to the mapping torus $T_\psi$, where $\psi=\phi\delta_P^k$ and $\delta_P$ is a Dehn multitwist about the pants decomposition $P$. By Thurston's Hyperbolic Dehn Surgery~\cite[Chapter~E.6]{BP92}) (or by invoking Theorems~\ref{hk surgery} and~\ref{deformation}), if $k$ is large enough, then the resulting manifold carries a hyperbolic metric for which the core curves of the added solid tori are very short geodesics. Hence, for $k$ large enough, $\psi$ is a pseudo-Anosov mapping class with a short pants decomposition $P$ of length at most $\ep$.
\end{proof}

%\begin{figure}[h]
%\begin{center}
%\begin{overpic}{figure4.png}
%\put (-5,20) {${\bar T}_\psi$}
%\put (10,-3) {$\Sigma_{-2}$}
%\put (29,-3) {$\Sigma_{-1}$}
%\put (49,-3) {$\Sigma_0$}
%\put (68,-3) {$\Sigma_1$}
%\put (88,-3) {$\Sigma_2$}
%\put (35,10) {$[\Sigma_{-1},\Sigma_0]$}
%\put (100,10) {$\cdots$}
%\put (-5,10) {$\cdots$}
%\end{overpic}
%\end{center}
%\caption{Mapping torus.}
%\label{fig:figure3}
%\end{figure}

\subsection{The criterion}
We now describe some criteria (Propositions \ref{approximation4} and \ref{approximation5}) to copy and paste the geometry of a hyperbolic mapping torus with a short pants decomposition into the collar geometry of maximally cusped handlebodies $H_g$ and trivial bundles $\Sigma\times[0,1]$. These structures will automatically satisfy the filling criterion for a suitable choice of pants decompositions (Proposition~\ref{pa examples}). 

\begin{pro}
\label{approximation4}
There exists $\eta>0$ with the following property. %Let $\psi$ be a pseudo-Anosov mapping class with a short pants decomposition $P$ of length at most $\eta$.
For every $K>0$ and every pseudo-Anosov mapping class $\psi$, there exists an $m=m(\psi,K)>0$ such that the following holds. Let $P$ be a pants decomposition that makes $\psi$ a pseudo-Anosov mapping class with a short pants decomposition of length at most $\eta$. Define $P_{-n}:=\psi^{-n}P$ for $n\in\mb{N}$. Suppose that for some $n\ge m$ we have
\[
d_{\mc{C}}(P_n,\mc{D})\ge d_{\mc{C}}(P,P_n)+d_{\mc{C}}(P,\mc{D})-K.
\]
Consider the maximally cusped structure $N_n=H(P_n)$ on $\textrm{int}(H_g)$. Identify a collar of $\Sigma=\partial H_g$ with $\Sigma\times[0,2]$ where $\partial H_g=\Sigma\times\{0\}$. Then the geodesic realization of $P_{n-1}\times\{1\}$ in $N_n$ is isotopic to $P_{n-1}\times\{1\}$ and has length at most $2\eta$. 
\end{pro} 

We will prove Proposition \ref{approximation4} at the end of the section. The idea is that if $n$ is large enough then the geometry of $N_n$ near the boundary of its convex core
%\as{Removed ``of''} $\mc{CC}(N)$ of $N$ 
looks like the geometry of the region $[\Sigma_0,\Sigma_n]\subset{\hat T}_\psi$ and there the geodesic realization of $P_{n-1}$ lies as a short geodesic link on the level surface $\Sigma_{n-1}$ (see Figure \ref{fig:figure3}).

\begin{figure}[h]
\begin{center}
\begin{overpic}{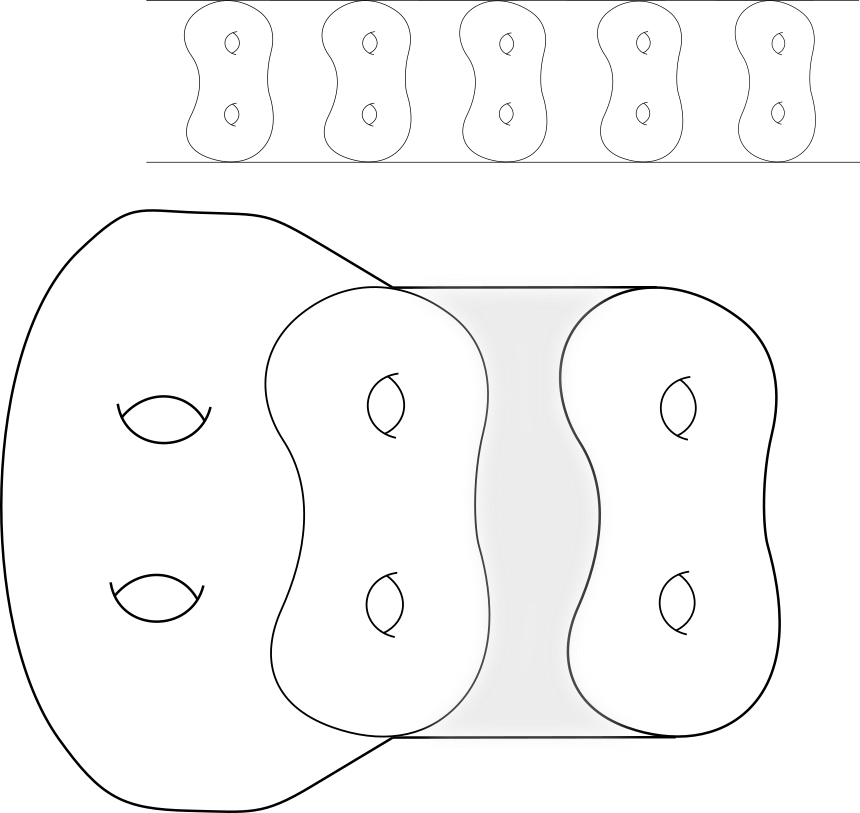}
\put (-22,32) {$H(P_n)$}
\put (5,80) {${\hat T}_\psi$}
\put (60,68) {$\downarrow$}
\put (45,80) {$[\Sigma_0,\Sigma_n]$}
%\put (40,32) {$\psi^{-1}P\times\{-1\}$}
%\put (8,5) {$\mb{N}_1^{\rm fill}$}
%\put (88,5) {$\mb{N}_2^{\rm fill}$}
%\put (47,60) {$\mb{Q}^{\rm fill}$}
%\put (20,35) {$[\Sigma_{-n},\Sigma_1]\subset{\hat T}_\psi$}
%\put (28,32) {$\downarrow$}
%\put (58,35) {$[\Sigma_{-1},\Sigma_n]\subset{\hat T}_\phi$}
%\put (66,32) {$\downarrow$}
%\put (20,41) {$[\Sigma_0,\Sigma_n]\subset{\hat T}_\psi$}
%\put (28,45) {$\uparrow$}
%\put (58,41) {$[\Sigma_{-n},\Sigma_0]\subset{\hat T}_\phi$}
%\put (66,45) {$\uparrow$}
\end{overpic}
\end{center}
\caption{Collars handlebodies.}
\label{fig:figure3}
\end{figure}

Similarly, we will also use the following variation (see also Figure \ref{fig:figure4}):

\begin{pro}
\label{approximation5}
There exists $\eta>0$ with the following property.
For every $K>0$ and every two pseudo-Anosov mapping class $\phi$ and $\psi$,
there exists an $m=m(\phi,\psi,K)>0$ such that the following holds.
 Let $P$ and $R$ be a pants decomposition that makes $\phi$ and $\psi$, respectively, pseudo-Anosov mapping class with a short pants decomposition of length at most $\eta$.
%Let $\psi,\phi$ be pseudo-Anosov mapping classes with short pants decompositions $P,R$ of length at most $\eta$. For every $K>0$, there exists an $m(\phi,\psi,K)>0$ such that the following holds.
Define $P_n,R_n:=\psi^nP,\phi^nR$ for $n\in\mb{Z}$. Suppose that for some $n\ge m$ we have
\[
d_{\mc{C}}(P_{-n},R_n)\ge d_{\mc{C}}(P_{-n},P)+d_{\mc{C}}(P,R)+d_{\mc{C}}(R,R_n)-K.
%d_{\mc{C}}(P_{-1},R_1)\ge d_{\mc{C}}(P_{-1},P_n)+d_{\mc{C}}(P_n,R_{-n})+d_{\mc{C}}(R_{-n},R_1)-K.
\]
Consider the maximally cusped structure $Q_n=Q(P_{-n},R_n)$ on ${\textrm{int}}(\Sigma\times[0,3])$. Then the geodesic realization of $P_{-n+1}\times\{1\}\sqcup R_{n-1}\times\{2\}$ in $Q$ is isotopic to $P_{-n+1}\times\{1\}\sqcup R_{n-1}\times\{2\}$ and has length at most $2\eta$.
\end{pro}

%Schematically, the situation is represented by Figure \ref{fig:figure1}

\begin{figure}[h]
\begin{center}
\begin{overpic}{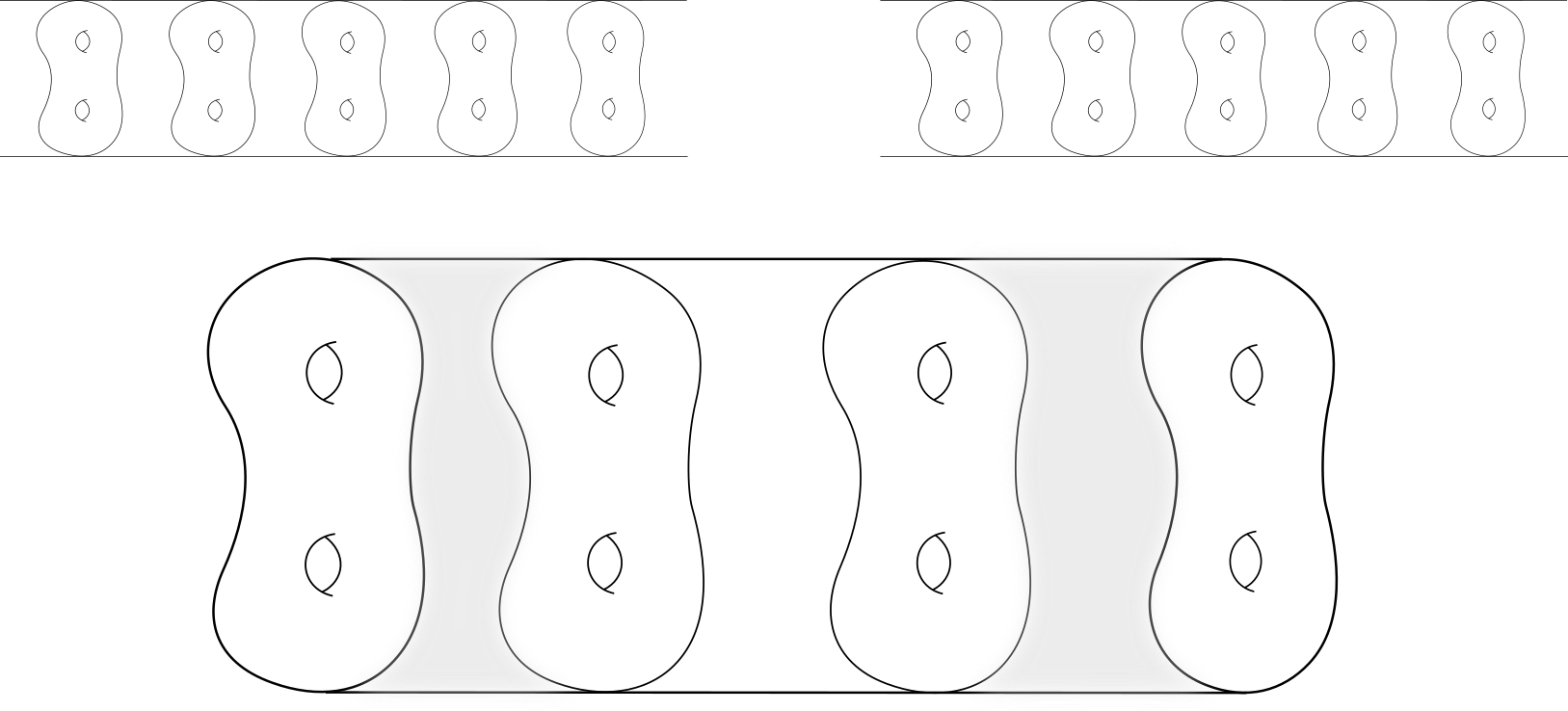}
\put (-5,15) {$Q(P_{-n},R_n)$}
\put (-5,38) {${\hat T}_\psi$}
\put (100,38) {${\hat T}_\phi$}
\put (28,32) {$\downarrow$}
\put (68,32) {$\downarrow$}
\put (23,38) {$[\Sigma_{-n},\Sigma_0]$}
\put (61,38) {$[\Sigma_0,\Sigma_n]$}
%\put (88,5) {$\mb{N}_2^{\rm fill}$}
%\put (47,60) {$\mb{Q}^{\rm fill}$}
%\put (20,35) {$[\Sigma_{-n},\Sigma_1]\subset{\hat T}_\psi$}
%\put (28,32) {$\downarrow$}
%\put (58,35) {$[\Sigma_{-1},\Sigma_n]\subset{\hat T}_\phi$}
%\put (66,32) {$\downarrow$}
%\put (20,41) {$[\Sigma_0,\Sigma_n]\subset{\hat T}_\psi$}
%\put (28,45) {$\uparrow$}
%\put (58,41) {$[\Sigma_{-n},\Sigma_0]\subset{\hat T}_\phi$}
%\put (66,45) {$\uparrow$}
\end{overpic}
\end{center}
\caption{Collars I-bundles.}
\label{fig:figure4}
\end{figure}

Having in mind the application to the random 3-manifold setup, we combine the two propositions in the following criterion which produces many examples of pants decompositions that satisfy the assumptions of Proposition \ref{model}:

\begin{pro}
\label{pa examples}
There exists $\eta>0$ with the following property. Let $o\in\T$ be a basepoint such that $\Upsilon(o)\in\mc{D}$. Let $\psi,\phi$ be pseudo-Anosov mapping classes with short pants decompositions $P,R$ of length at most $\eta$. Let $l_\psi,l_\phi:\mb{R}\to\T$ be their Teichmüller geodesics. For every $\delta>0$ there exist $B=B(\delta)$ and $h=h(\delta,\psi,\phi)$ such that the following holds.

Let $f$ be a gluing map. Set $S_0:=o,S_5:=fo$. Suppose that there are four points $S_1,S_2,S_3,S_4\in[o,fo]$ with the following properties:
\begin{enumerate}[\rm (i)]
\item{Along the Teichmüller segment $[o,fo]$ we have $S_j<S_{j+1}$ and their distance is at least $h$.}
\item{$[S_1,S_2]$ and $[S_3,S_4]$ $\delta$-fellow travel $l_\psi$ and $l_\phi$ respectively.}
\item{We have $d_{\mc{C}}(\Upsilon[S_1,S_2],\mc{D})\ge h$ and $d_{\mc{C}}(\Upsilon[S_3,S_4],f\mc{D})\ge h$.}
\end{enumerate}

Then, there exist $a,b\in\mb{Z}$ such that 
\[
(P_1,P_2,P_3,P_4)=(\psi^{a-1}P,\psi^aP,\phi^bR,\phi^{b+1}R)
\]
satisfy the {filling criterion} with parameter $\eta$ for the gluing map $f$ and
\[
d_{\mc{C}}(P_2,\Upsilon[S_1,S_2]),d_{\mc{C}}(P_3,\Upsilon[S_3,S_4])\le B.
\]
\end{pro}

In Section~\ref{proof1}, we will use Proposition \ref{pa examples} to detect whether, for a random gluing map $f$, $M_f$ can be described as one of the examples we constructed simply by staring at the geometry of the Teichmüller segment $[o,fo]$, where $o\in\T$ is some basepoint that we will fix once and for all.

\subsection{The proof of Proposition \ref{pa examples}}
Assuming Propositions~\ref{approximation4} and\ref{approximation5}, we give a proof of Proposition~\ref{pa examples}. 
\begin{proof}[Proof of Proposition \ref{pa examples}]
Let $P_n,R_n:=\psi^nP,\phi^nR$. 

As $\psi,\phi$ are pseudo-Anosov maps, they act on Teichmüller space as hyperbolic isometries with respect to the Teichmüller metric: There are $\psi$-, $\phi$-invariant Teichmüller geodesics $l_\psi,l_\phi:\mb{R}\to\T$, parametrized by arc length, on which $\phi,\psi$ act as $\psi l_\psi(t)=l_\psi(t+L(\psi))$ and $\phi l_\phi(t)=l_\phi(t+L(\phi))$. The quantities $L(\psi),L(\phi)>0$ are the translation lengths of $\psi,\phi$.
% and can also be characterized by $L(\psi)=\min\{d_\T(S,\psi S)\},L(\phi)=\min\{d_\T(S,\psi S)\}$. 

Let $L:=\max\{L(\psi),L(\phi)\}$ be the maximum of the translation length of $\psi,\phi$. Let $h:=2nL$, with $n$ much larger than $m(\psi),m(\phi),m(\psi,\phi)$ as given by Propositions \ref{approximation4} and \ref{approximation5}.

As the geodesic representatives of $P_n,R_n$ are very short in the infinite cyclic coverings ${\hat T}_\psi,{\hat T}_\phi$ of the hyperbolic mapping tori $T_\psi,T_\phi$, it follows from work of Minsky \cite{M00} combined with the structural results of Masur and Minsky \cite{MasurMinsky:I}, \cite{MasurMinsky:II},  that there exists a uniform constant $B$, only depending on $\Sigma$, and $t,r\in\mb{R}$ such that 
\[
d_{\mc{C}}(P_j,\Upsilon l_\psi(t+jL(\psi))),d_{\mc{C}}(R_i,\Upsilon l_\phi(r+iL(\phi)))\le B
\] 
for every $j,i\in\mb{Z}$. Up to reparametrization, we can assume $t=r=0$ and simplify the notation by introducing 
\[
E:=l_\psi(0),F:=l_\phi(0)
\]
and 
\[
E_j:=l_\psi(jL(\psi))=\psi^jE,F_i:=l_\phi(iL(\phi))=\phi^iF.
\] 
By the above discussion $d_{\mc{C}}(P_j,\Upsilon(E_j)),d_{\mc{C}}(R_i,\Upsilon(F_i))\le B$ for every $i,j\in\mb{Z}$.

Notice that the $E_j$ and $F_i$ are linearly aligned along $l_\psi$ and $l_\phi$.

By assumption (ii) of Proposition~\ref{pa examples} and the choice of $h\ge 2nL(\psi),2nL(\phi)$, we can find $a,b\in\mb{Z}$ such that the following hold:
\begin{equation}\label{eq:deltaNbh}
\begin{split}
\bullet\;&\text{$l_\psi[(a-n)L(\psi),(a+n)L(\psi)]$ is in a $\delta$-neighborhood of $[S_1,S_2]$.}\\
\bullet\;&\text{$l_\phi[(b-n)L(\phi),(b+n)L(\phi)]$ is in a $\delta$-neighborhood of $[S_3,S_4]$.}
\end{split}
%\item{$P_{a+j}$ lies $B$-close to $\Upsilon l_\psi(t+jL(\psi))$ for every $j\in[-n,n]$.}
%\item{$R_{b+i}$ lies $B$-close to $\Upsilon l_\phi(r+iL(\phi))$ for every $i\in[-n,n]$.}
\end{equation}
Again, up to renumbering, we can assume $a=b=0$. 

%Since the segments $l_\psi[-nL(\psi),nL(\psi)]$ and $l_\phi[-nL(\phi),nL(\phi)]$ are in a $\delta$-neighborhood of the segments $[S_1,S_2],[S_3,S_4]$, for every $j,i\in[-n,n]$
By~\eqref{eq:deltaNbh} we find surfaces $X_j\in[S_1,S_2],Y_i\in[S_3,S_4]$ that are $\delta$-close to $E_j\in l_\psi[-nL(\psi),nL(\psi)],F_i\in l_\phi[-nL(\phi),nL(\phi)]$, respectively. Furthermore, up to increasing a little bit $\delta$ in a uniform way if necessary, we can also assume that the $X_j$ and $Y_i$ are linearly aligned along $[S_1,S_2]$ and $[S_3,S_4]$ (as the same property holds for $E_j,F_i$). 

As the projection $\Upsilon$ is Lipschitz, we have that $\Upsilon(X_j),\Upsilon(Y_i)$ are uniformly close (in terms of $\delta$) to $\Upsilon(E_j),\Upsilon(F_i)$ and the latter are uniformly close to $P_j,R_i$. 

We now show that 
\[
(P_1,P_2,P_3,P_4)=(P_{-1},P,R,R_1)
\]
satisfies the filling criterion with parameter $2\eta>0$, where $\eta$ is the minimum of the $\eta$s in Propositions~\ref{approximation4} and~\ref{approximation5}.

We first make sure that the eight building blocks \eqref{5pieces} and~\eqref{3pieces} are pared acylindrical
(and hence admit maximally cusped structures).
As the maps $\psi,\phi$ are pseudo-Anosov, the pared acylindrical assumptions are automatically satisfied for $(P_1,P_2)$ and $(P_3,P_4)$. In order to check the others, we use Lemma~\ref{lem:d>=3impliesparedacyl}.

Let us check that $d_{\mc{C}}(P_1,\mc{D}),d_{\mc{C}}(P_2,\mc{D})\ge 3$: We know that $P_1,P_2$ are uniformly close to $\Upsilon(E_{-1}),\Upsilon(E)$ which, in turn, are $\delta$-close to $[S_1,S_2]$. By the assumption (iii), we have $d_{\mc{C}}(\Upsilon[S_1,S_2],\mc{D})\ge h$. Thus, provided that $h$ is large enough, the claim holds. The same argument shows that $d_{\mc{C}}(P_3,f\mc{D}),d_{\mc{C}}(P_4,f\mc{D})\ge 3$ provided that $h$ is sufficiently large.

Let us now check that $d_{\mc{C}}(P_1,P_4),d_{\mc{C}}(P_2,P_3)\ge 3$: Consider $(P_2,P_3)$ first, the argument for the other pair $(P_1,P_4)$ is completely analogous. We know that $P_2,P_3$ are uniformly close to $\Upsilon(E),\Upsilon(F)$. By \eqref{eq:deltaNbh}, we have that $E,E_n=\psi^nE,F_{-n}=\phi^{-n}F,F$ are $\delta$ close to points $X,X_n,Y_{-n},Y$ that appear in this order along the segment $[o,fo]$ (as described above). So, in order to prove that $d_{\mc{C}}(P_2,P_3)\ge 3$ it is enough to show that $d_{\mc{C}}(\Upsilon(X),\Upsilon(Y))$ is large. As the restriction of $\Upsilon$ to a Teichmüller segment is a uniform unparametrized quasi-geodesic, we have
\begin{align*}
d_{\mc{C}}(\Upsilon(X),\Upsilon(Y))\gtrsim d_{\mc{C}}(\Upsilon(X),\Upsilon(X_n))+d_{\mc{C}}(\Upsilon(X_n),\Upsilon(Y_{-n}))\\
+d_{\mc{C}}(\Upsilon(Y_{-n}),\Upsilon(Y)).
\end{align*} 
%l_\psi(0),l_\psi(nL(\psi))=\psi^n(l_\psi(0)),l_\phi(-nL(\phi))=\phi^{-n}l_\phi(0),l_\phi(0)
Here the symbol $\gtrsim$ means greater or equal up to a uniform additive constant.

The terms $d_{\mc{C}}(\Upsilon(X),\Upsilon(X_n)),d_{\mc{C}}(\Upsilon(Y_{-n}),\Upsilon(Y))$ are comparable with $n$ times the translation distances of $\psi,\phi$ on the curve graph, so they are as large as we want provided that $n$ is large enough.

Thus, the pared acylindrical assumptions are satisfied and the eight pieces admit maximally cusped structures. We now check the length-isotopy properties as given by the filling criterion. 

Consider the maximally cusped structures $N_1=H(P),N_2=H(R)$ on $H_g$ and $Q=Q(P_{-1},R_1)$ on $\Sigma\times[0,3]$. Identify a collar of the boundary $\Sigma=\partial H_g$ with $\Sigma\times[0,2]$ in such a way that $\partial H_g=\Sigma\times\{0\}$. Then 

\begin{enumerate}
\item{\label{((1))}The collection $P_{-1}\times\{1\}$ is isotopic to its geodesic realization in $N_1$ and has length at most $\eta$.}
\item{\label{((2))}The collection $R_1\times\{1\}$ is isotopic to its geodesic realization in $N_2$ and has length at most $\eta$.}
\item{\label{((3))}The collection $P\times\{1\}\sqcup R\times\{2\}$ is isotopic to its geodesic realization in $Q$ and has length at most $\eta$.}
\end{enumerate} 

We use Proposition~\ref{approximation4} to check the~\eqref{((1))} and~\eqref{((2))}. We use Proposition~\ref{approximation5} to check~\eqref{((3))}. 

According to Proposition \ref{approximation4} applied to $P_{-n}$ for~\eqref{((1))} and to $R_n$ for~\eqref{((2))}, to prove~\eqref{((1))} and~\eqref{((2))} we only need that:
\begin{enumerate}[(a)]
\item{$n\ge m(\psi),m(\phi)$.}
\item{$d_{\mc{C}}(P,\mc{D})\ge d_{\mc{C}}(P_{-n},P)+d_{\mc{C}}(P_{-n},\mc{D})-K$.}
\item{$d_{\mc{C}}(R,f\mc{D})\ge d_{\mc{C}}(R,R_n)+d_{\mc{C}}(R_n,f\mc{D})-K$.}
\end{enumerate}

According to Proposition \ref{approximation5}, for the last point we need:
\begin{enumerate}[(a')]
\item{$n\ge m(\psi,\phi)$.}
\item{$d_{\mc{C}}(P_{-1},R_1)\ge d_{\mc{C}}(P_{-1},P_n)+d_{\mc{C}}(P_n,R_{-n})+d_{\mc{C}}(R_{-n},R_1)-K$.}
\end{enumerate}

As $P_j,R_i$ are uniformly close to $\Upsilon(X_j),\Upsilon(Y_i)$, it is enough to check the analogue conditions enumerated above for $\Upsilon(X_j),\Upsilon(Y_i)$.

Properties (b) and (c) follow from Lemma \ref{far from disks}: We only discuss (b) as (c) is completely analogous. Again, we uniformly approximate $P_j,R_i$ with $\Upsilon(X_j),\Upsilon(Y_i)$ with $X_j\in[S_1,S_2]$ and $Y_i\in[S_3,S_4]$. Consider the restriction of $\Upsilon$ to the Teichmüller segment $[o,fo]$. By Lemma \ref{far from disks} there exists a uniform $h_0$ such that if $d_{\mc{C}}(S,\mc{D})\ge 2h_0$ for some $S\in[o,fo]$, then for every $S'>S$ we have 
\[
d_{\mc{C}}(\Upsilon(S'),\mc{D})\ge d_{\mc{C}}(\Upsilon(S'),\Upsilon(S))+d_{\mc{C}}(\Upsilon(S),\mc{D})-h_0.
\]

As we assumed that $d_{\mc{C}}(\Upsilon[S_1,S_2],\mc{D})\ge h$, if $h$ is larger than $h_0$, then the above inequality holds for $S=X_{-n}<X_1=S'$. Up to increasing uniformly $h_0$, the same type of inequality holds for $P_{-n},P_1$.   

Property (b') follows from the fact that the restriction of $\Upsilon$ to the Teichmüller geodesic $[o,fo]$ is an unparametrized uniform quasi-geodesics and, along $[o,fo]$, we have $X<X_n<Y_{-n}<Y$, so
\begin{align*}
d_{\mc{C}}(\Upsilon(X),\Upsilon(Y))\gtrsim d_{\mc{C}}(\Upsilon(X),\Upsilon(X_n))&+d_{\mc{C}}(\Upsilon(X_n),\Upsilon(Y_{-n}))\\
&+d_{\mc{C}}(\Upsilon(Y_{-n}),\Upsilon(Y)). \qedhere
\end{align*}
\end{proof}

%Property (b') follows from the fact that $\Upsilon$ sends the Teichmüller geodesic $[o,fo]$ to an unparametrized uniform quasi-geodesics and, along $[o,fo]$, we have $X_0<X_n<Y_0$ and $X_0<Y_{-n}<Y_0$. \qed

\subsection{The proof of Proposition~\ref{approximation4}}
We adapt a strategy of Namazi and Souto \cite{NS09} and Brock, Minsky, Namazi, and Souto~\cite{BMNS16} and argue by contradiction.

We briefly outline the argument: Suppose that for large $n$ we have that
\begin{equation}\label{eq:assumtiontowardscont}
d_{\mc{C}}(P_n,\mc{D})\ge d_{\mc{C}}(P_n,P)+d_\mc{C}(P,\mc{D})-K,
\end{equation}
but the geodesic realization of the pants decomposition $P_{n-1}$ in the maximally cusped handlebody $N_n=H(P_n)$ does not satisfy the conclusion of the proposition. We show that the sequence of hyperbolic manifolds $N_n$ %on $H_g$
converges to a hyperbolic manifold $Q$ diffeomorphic to $\Sigma\times(-\infty,0)$ on which the curves $\psi^jP\times\{j\}$ are geodesic links of length at most $\eta$ for $j<0$ and the curves in $P\times\{0\}$ represent rank one cusps. Convergence is essentially in the sense of the algebraic and geometric convergence.
This implies the following. For fixed $T>0$, if $n$ is large enough, then there is a 2-bilipschitz embedding $\phi_n:\Sigma\times[-T,-1/T]\to N_n$ such that the restriction of $\phi_n$ to $\Sigma\times\{-1\}$ is homotopic to the composition of the inclusion $f_n:\Sigma=\partial H_g\to H_g$ with $\psi^n:\Sigma\to\Sigma$. As the map $\phi_n$ is 2-bilipschitz, a small variation of stability of quasi-geodesics (Lemma \ref{isotopy}) gives that $\phi_n(\psi^{-1}P\times\{1\})$ is isotopic to its geodesic realization in $N_n$ which has length at most $2\eta$. Notice that $\phi_n(\psi^{-1}P\times\{-1\})$ is homotopic to $f_n(\psi^{n-1}P)$. A topological argument then shows that $\phi_n(\psi^{-1}P\times\{-1\})$ is isotopic to $\psi^{n-1}P\subset\partial H_g$. Thus, we have shown that, for $n$ large, the geodesic representatives of $\psi^{n-1}P\subset\partial H_g$ in $N_n$ are isotopic to $\psi^{n-1}P$ and have length at most $2\eta$. This contradicts the initial assumptions.

For convenience, the proof, which will take the rest of the section, is divided into small steps. 
In the proof we use the general fact, which we prove in Appendix~\ref{appendix}. 

\begin{lem}
\label{isotopy}
There exists $\eta<\eta_3/2$ such that the following holds. Let $\mb{T}_{\eta_3}$ be a Margulis tube with core geodesic $\alpha$ of length $l(\alpha)\le\eta$. Suppose that there exists a $2$-bilipschitz embedding of the tube in a hyperbolic 3-manifold $f:\mb{T}_{\eta_3}\to M$. Then $f(\alpha)$ is homotopically non-trivial and it is isotopic to its geodesic representative within $f(\mb{T}_{\eta_3})$.
\end{lem}

\begin{proof}[Proof of Proposition~\ref{approximation4}] {We choose the constant $\eta$ as the minimum of the constant from Lemma~\ref{isotopy}
%(which is independent of the rest of the argument)
and the constant $\eta_{{\rm drill}}$ for $K=2$ from Theorem~\ref{drilling thm}.} Let $\psi$ be a pseudo-Anosov mapping class with a short pants decomposition $P$ of length at most $\eta$.
For a fixed $K>0$,
assume towards a contradiction that there exist arbitrarily large $n$ for which~\eqref{eq:assumtiontowardscont} holds and either the geodesic realization of $P_{n-1}\times\{1\}$ in $N$ is not isotopic to $P_{n-1}\times\{1\}$ or it is isotopic, but has length strictly larger than $2\eta$.

We begin by discussing the algebraic convergence of the manifolds $N_n$.

Let $\rho_n:\pi_1(\partial H)\to{\rm PSL}_2(\mb{C})$ be the composition of the holonomy $\pi_1(H)\to{\rm PSL}_2(\mb{C})$ of $N_n$ with the map $\pi_1(\partial H)\to\pi_1(H)$ induced by the inclusion of the boundary $\partial H\subset H$. We consider the sequence of representations $\theta_n:=\rho_n\psi^n_*$.

\begin{lem}
\label{asymptotically faithful}
For every $\gamma\in\pi_1(\Sigma)$ there exists $n_\gamma>0$ such that for every $n\ge n_\gamma$ we have $\gamma\not\in{\rm ker}(\theta_n)$.
\end{lem}

\begin{proof}
We can assume that $\gamma$ is primitive. By the Loop Theorem, if $\gamma\in{\rm ker}(\theta_n)$, or, equivalently, $\psi^n(\gamma)\in{\rm ker}(\rho_n)$, then, after representing $\gamma$ as a 4-valent graph on $\Sigma$, there exists an essential simple cycle $\alpha_n\subset\psi^n(\gamma)$ which is also compressible. Since there are only finitely many simple cycles on $\gamma$, we can assume that $\alpha_n=\psi^n(\alpha)$ with $\alpha\subset\gamma$ a fixed simple cycle. Therefore, $\alpha\in{\rm ker}(\theta_n)$ which means $\psi^n(\alpha)\in\mc{D}$. However, 
\begin{align*}
d_{\mc{C}}(\psi^n(\alpha),\mc{D}) &\ge d_{\mc{C}}(\psi^n(P),\mc{D})-d_{\mc{C}}(\psi^n(P),\psi^n(\alpha))\\
 &=d_{\mc{C}}(\psi^n(P),\mc{D})-d_{\mc{C}}(P,\alpha)\\
 &\ge d_{\mc{C}}(\psi^n(P),P)-K\to\infty.\qedhere
\end{align*}
\end{proof}

We now obtain algebraic convergence of the sequence $\theta_n$ in the following sense.

\begin{lem}
\label{algebraic convergence}
Up to subseqences and conjugating, $\theta_n$ algebraically converges to a discrete and faithful representation $\theta:\pi_1(\Sigma)\to{\rm PSL}_2(\mb{C})$.
\end{lem}
In different language, Lemma~\ref{algebraic convergence} states that $\theta_n$ has a subsequence that converges in the character variety.
\begin{proof}
The proof is an application of \cite[Theorem 3.1]{BMNS16}. We briefly recall the statement in our setup. Let $\theta_n:\pi_1(\Sigma\times[-1,1])\to{\rm PSL}_2(\mb{C})$ be an {\em eventually faithful} sequence of representations, that is, a sequence of representations with discrete image satisfying the conclusion of Lemma \ref{asymptotically faithful}. Suppose that 
\begin{itemize}
\item{Each $\theta_n$ maps every curve in $P\times\{1\}$ to a parabolic isometry.}
\item{There exists a sequence of simple closed curve $\zeta_n\subset\Sigma\times\{-1\}$ converging to a filling lamination $\lambda$ and such that $\ell_{\theta_n}(\zeta_n)$ remains bounded.}
\end{itemize} 
Then there exists a subsequence that converges algebraically to a discrete and faithful representation $\theta$. 
(In fact \cite[Theorem 3.1]{BMNS16} allows for $\zeta_n$ to be certain multicurves called markings, but we only apply the theorem for simple closed curves.)
%In general, in the terminology of \cite{BMNS16}, a marking denotes an essential multicurve $\mu\subset\Sigma$ together with at most one transversal for each $\gamma\subset\mu$, that is, an essential curve $\alpha\subset\Sigma-(\mu-\gamma)$ intersecting $\gamma$ minimally. However, a marking can also be just a simple closed curve. 

In our situation we can choose $\zeta_n:=\psi^{-n}\delta$, where $\delta\in\mc{D}$ is a fixed disk-bounding curve, so, by construction, $\ell_{\theta_n}(\zeta_n)=\ell_{\rho_n}(\delta)=0$. We further note that for every disk bounding curve $\delta\in\mc{D}$, the sequence $\psi^{-n}\delta\in\mc{PML}$ converges to the projective class  $\lambda\in\mc{PML}$ of the unstable lamination of $\psi$, which is a filling lamination. %$\zeta_n$ converges to {a filling lamination.}
%the repelling lamination of the pseudo-Anosov map $\psi$ and it is well-known that such lamination is filling.
Hence, we can apply \cite[Theorem 3.1]{BMNS16}.
\end{proof}

We fix once and for all a basepoint $o\in\mb{H}^3$. 

Let $Q:=\mb{H}^3/\theta(\pi_1(\Sigma))$. Let $y\in Q$ be the projection of $o$ to $Q$. By work of Thurston \cite[Chapter 8]{ThNotes} and Bonahon~\cite{Bo86}, the manifold $Q$ is diffeomorphic to $\Sigma\times\mb{R}$. Let $f:\Sigma\to Q$ be a marking for $Q$ inducing $\theta$ at the level of the holonomy. We now compute the end invariants of $Q$ in terms of the reference marking $f:\Sigma\to Q$. For a comprehensive discussion of ends and end invariants we refer to Section 2 of \cite{M10}.

We already know that $Q$, has rank one cusps at the curves of $P$ since, by construction as a limit of the $\theta_n$, $\theta$ maps every curve in $P$ to a parabolic isometry.

We now compute the other end invariants. In order to do so, we discuss the geometric convergence of the sequence $N_n$. 

Consider the discrete groups $\Gamma_n:=\theta_n(\pi_1(\Sigma))$ and the hyperbolic 3-manifolds $N_n=\mb{H}^3/\Gamma_n$. Let $x_n$ be the projection of $o$ to $N_n$. 

As $\inf_n\{{\rm inj}_{x_n}(N_n)\}>0$, we have that, up to subsequences, the sequence of groups $\Gamma_n$ converges in the Chabauty topology to a discrete group $\Gamma$ (see Chapter E of \cite{BP92}). This means that 
\begin{enumerate}
\item{For every $\gamma\in\Gamma$ there exists a sequence $\gamma_n\in\Gamma_n$ such that $\gamma_n\to\gamma$.}
\item{If $\gamma_{n_j}\in\Gamma_{n_j}$ converges to $\gamma$ in ${\rm PSL}_2(\mb{C})$, then $\gamma\in\Gamma$.}
\end{enumerate}

In particular, for our sequence $\Gamma_n$ we have $\theta(\pi_1(\Sigma))<\Gamma$ by the second property. Consider $N:=\mb{H}^3/\Gamma$ and let $x\in N$ be the projection of $o\in\mb{H}^3$. As $\theta(\pi_1(\Sigma))<\Gamma$, we have a covering projection $\pi:Q\to N$.

It is well-known that convergence of discrete groups $\Gamma_n\to\Gamma$ in the Chabauty topology implies convergence in the pointed geometric topology of the associated hyperbolic manifolds $(N_n,x_n)\to(N,x)$ (see Theorem E.1.13 of \cite{BP92}). This means that given $R$ and $L>1$, for every $n$ large enough there exists a smooth map $\phi_n:B(x,R)\subset N\to B(x_n,R)\subset N_n$ which is $L$-bilipschitz and maps $\phi_n(x)=x_n$. 

Suppose that $f(\Sigma)\subset B(y,R)\subset Q$ so that $\pi f(\Sigma)\subset B(x,R)\subset N$.

{\bf Fact}. We can assume that $\phi_n\pi f$ induces $\theta_n$ at the level of the holonomy. 

In order to explain this, we briefly recall the construction of the maps $\phi_n$ as in the proof of Theorem E.1.13 in \cite{BP92}: Consider $B(o,R)\subset\mb{H}^3$. For every $\sigma>0$, define the (finite) set
\[
\Gamma(R+\sigma):=\{\gamma\in\Gamma\left|\gamma B(o,R+\sigma)\cap B(o,R+\sigma)\neq\emptyset\right.\}=\{\gamma_1,\cdots,\gamma_t\}.
\]
Let $\{\alpha_1,\cdots,\alpha_m\}$ be a set of generators for $\pi_1(\Sigma)$. By our assumption $\pi f(\Sigma)\subset B(x,R)$, we have $\{\theta(\alpha_1),\cdots,\theta(\alpha_m)\}\subset\Gamma(R)$. By convergence in the Chabauty topology, for every $\gamma_j\in\Gamma(R)$ we can find $\gamma_j^n\in\Gamma_n$ such that $\gamma_j^n\to\gamma_j$ as $n\to\infty$. Notice that, by algebraic convergence $\theta_n\to\theta$, for each $\theta(\alpha_j)\in\Gamma(R)$ we can choose $\theta_n(\alpha_j)$ as an approximating element.

Lemma E.1.16 of \cite{BP92} shows that, up to slightly increasing $R$ and choosing $\sigma$ appropriately small, for every $n$ large enough we have 
\[
\Gamma_n(R+\sigma)=\{\gamma\in\Gamma_n\left|\gamma B(o,R+\sigma)\cap B(o,R+\sigma)\neq\emptyset\right.\}=\{\gamma_1^n,\cdots,\gamma_t^n\}.
\] 

In \cite{BP92} it is then explained how to construct a sequence of smooth embeddings ${\hat \phi}_n:B(o,R)\to B(o,R+\sigma)$ such that
\[
{\hat \phi}_n(\gamma_j p)=\gamma_j^n{\hat \phi}_n(p)
\] 
and ${\hat \phi}_n$ converges to the natural inclusion in the smooth topology.

Such a map ${\hat \phi}_n$ descends to the desired approximating map $\phi_n:B(x,R)\subset N\to B(x_n,R)\subset N_n$. We are now ready to conclude: Let ${\hat f}:{\hat \Sigma}\to\mb{H}^3$ be the lift of the marking $f$ to the universal covers. We have
\[
{\hat \phi}_n{\hat f}(\alpha_i q)={\hat \phi}_n(\theta(\alpha_j){\hat f}(q))=\theta_n(\alpha_j){\hat \phi}_n({\hat f}(q))
\]
for every generator $\alpha_j$. This shows that $\phi_n\pi f$ induces $\theta_n$ at the level of the holonomy.

By the fact, we have that $\phi_n\pi f$ is homotopic to $f_n\psi^n:\Sigma\to N_n$, where $f_n:\Sigma\to N_n$ is the marking inducing $\rho_n\circ(\pi_1(\partial H_g)\to\pi_1(H_g))$ at the level of the holonomy: Both $f_n\psi^n$ and $\phi_n\pi f$ induce the same map $\theta_n$ at the level of the holonomy. Hence, they induce the same map at the level of fundamental groups. As both source and target are aspherical manifolds, the maps are freely homotopic.  

We use this information to prove the following:

\begin{lem}
\label{end invariants}
The lamination $\lambda$ is not realized in $Q$.
\end{lem}

\begin{proof}
Suppose that $\lambda$ is realized. Since $\psi^{-n}\delta$ converges to $\lambda$, we can also realize $\psi^{-n}\delta$ for every $n$ large enough in a neighborhood of $\lambda$. Thus for every $n$ large enough, the curve $f(\psi^{-n}\delta)$ is homotopic to its geodesic representative $\beta_n$ in $Q$ within $B(y,R)$ for some large $R$. Let $\phi_n:B(x,R)\subset N\to N_n$ be the $L$-bilipschitz approximating map provided by geometric convergence. As $L$ can be chosen arbitrarily close to 1, the curve $\phi_n\pi(\beta_n)$ is not null-homotopic in $N_n$. However, $\pi(\beta_n)$ is homotopic within $B(x,R)$ to $\pi f(\psi^{-n}\delta)$ and $\phi_n\pi f$ is homotopic to $f_n\psi^n$.
%where $f_n:\Sigma\to N_n$ is the natural marking of $N_n$ induced by the inclusion of the boundary.
Hence, $\phi_n\pi f(\psi^{-n}\delta)$ is homotopic to $f_n(\delta)$ which is null-homotopic in $N_n$. This provides a contradiction and finishes the proof.
\end{proof}

Since $P$ and $\lambda$ together fill the surface $\Sigma$, there is no room for other end invariants. By the solution of the Ending Lamination Conjecture by Minsky \cite{M10} and Brock, Canary, and Minsky \cite{BrockCanaryMinsky:ELC2}, we conclude that $Q$ is the unique hyperbolic structure on $\Sigma\times(-\infty,0)$ with end invariants $(\lambda,P)$. We will also denote such a structure by $Q=Q(\lambda,P)$.

We now show that $Q$ coincides with $N$.

\begin{lem}
\label{geometric convergence}
The covering map $\pi:Q\to N$ is trivial.
\end{lem}

%Our argument here differs from the one of \cite{NS09} and is more similar to the argument of Lemma 4.4 in \cite{HV}.

\begin{proof}
We prove that every component $E$ of $Q-\mc{CC}(Q)$ embeds in $N$ under the covering map $\pi$ and distinct components $E,E'$ have disjoint images under $\pi$. Then we argue that this is enough to conclude that $\pi$ is trivial.

By the structure of the relative ends of $Q$, the components of $Q-\mc{CC}(Q)$ correspond 1-to-1 to the components $W$ of $\Sigma-P$.

They are all of the following form: Consider a component $W\subset\Sigma-P$. Notice that $\theta_n(\pi_1(W))=\rho_n(\pi_1(\psi^nW))$ leaves invariant a half space $H_n\subset\mb{H}^3$ which is a lift to $\mb{H}^3$ of the the component of $N_n-\mc{CC}(N_n)$ corresponding to $\psi^nW$. Moreover, every element $\rho_n(\gamma)$ with $\gamma\in\pi_1(H_g)-\pi_1(\psi^nW)$ moves $H_n$ off itself, that is, $H_n$ is {\em precisely invariant} under $\theta_n(\pi_1(W))=\rho_n(\pi_1(\psi^nW))$ in $\Gamma_n$. 

By assumption, the sequence of fuchsian representations $\theta_n:\pi_1(W)\to{\rm PSL}_2(\mb{C})$ of the pair of pants $W$ converges to $\theta:\pi_1(W)\to{\rm PSL}_2(\mb{C})$. It is immediate to check that the limit is also a fuchsian representation, that it preserves a half space $H$ which is the limit of the half spaces $H_n$, and that $\theta_n(\pi_1(W))$ converges to $\theta(\pi_1(W))$ in the Chabauty topology. 

The quotient $E=H/\theta(\pi_1(W))$ corresponds to the $W$-component of $Q-\mc{CC}(Q)$.

{\bf Claim}. The restriction of $\pi$ to a component $E$ of $Q-\mc{CC}(Q)$ is a homeomorphism onto the image.

{\em Proof of claim}. Suppose that there exists $\alpha\in\Gamma-\theta(\pi_1(W))$ such that $\alpha H\cap H\neq\emptyset$. Let $\alpha_n\in\Gamma_n$ be a sequence that approximates $\alpha$. Since $H_n\to H$ and $\alpha_n\to\alpha$ we deduce that $\alpha_n H_n\cap H_n\neq\emptyset$ for $n$ large enough. Therefore, $\alpha_n\in\theta_n(\pi_1(W))$ as every element in $\Gamma_n-\theta_n(\pi_1(W))$ moves $H_n$ off itself. This implies $\alpha\in\theta(\pi_1(W))$ as the peripheral subgroups $\theta_n(\pi_1(W))$ are converging to $\theta(\pi_1(W))$ in the Chabauty topology. As a consequence, the restriction of the projection $\pi$ to $E=H/\theta(\pi_1(W))$ is a homeomorphism onto the image.\qed

{\bf Claim}. If $E,E'$ are distinct components of $Q-\mc{CC}(Q)$, then $\pi(E)\cap\pi(E')=\emptyset$.

{\em Proof of claim}. Suppose that $\pi(E)\cap\pi(E')\neq\emptyset$. Represent the components as $E=H/\theta(\pi_1(W))$ and $E'=H'/\theta(\pi_1(W'))$ where $H,H'\subset\mb{H}^3$ are precisely invariant half-spaces under $\theta(\pi_1(W)),\theta(\pi_1(W'))$ and $W,W'$ are distinct components of $\Sigma-P$. 

As $\pi(E)\cap\pi(E')\neq\emptyset$, there exists $\alpha\in\Gamma$ such that $\alpha H'\cap H\neq\emptyset$. In this case, we must have $\alpha H'=H$: If $\alpha H'\neq H$, the circle at infinity bounding $\alpha\partial H'$, which is contained in the limit set of $\theta(\pi_1(\Sigma))$, would intersect the round disk at infinity bounding $H$ and such a disk is contained in the domain of discontinuity of $\theta(\pi_1(\Sigma))$. 

Therefore, as $\alpha H'=H$, we have $\theta(\pi_1(W))=\alpha\theta(\pi_1(W'))\alpha^{-1}$. Since $\Gamma_n$ converges in the Chabauty topology to $\Gamma$, this implies that $\theta_n(\pi_1(W))=\rho_n(\pi_1(\psi^nW))$ and $\theta_n(\pi_1(W'))=\rho_n(\pi_1(\psi^n W'))$ are also conjugate in $\Gamma_n$ for $n$ large enough. However, this can happen only if $W=W'$.\qed

We are now ready to prove that the projection $\pi:Q\to N$ is an isometry: Consider $\pi(E)\subset N$. If $\pi$ is not an isometry, then, by the first claim, there is a component of the pre-image of $\pi(E)$ not contained in $E$. Such component must intersect a component $E'$ of $Q-\mc{CC}(Q)$ different from $E$: For example, because it contains points with arbitrarily large injectivity radius and there is a uniform upper bound on the injectivity radius at points in $\mc{CC}(Q)$ (see Canary \cite{C96}). But distinct components $E,E'$ map to disjoint sets by the second claim.  
\end{proof}

Thus, we have shown that the sequence of $(N_n,x_n)$ converges geometrically to $(Q,y)$. The manifold $Q(\lambda,P)$ looks very similar to the infinite cyclic covering ${\hat T}_\psi$ of the mapping torus $T_\psi$. We do not need to make this assertion very precise as, for our purposes, it is sufficient to prove the following:

\begin{lem}
\label{limit unlinked}
There exists an identification of $Q$ with $\Sigma\times(-\infty,0)$ that induces $\theta:\pi_1(\Sigma)\to{\rm PSL}_2(\mb{C})$ at the level of the holonomy and such that $\psi^{-1}P\times\{-1\}$ for $j<0$ is a geodesic link of length at most $2\eta$.
\end{lem}

%Lemma \ref{limit unlinked} can be certainly deduced from the structure of the model manifold as described in \cite{M10} and \cite{BrockCanaryMinsky:ELC2}. For simplicity and completeness, we include a short proof that uses additional technology, but we stress the fact that it is possible to obtain the same result using only the model manifold description.

\begin{proof}[Proof of Lemma \ref{limit unlinked}]
As $\psi$ is a pseudo-Anosov mapping class with a short pants decomposition $P$ of length at most $\eta$  (by our choice of $\eta$ at the beginning of the proof), the mapping torus 
\[
T_{\psi}=\Sigma\times[0,1]/(x,0)\sim(\psi(x),1)
\]
admits a hyperbolic metric for which the curves $P\times\{0\}$ on the fiber $\Sigma\times\{0\}$ form a geodesic link of length at most $\eta$. Let ${\hat T}_\psi=\Sigma\times\mb{R}$ be the infinite cyclic covering of $T_\psi$. Concretely, $T_\psi={\hat T}_\psi/\Psi$ where
\[
\Psi:{\hat T}_\psi\to{\hat T}_\psi
\]
is the isometry $\Psi(x,t):=(\psi(x),t+1)$ generating the deck group of the covering.

We now show that $Q=Q(\lambda,P)$ is geometrically very close to the region $\Sigma\times(-\infty,0]-P\times\{0\}\subset{\hat T}_\psi$. As the curves in $\psi^{-1}P\times\{-1\}\subset\Sigma\times\{-1\}$ form a geodesic link of length at most $2\eta$, this will imply that the same property holds for $Q$.

Consider the cyclic coverings 
\[
T_{\psi^n}=\Sigma\times[0,n]/(x,0)\sim(\psi^n(x),1)={\hat T}_\psi/\Psi^n
\]
of $T_\psi$. Each of the pants decompositions $\psi^jP\times\{j\}$ on the level surfaces $\Sigma\times\{j\}$ is a geodesic link of length at most $\eta$.

By our choice of $\eta$ at the beginning of the proof, there is a (unique) complete finite volume hyperbolic structure on $T_{\psi^n}-P\times\{0\}$, which we denote by $T'_{\psi^n}$, and by Theorem \ref{drilling thm} there is a 2-bilipschitz map 
\[
\Phi_n:T_{\psi^n}-\bigsqcup_{\gamma\subset P\times\{0\}}{\mb{T}_\ep(\gamma)}\to T'_{\psi^n}-\bigsqcup_{\gamma\subset P\times\{0\}}{\mb{T}_\ep(\gamma)}.
\] 
%in the isotopy class of the inclusion $T'_{\psi^n}\subset T_{\psi^n}$. 

%Let us fix a level surface $S\subset T_\psi$ that avoids the $\ep$-Margulis neighborhoods of the geodesics in $P\times\{0\}$. Let $S_n\subset T_{\psi^n}$ be the unique lift of $S$ to $\Sigma\times[n-1,n]\subset \T_{\psi^n}$. Let us also fix coherently basepoints $x\in S$ and $x_n\in S_n$.  

Let us fix a basepoint $x\in\Sigma\times[0,1]\subset{\hat T}_\psi$ in the $\ep$-thick part. Let $x_n\in\Sigma\times[n-1,n]$ be the projection of $\Psi^n(x)$ to $T_{\psi^n}$. Consider the sequence of pointed manifolds $(T_{\psi}^n,x_n)$ and $(T_{\psi^n}',x_n':=\Phi_n(x_n))$. 

By construction, the sequence $(T_{\psi^n},x_n)$ converges in the pointed geometric topology to the infinite cyclic covering ${\hat T}_\psi$ of $T_\psi$: For example, we can choose the approximating maps to be the restrictions of the covering map ${\hat T}_\psi\to T_{\psi^n}$ to larger and larger metric balls around $x$. 

Notice that, as $\Phi_n$ is 2-bilipschitz, the injectivity radius of $T_{\psi^n}'$ at $x_n'=\Phi_n(x_n)$ is bounded from below by $\ep/2$. Therefore, the sequence $(T_{\psi^n}',x_n')$ also converges in the pointed geometric topology to a hyperbolic manifold $(T',y)$. Notice that each $\ep$-Margulis neighborhood of a rank 2 cusp $\mb{T}_\ep(\gamma)\subset T'_{\psi^n}$ for $\gamma\subset P\times\{0\}$ converges to a rank two cusp $\mb{T}_\ep(\gamma)\subset T'$.

Furthermore, by Ascoli-Arzelà, the sequence of 2-bilipschitz maps $\Phi_n$ converge to a 2-bilipschitz map
\[
\Phi:{\hat T}_\psi-\bigsqcup_{\gamma\subset P\times\{0\}}{\mb{T}_\ep(\gamma)}\to T'-\bigsqcup_{\gamma\subset P\times\{0\}}{\mb{T}_\ep(\gamma)}.
\]

The curves in $\psi^{-1}P\times\{-1\}\subset{\hat T}_\psi$ form a geodesic link of length at most $2\eta$ contained on some level surface $S\subset{\hat T}_\psi-\bigsqcup_{\gamma\subset P\times\{0\}}{\mb{T}_\ep(\gamma)}$. Applying Lemma \ref{isotopy}, we deduce that the geodesic representatives of the curves $\Phi(\psi^{-1}P\times\{-1\})\subset\Phi(S)$ are isotopic to $\Phi(\psi^{-1}P\times\{-1\})$ within their standard $\ep$-Margulis neighborhoods. We lift such isotopy to the $\Phi(S)$-covering $Q'\to T'$.

We now prove that $Q'$ is isometric to $Q=Q(\lambda,P)$ with an isometry in the isotopy class of the identity with respect to the natural marking on $Q'$ and $Q$. This suffices to finish the proof. 
%as it shows that the geodesic representatives of $\Phi(\psi^{-1}\times\{-1\})$ in $Q'$ are isotopic to $\Phi(\psi^{-1}\times\{-1\})$

The covering $Q'$ is a hyperbolic structure on $\Sigma\times\mb{R}$. By construction, the curves in $\Phi(P\times\{0\})$ represent rank one cusps. The restriction of $\Phi$ to $\Sigma\times[-1,-\infty)$ lifts to a 2-bilipschitz embedding to $Q'$. As the end $\Sigma\times[-1,-\infty)$ is simply degenerate with ending lamination $\lambda\subset S$ and $\Phi$ is 2-bilipschitz, the image $\Phi(\Sigma\times[-1,-\infty))$ is a simply degenerate end with ending lamination $\Phi(\lambda)\subset\Phi(S)$. By the solution of the Ending Lamination Conjecture, we conclude that $Q'$ is isometric to $Q=Q(\lambda,P)$.
%Since $\psi^{-n}P\to\lambda$, the manifold $Q(P,\lambda)$ is the algebraic and geometric limit of the sequence $Q_n=Q(P,\psi^{-n}P)$ (see for example Brock, Bromberg, Canary, and Minsky \cite{BBCM}). So, by Lemma \ref{isotopy}, it is enough to show that for every $n$ large enough, the geodesic representative of $\psi^{-1}P$ in $Q_n$ is isotopic to its geodesic realization in $Q_n$. Since $\psi$ is a pseudo-Anosov mapping class with a short pants decomposition, the hyperbolic mapping torus $T_{\psi^n}$ can be described as a long Dehn filling of the drilled mapping torus 
%\[
%DT_{\psi^n}=\mc{CC}(Q_n)/\left(\psi^n:\partial^-\mc{CC}(Q_n)\to\partial^+\mc{CC}(Q_n)\right)\simeq T_{\psi^n}-P\times\{0\}
%\]
%obtained by gluing the boundary components of the boundary of the convex core of $Q_n$ according to $\psi^n$. 
%As $\ell_{T_{\psi^n}}(P)<\eta$, by Theorem \ref{drilling thm}, there is a 2-bilipschitz map 
%\[
%f:T_{\psi^n}-\bigsqcup_{\gamma\subset P\times\{0\}}{\mb{T}_\ep(\gamma)}\to DT_{\psi^n}-\bigsqcup_{\gamma\subset P\times\{0\}}{\mb{T}_\ep(\gamma)}
%\] 
%in the isotopy class of the inclusion $T_{\psi^n}-P\times\{0\}\subset T_{\psi^n}$. Notice that, by construction, in $T_{\psi^n}$ has a cyclic isometry $\Psi:T_{\psi^n}\to T_{\psi^n}$ that can be described as follows: Topologically, $T_{\psi^n}=\Sigma\times[0,n]/(x,0)\sim(\psi^n x,n)$ and $\Psi(x,t):=(\psi x,t+1)$. The curves $\psi^jP\times\{j\}=\Psi^j(P\times\{0\})$ are all geodesic links of length at most $\eta$. The conclusion follows from Lemma \ref{isotopy}.
\end{proof}

Lastly, we exploit the geometric convergence of $N_n$ to $Q$ combined with some topological argument to show that the geodesic representative of $\psi^{n-1}P\times\{1\}$ in $N_n$ is short and isotopic to $\psi^{n-1}P\subset\partial H_g$, thus producing the desired contradiction. 

By Lemma \ref{limit unlinked}, we know that we can identify $Q=\Sigma\times(-\infty,0)$ so that $\psi^{-1}P\times\{-1\}$ is a geodesic link of length at most $2\eta$.

We consider some large product region $\Sigma\times[-T,-1/T]$ that contains the union of the Margulis neighborhoods of the curves in $\psi^{-1}P\times\{-1\}$. Let $\phi_n:\Sigma\times[-T,-1/T]\to N_n$ be a 2-bilipschitz approximating map provided by geometric convergence. As in the above discussion, we can assume that the map $\phi_n$ induces $\theta_n$ at the level of the holonomy. 

By Lemma \ref{isotopy}, $\phi_n(\psi^{-1}P\times\{-1\})$ is isotopic to its geodesic representative. So, in order to conclude the proof it is enough to show the following:

\begin{lem}
\label{marking isotopic}
$\phi_n(\psi^{-1}P\times\{-1\})$ is isotopic to $\psi^{n-1}P\subset\partial H_g$.
\end{lem}

Notice that, as $\phi_n$ induces $\theta_n$ at the level of the holonomy, we already know that $\phi_n(\psi^{-1}P\times\{-1\})$ is homotopic to $\psi^{n-1}P\subset\partial H_g$. What we have to do is to replace being homotopic with being isotopic. The idea of the proof is that $\phi_n(\Sigma\times\{-1\})$ is parallel to $\partial H_g$ in such a way that $\phi_n(\gamma\times\{-1\})$ is parallel to $\psi^n(\gamma)\subset\partial H_g$. 

The proof is a little bit tedious so we summarize the main steps. Let us point out that the main issue is to keep track of the markings up to isotopy rather than up to homotopy for the convergence $N_n\to Q$. This is complicated by the fact that the map $\phi_n:\Sigma\times\{-1\}\to N_n$ is compressible.

%{\bf Step 1}. We isotope $\phi_n(\Sigma\times\{-1\})$ so that $\phi^{-1}(\mb{T}_{2\ep}(\psi^n\gamma))=A_\gamma\times\{-1\}$ for every $\gamma\subset P$ where $A_\gamma\subset\Sigma$ is a tubular neighborhood of $\gamma$.

{\bf Step 1}. We show that $\phi_n(\Sigma\times\{-1\})$ is parallel to $\Sigma=\partial H_g$. 

This immediately implies that $\phi_n:\Sigma\times\{-1\}\to N_n$ is isotopic to the composition $f_nh_n$, where $h_n:\Sigma\to\Sigma$ is a homeomorphism and $f_n:\Sigma\to H_g$ is the natural inclusion. 

{\bf Step 2}. The homeomorphism $h_n$ has the property that $h_n(\gamma)=\psi^n(\gamma)$ for every $\gamma\subset P$.

As a consequence, we have that $DT=h_n\psi^{-n}$ is a homeomorphism of $\Sigma$ that leaves invariant the pants decomposition $P_n=\psi^nP$. Hence, it is a product of Dehn twists around the curves in $P_n$. 

{\bf Step 3}. The homeomorphism $DT$ is isotopic to the identity.

Combining the three steps, we conclude that $\phi_n$ is isotopic to $f_n\psi^n$.

\begin{proof}[Proof of Lemma~\ref{marking isotopic}]
We now prove the steps outlined above.

{\em Proof of Step 1}. The first step comes from geometric convergence and some standard 3-dimensional topology. 

Consider the non-cuspidal part of the convex core $\mc{CC}_0(N_n)\subset N_n$ obtained by removing from the convex core $\mc{CC}(N_n)$ the standard $\ep$-Margulis neighborhoods of the cusps. 

We claim that the boundary $\partial\mc{CC}_0(N_n)$ lies in a uniform neighborhood of the surface $\phi_n(\Sigma\times\{-1\})$. To see this, notice that the boundary $\partial\mc{CC}_0(N_n)$ decomposes as the disjoint union of the $\ep$-non cuspidal parts of the boundary components $\partial\mc{CC}(N_n)$ joined by flat annuli on the boundary of the standard $\ep$-Margulis neighborhoods of the cusps. 

A component $C_W^n$ of $\partial\mc{CC}_0(N_n)$ that comes from $\partial\mc{CC}(N_n)$ is homotopic to $\psi^nW$ for a component $W\subset\Sigma-P$. Moreover, it has the following properties. First, it has uniformly bounded diameter ${\rm diam}(C_W^n)\le D$. Secondly, it contains a (self-intersecting) closed geodesic $\gamma_n^*$ homotopic to $\psi^n\gamma_W$ of uniformly bounded length $\ell(\gamma_n^*)\le L$. Both properties come from the fact that $C_W^n$ is isometric to the $\ep$-thick part of the unique complete finite area hyperbolic metric on $\psi^nW$.

By standard hyperbolic geometry, using the fact that $\phi_n$ is 2-bilipschitz, the distance between $\phi_n(\Sigma\times\{-1\})$ and $C_W^n$ can be bounded in terms of the length of $\gamma_W$ on $\Sigma\times\{-1\}$ and $L$. Thus each component $C_W^n$ lies in a uniform neighborhood of $\phi_n(\Sigma\times\{-1\})$. 

In order to conclude, we just have to consider the distance of the annular components $A_\gamma^n$ of $\partial\mc{CC}_0(N_n)$ from $\phi_n(\Sigma\times\{-1\})$. We already know that the distance of the boundary components of such annuli $A_\gamma^n$ from $\phi_n(\Sigma\times\{-1\})$ is uniformly bounded.  In particular, the distance of the boundary components is also bounded (in terms of ${\rm diam}(\Sigma\times\{-1\})$). As the intrinsic diameter of $A_\gamma^n$ can be bounded in terms of the distance of the boundary components, the claim follows.

%As the annuli lift to horospheres in $\mb{H}^3$, by convexity of distances, we can bound the distance of the distance of any other point in the annulus from 

%For example, one can uniformly bound their distance by bounding the distance of $\phi_n(\Sigma\times\{-1\})$ from the cuspidal part. The latter is bounded uniformly by the length of the geodesic representatives of the curves $\gamma\subset P$ on $\Sigma\times\{-1\}$ and some standard hyperbolic geometry. 

Therefore, if $T$ is large enough, then $\phi_n(\Sigma\times[-T,-1/T])$ contains $\partial\mc{CC}_0(N_n)$. 

We now argue that $\partial\mc{CC}_0(N_n)$ is incompressible in $\phi_n(\Sigma\times[-T,-1/T])$. By standard 3-manifold topology (see \cite[Proposition~3.1 and Corollary~3.2]{W69}), this implies that $\partial\mc{CC}_0(N_n)$ is parallel to the boundary surfaces of the product. Since, by properties of convex cores, $\partial\mc{CC}_0(N_n)$ is also parallel to $\partial H_g$, this will be sufficient to prove the claim of the first step.  

If $\partial\mc{CC}_0(N_n)$ was compressible, then one of the following two cases happen. Either, after some compressions, we find a closed incompressible surface in $\phi_n(\Sigma\times[-T,-1/T])$ of genus strictly smaller than the genus of $\Sigma$ or, we have that the entire convex core $\mc{CC}_0(N_n)$ is contained in $\phi_n(\Sigma\times[-T,-1/T])$. The first case is ruled out by the classification of incompressible surfaces in products $\Sigma\times[0,1]$. The second case is ruled out by the fact that the boundary $\phi_n(\Sigma\times\{-T\})$ is contained in $\mc{CC}(N_n)$. Thus, we must have that $\partial\mc{CC}_0(N_n)$ is incompressible in $\phi_n(\Sigma\times[-T,-1/T])$.\qed

{\em Proof of Step 2}. Let us describe the homeomorphism $h_n$.

Consider the non-cuspidal part of the convex core $\mc{CC}_0(N_n)$. By the structure of the ends of hyperbolic manifolds, there is a homeomorphism $\Sigma\times(0,1)\to N_n-\mc{CC}_0(N_n)$ with the following properties: The inclusion $\Sigma\times\{t\}\to N_n$ is isotopic to $f_n$, the standard marking induced by the inclusion $\Sigma=\partial H_g\subset H_g$. Each cusp $\mb{T}_\ep(\psi^n\gamma)$ corresponds to $A_\gamma\times(0,1)$ where $A_\gamma$ is an annulus around $\psi^n(\gamma)$. Let ${\rm pr}_1:\Sigma\times(0,1)\to\Sigma$ be the projection of the first factor. 

As $\phi_n(\Sigma\times\{-1/T\})$ is contained in $N_n-\mc{CC}(N_n)$ and is parallel to $\Sigma\times\{0\}$ the composition ${\rm pr}_1\phi_n:\Sigma\to\Sigma$ is a degree one homotopy equivalence and $f_n{\rm pr}_1\phi_n$ is homotopic to $\phi_n$. Let $h_n$ be an orientation preserving homeomorphism homotopic to ${\rm pr}_1\phi_n:\Sigma\to\Sigma$. 

In order to compute the homotopy class of $h_n(\gamma)$ for $\gamma\subset P$ we can proceed as follows: In the product $\Sigma\times[-T,-1/T]$ we can isotope the inclusion $\iota_{-1}:\Sigma\to\Sigma\times\{-1\}$ to $\iota:\Sigma\to\Sigma\times[-T,-1/T]$ so that it intersects each of the cusps $\mb{T}_{\ep/2}(\gamma)$ in an annulus $\iota(U_\gamma)$ around $\gamma$. As $\phi_n$ is 2-bilipschitz and $\phi_n\iota$ is homotopic to $f_n\psi^n$, we have that $\phi_n\iota(U_\gamma)\subset\mb{T}_\ep(\psi^n(\gamma))$. The homotopy class of $h_n(\gamma)$ is the homotopy class of the curve ${\rm pr}_1(\phi_n\iota(\gamma))$ which is a primitive essential curve contained in $A_\gamma$. Thus $h_n(\gamma)\simeq\psi^n(\gamma)$.\qed

%Notice that $f_n(h_n(\gamma))$ is homotopic to $\phi_n(\gamma)$ which is homotopic to $f_n(\psi^n(\gamma))$, the core of the cusp $\mb{T}_{2\ep}(\psi^n(\gamma))$. Thus, we found a homotopy $(A,\partial A)\to(H_g,\partial H_g)$ between $h_n(\gamma)$ and $\psi^n(\gamma)$. By the pared acylindrical assumptions, the homotopy can be deformed into $\partial H_g$ so that we can conclude $\psi^n(\gamma)=h_n(\gamma)$.\qed

{\em Proof of Step 3}. Observe that $DT$ has the following property: Let $\delta\in\mc{D}$ be a disk. We have that $f_n(DT(\delta))=f_nh_n(\psi^{-n}\delta)$ is homotopic to $\phi_n(\psi^{-n}\delta)$ which, in turn, is homotopic to $f_n(\psi^n\psi^{-n}\delta)=f_n(\delta)$. Therefore $DT(\delta)$ is again a compressible curve. By the Loop Theorem, we conclude $DT(\delta)\in\mc{D}$. In other words $DT$ preserves $\mc{D}$. However, by assumption, no component of $P_n$ is contained in $\mc{D}$ and this implies that any non-trivial product of Dehn twists around curves in $P_n$ cannot preserve $\mc{D}$.\qed

Thus, we proved Lemma \ref{marking isotopic}.
\end{proof}

This concludes the proof of Proposition~\ref{approximation4}.\end{proof}

\subsection{Sketch of proof for Proposition~\ref{approximation5}}
The proof of Proposition~\ref{approximation5} follows exactly the same strategy as the proof of Proposition~\ref{approximation4}, so we only discuss the small modifications needed.

First, we prove that, under the assumptions of Proposition \ref{approximation5}, namely 
\[
d_{\mc{C}}(P_{-n},R_n)\ge d_{\mc{C}}(P_{-n},P)+d_{\mc{C}}(P,R)+d_{\mc{C}}(R,R_n)-K.
\]
and $n$ large enough, the pants decompositions $P_{-n+1}\times\{1\}$ and $R_{n-1}\times\{2\}$ are (individually) isotopic to their geodesic realizations in $Q(P_{-n},R_n)$ which are links of length at most $2\eta$. Then, we argue that the two isotopies can be realized simultaneously.

%Note that, for the sake of clarity in the explanation, we conveniently changed the labeling of the $P_j$ and $R_i$ by shifting back the indices by $n$, this does not affect anything.

{\em The case of $R_{n-1}\times\{2\}$}. The above argument by contradiction still works if we replace $N_n=H(P_n)$ with $Q_n=Q(P_{-n},R_n)$: The sequence of manifolds $\phi^{-n}Q_n$ with holonomies $\theta_n=\rho_{Q_n}\circ\psi^n$ converges algebraically and geometrically to $Q(\lambda_\phi,R)$ with holonomy $\theta=\rho_{Q(\lambda_\phi,R)}$, where $\lambda_\phi$ is the repelling lamination of $\phi$. Let us point out that convergence here is more classical and we refer the reader to Theorem 1.1 of \cite{BBCM13} that covers a much more general setup. 

The manifold $Q(\lambda_\phi,R)$ is geometrically very close to $\Sigma\times(-\infty,0]\subset{\hat T}_\phi$. 

Again, as above (Lemma \ref{marking isotopic}), convergence implies that $R_{n-1}\times\{2\}$ is isotopic to its geodesic representatives in $Q_n$ for every large enough and such geodesic representatives have length at most $2\eta$ (thus contradicting the initial assumption). Let us notice that Lemma \ref{marking isotopic}, modulo some standard 3-dimensional topology, is now straightforward because of incompressibility of the marking. 

%The case of $P_{-n+1}\times\{1\}$: The above argument by contradiction works the as well if we replace $N_n=H(P_n)$ with $Q_n=Q(P_{-n},R_n)$. The sequence of holonomies $\theta_n=\rho_{Q_n}\circ\psi^{-n}$ converges algebraically to $\theta$ and the sequence $Q_n$ geometrically to $Q(P,\lambda')$, where $\lambda'$ is the stable lamination of $\psi$. Let us point out that convergence here is more classical and we refer the reader to Theorem 1.1 of \cite{BBCM13} that covers a much more general setup. The manifold $Q(P,\lambda')$ is geometrically very close to $\Sigma\times[0,\infty)\subset{\hat T}_\psi$ (same proof). Again, convergence implies that $P_{-n+1}\times\{1\}$ is isotopic to its geodesic representatives in $Q_n$ for every large enough and such geodesic representatives have length at most $2\eta$. 

{\em The case of $P_{-n+1}\times\{1\}$}. The proof is the same where we replace $\phi^{-n}Q_n$ with $\psi^nQ_n$ and $Q(\lambda_\phi,R)$ with $Q(P,\lambda_\psi)$.

At this point we know that, under the assumptions of Proposition \ref{approximation5}, the pants decompositions $P_{-n+1}\times\{1\}$ and $R_{n-1}\times\{2\}$ are individually isotopic to their geodesic realizations in $Q(P_{-n},R_n)$ which are links of length at most $2\eta$. The only property that remains to be checked is that the two isotopies can be realized simultaneously.

This can be done as follows: Let $\beta_P,\beta_R$ be the geodesic realizations of $P_{-n+1},R_{n-1}$ in $Q(P_{-n},R_n)$. By Theorem \ref{blocksep}, the distance between two $\ep$-Margulis neighborhoods $\mb{T}_\ep(\alpha),\mb{T}_\ep(\beta)$ is bounded by $d_{\mc{C}}(\alpha,\beta)$.

This implies that the $\ep$-Margulis neighborhoods around the curves of $\beta_P$ have (electric) distance from $\partial_+\mc{CC}_0(Q)$ uniformly bounded in terms of 
\[
D_P:=\max\{d_{\mc{C}}(\gamma,\gamma')\left|\gamma\subset P_{-n},\gamma'\subset P_{-n+1}\right.\}.
\]

Similarly, the (electric) distance of the $\ep$-Margulis tubes around the curves in $\beta_R$ from $\partial_-\mc{CC}_0(Q)$ is bounded in terms of 
\[
D_R:=\max\{d_{\mc{C}}(\gamma,\gamma')\left|\gamma\subset R_n,\gamma'\subset R_{n-1}\right.\}.
\]

On the other hand the (electric) distance between $\partial_+\mc{CC}_0(Q)$ and $\partial_-\mc{CC}_0(Q)$ is bounded from below by $d_\mc{C}(P_{-n},R_n)$. If $d_{\mc{C}}(P_{-n},R_n)$ is much larger than $D_P+D_R$, then we can isotope $\beta_P$ into $\partial_+\mc{CC}_0(Q)$ and $\beta_R$ into $\partial_-\mc{CC}_0(Q)$ simultaneously. This comes from the following:

{\bf Claim}. It is possible to find a boundary parallel surface $S\subset\mc{CC}_0(Q)$ that separates $\beta_P$ from $\beta_R$.

{\em Sketch of proof}. First, as the distance between $\partial_+\mc{CC}_0(Q)$ and $\partial_-\mc{CC}_0(Q)$ is very large, we can find a pleated surface that separates $\beta_P$ from $\beta_R$. This can be done as follows: By the Filling Theorem of Canary~\cite{C96} and Thurston~\cite{ThNotes}, one can find a pleated surface $S'$ passing uniformly near a point in $\mc{CC}_0(Q)$ whose distance from $\partial\mc{CC}_0(Q)$ is equal to $d(\partial_+\mc{CC}_0(Q),\partial_-\mc{CC}_0(Q))/2$. Such a pleated surface has uniformly bounded intrinsic (electric) diameter, in particular, it is disjoint from the $D_P,D_R$-(electric) neighborhoods $U_+,U_-$ of $\partial_+\mc{CC}_0(Q),\partial_-\mc{CC}_0(Q)$ provided that $d(\partial_+\mc{CC}_0(Q),\partial_-\mc{CC}_0(Q))\asymp d_{\mc{C}}(P_{-n},R_n)$ is large enough. As $\mc{CC}_0(Q)\simeq\Sigma\times\{0\}$ and the pleated surface $S'$ is homotopic to $\Sigma\times\{0\}$, we have that $S'$ separates such neighborhoods $U_+,U_-$. Using standard techniques, we can find in a small neighborhood of $S'$, still disjoint from $U_+,U_-$, an embedded surface $S\subset\mc{CC}_0(Q)$ homotopic to $S'$. The fact that $S$ is boundary parallel comes from the classification of incompressible surfaces in $\mc{CC}_0(Q)\simeq\Sigma\times[0,1]$.\qed

% (it can be chosen to be in a uniform neighborhood of pleated surface that separates $\beta_P$ from $\beta_R$, such a pleated surfacehas distance from $\partial\mc{CC}_0(Q)$ roughly $\partial_+\mc{CC}_0(Q)$ and $\partial_-\mc{CC}_0(Q))

So, in order to conclude, it is sufficient to observe that
\[
d_{\mc{C}}(P_{-n},R_n)\gtrsim d_{\mc{C}}(P_{-n},P)+d_{\mc{C}}(R,R_n)
\]
and the right hand side, being comparable to $n$ times the minimal translation distance of $\psi,\phi$ on the curve graph, can be made as large as we want provided that $n$ is large.\qed

%have distance uniformly bounded from the Margulis tubes  Let $\Sigma_P,\Sigma_R$ be surfaces containing the geodesic realizations $\beta_P,\beta_R$ of $P,R$ and such that $\Sigma_P,\Sigma_R$ have distance at most $D$ from $\partial\mc{CC}(Q)$ and intrinsic diameter at most $D$. Assuming that $d_{\mc{C}}(\psi^{-1}P,\phi R)$ is sufficiently large, which is the case if $n$ is very large, then the distance between $\partial^+\mc{CC}(Q)$ and $\partial^-\mc{CC}(Q)$ is much larger than $4D$ so that $\Sigma_P,\Sigma_R$ are disjoint with $\Sigma_P$ closer to $\partial^+\mc{CC}(Q)$ than $\Sigma_R$. 

Let us also point out that it is not hard to deduce Proposition~\ref{approximation5} directly from the structure of the model manifold of~\cite{M10} and~\cite{BrockCanaryMinsky:ELC2}.

\part{Geometric properties of random 3-manifolds}\label{part:random&appl}
\section{The proof of Theorems 1 and 2}
\label{proof1}

In this section we prove a precise form of Theorem~\ref{main1} about the structure of random 3-manifolds. Before we state the theorem, we recall some background and set some notation regarding random walks.

\subsection{Random walks}
We start by recalling some background material on random walks on the mapping class group. We crucially consider only random walks driven by probability measures $\mu$ whose support $S$ is a finite symmetric generating set for the {\em entire} mapping class group.

\begin{dfn}[Random Walk]
Let $(s_n)_{n\in\mb{N}}$ be a sequence of independent random variables with values in $S$ and distribution $\mu$. The {\em $n$-th step of the random walk} is the random variable $f_n:=s_1\dots s_n$. We denote by $\mb{P}_n$ its distribution. The {\em random walk driven by $\mu$} is the process $(f_n=s_1\dots s_n)_{n\in\mb{N}}\in{\rm Mod}(\Sigma)^{\mb{N}}$. It has a distribution which we denote by $\mb{P}$.
\end{dfn} 

The mapping class group acts on Teichmüller space $\text{\rm Mod}(\Sigma)\curvearrowright\T$. If we fix a base point $o\in\T$ we can associate to every random walk $(f_n)_{n\in\mb{N}}$ an orbit $\{f_no\}_{n\in\mb{N}}\subset\T$.

We choose $o\in\T$ with the following property:

\begin{center}
\label{basepoint}
\begin{minipage}{.8\linewidth}
{\bf Standing assumption}: The base point $o\in\T$ is chosen such that its projection to the curve graph lies on the disk set, $\Upsilon(o)\in\mc{D}$.
\end{minipage}
\end{center}

It is a standard consequence of the subadditive ergodic theorem that there exists a constant $L\ge 0$, called the \emph{drift} of the random walk on Teichmüller space, such that for $\mb{P}$-almost every sample path $(f_n)_{n\in\mb{N}}$ we have
\[
\frac{d_\T(o,f_no)}{n}\stackrel{n\to\infty}{\longrightarrow}L.
\]
In general, the drift can be 0. However, it has been established by Kaimanovich and Masur~\cite{KM:Poisson} that, in our case, $L>0$.

\subsection{Statement and discussion}

We are now ready to state the precise version of Theorem \ref{main1}, using the above setup. We say that a sequence of events has asymptotic probability $1$ if the probability of the events goes to $1$ as $n$ tends to infinity.

%\todo{Need to adapt notation and proof to the new setting. It should not require a lot of work.}
\begin{mythm}[2]
\label{main1precise}
Fix $K>1$ and $\ep>0$. Let $o\in\T$ be a fixed basepoint. Denote by $L>0$ the drift of the random walk and by $\tau_n$ the parametrization of the geodesic segment $[o,f_no]$ by arc length. With asymptotic probability $1$, we have the following for some uniform constant $B$: There exist pants decompositions $P_n$ and $R_n$ of $\Sigma$ such that
%that are the shortest pants decompositions of some $S_2^n\in\tau_n[\ep Ln,2\ep Ln]$ and $S_3^n\in\tau_n[(1-2\ep)Ln,(1-\ep)Ln]$ respectively, that is $P_j^n=\Upsilon(S_j^n)$, such that 
\begin{enumerate}[(a)]
\item{The manifold $M_{f_n}$ admits a hyperbolic metric such that restriction of the metric to a tubular neighborhood of the Heegaard surface is $K$-bilipschitz to the $\eta_3$-non cuspidal part of the convex core of the maximally cusped structure $Q(P_n,R_n)$.}
%\item{The maximal cusp $Q(P_n,R_n)$ is also $K$-bilipschitz to $Q(S_2^n,S_3^n)$ away from its cusps (as in Proposition~\ref{model}).\todo{Need to replace this.}}
\item{$d_{\mc{C}}(\Upsilon\tau_n[\ep Ln,2\ep Ln],P_n),d_{\mc{C}}(\Upsilon\tau_n[(1-2\ep) Ln,(1-\ep)Ln],R_n)\le B$ where $\tau_n$ is the Teichmüller segment $[o,f_no]$.}
\end{enumerate}
\end{mythm}

In Subsection~\ref{subsec:thm1viathm3} we also describe how to derive Theorem~\ref{main2} from Theorem~\ref{thm:main} and properties of the random walk.

The proof of Theorem~\ref{main1} does not use 3-dimensional hyperbolic geometry anymore. Rather, via Proposition~\ref{pa examples}, we will only have to work with the dynamics of a random walk on Teichmüller space and the curve graph.

The idea is that thanks to the work done in the previous sections, namely Proposition~\ref{model} and Proposition~\ref{pa examples}, we only need to check that the Teichmüller segment $\tau_n=[o,f_no]$ contains four points $o<S_1<S_2<S_3<S_4<f_no$ satisfying the conditions of Proposition~\ref{pa examples}.

The heuristic picture is the following. Consider the curve graph projection $\Upsilon[o,f_no]$ of the segment $[o,f_no]$. The endpoints $\delta=\Upsilon(o)$ and $f_n\delta=\Upsilon(f_no)$ lie on the disk sets $\mc{D}$ and $f\mc{D}$. 
Hyperbolicity of the curve graph, quasi-convexity of the disk sets and the fact that $\Upsilon$ is monotone along geodesics together imply that, if $\mc{D}$ and $f\mc{D}$ are sufficiently far away, then the path $\Upsilon[o,f_no]$ roughly decomposes into three parts: Initially, it fellow travels $\mc{D}$. Then, it follows a shortest geodesic between $\mc{D}$ and $f_n\mc{D}$. Lastly, it fellow travels $f_n\mc{D}$. 

Any subsegment of the middle piece automatically satisfies property (iii).

Properties (i) and (ii) follow, instead, from ergodic properties of the random walk, see below for discussion and references. In particular, we will use that for any pseudo-Anosov $\phi$, the segment $[o,f_no]$ often fellow travels a translate of the axis $l_\phi$ of the pseudo-Anosov. Therefore, we just have to make sure that the two needed long fellow travelings happen on the subsegment that projects to the middle piece of $\Upsilon[o,f_no]$.

We will deduce this combining the aforementioned ergodic properties of random walks with work of Maher \cite{Maher:Heegaard} who proved that, with asymptotic probability 1, the distance between $\mc{D}$ and $f_n\mc{D}$ increases linearly and up to a sublinear error is the distance between the endpoints $\delta$ and $f_n\delta$. Hence, the middle piece in the above description takes up almost all of $\Upsilon[o,f_no]$.

\subsection{Ergodic properties of random walks}
We can now state the ergodic property of random walks that we need. It is inspired by \cite[Proposition 6.9]{BGH16}. In fact, we believe that the following statement can be extracted from its proof, with the exception, perhaps, of the logarithmic size of the fellow traveling. We include a complete proof of the precise form that we need.

\begin{thm}
\label{randergodic1}
Let $\phi\in{\rm Mod}(\Sigma)$ be a pseudo-Anosov with axis $l_\phi$ in Teichmüller space, and let $0<a<b<1$. Denote by $L>0$ the Teichmüller drift of the random walk. There exists $\ep_0>0$ such that with asymptotic probability $1$ the following holds. Denote by $\tau_n$ the segment $[o,f_no]$. Then $l_\phi$ has a subsegment of length $\ep_0\log(n)$ one of whose translates uniformly fellow-travels a subsegment of $\tau_n[aLn,bLn]$.
\end{thm}

\begin{proof}
For $g$ in ${\rm Mod}(\Sigma)$, we denote by $\pi^g$ the closest-point projection to $gl_\phi$. 

Such projections have strong contraction properties described in the Contraction Theorem of \cite{M96}. In particular, it is well-known that they imply the following: There exists a constant $D>0$, depending on $\phi$, such that if $d_\T(\pi^g(x),\pi^g(y))\ge D$ then the geodesic $[x,y]$ has a subsegment $[x_1,y_1]$ with $d_\T(x_1,\pi^g(x))$, $d_\T(y_1,\pi^g(y))\le D$. By Theorem 4.2 of \cite{M96}, the segment $[x_1,y_1]$ $\delta$-fellow travels $l_\phi$ for some $\delta$ only depending on $l_\phi$. 

Therefore, in order to conclude, it suffices to prove the following claim (the theorem follows up to moving $a,b$ an arbitrarily small amount and modifying $\ep$).

{\bf Claim:} Given $a,b,\phi$ as in the statement, there exists $\epsilon>0$ such that with asymptotic probability 1, there exists $g$ such that $d_\T(\pi^g(o),\pi^g(f_no))\geq \epsilon \log(n)$ and $d_\T(o,\pi^g(o)), d_\T(o, \pi^g(f_no))\in [aLn,bLn]$.

The claim is a consequence of the following properties, which can be found in the existing literature as we explain below: There exist $a'<b'$, $\epsilon\in (0,(b-a)/10)$ and $C>0$ such that the following hold with asymptotic probability 1
\begin{enumerate}[(i)]
 \item{$d_\T(f_jo,\tau_n)\le C\log(n)$ for every $j\le n$.} 
 \item{$d_\T(f_{\lfloor a'n\rfloor}o,o)\in [(a+\ep)Ln,(a+2\ep)Ln)]$ and $d_\T(f_{\lfloor b'n\rfloor}o,o)\in[(b-2\ep)Ln,(b-\ep)Ln)]$.}
 \item{There are $g$ and $\ep>0$ such that $d_\T(\pi^g(f_{\lfloor a'n\rfloor}o),\pi^g(f_{\lfloor b'n\rfloor}o))\ge\ep\log(n)$.} 
 \item{For the same $g$ of (iii), we have $d_\T(\pi^g(f_{\lfloor a'n\rfloor}o), \pi^g(o))\le\ep\log(n)/3$ and $d_\T(\pi^g(f_{\lfloor b'n\rfloor}o),\pi^g(f_no))\le\ep\log(n)/3$.} 
\end{enumerate}

Assuming (i)-(iv) we prove the claim. Afterwards, we give the references to the literature.

\begin{proof}[Proof of the claim]
In the whole proof, all statements and inequalities are meant to hold with asymptotic probability 1. By properties (iii) and (iv), it follows that $d(\pi^g(o),\pi^g(f_no))\ge\ep\log(n)/3$, whence the first part of the claim. We now argue that $\pi^g(o)$ and $\pi^g(f_no)$ have distance within the desired interval from $o$. 

Observe that the geodesic joining $f_{\lfloor a'n\rfloor}o$ to $f_{\lfloor b'n\rfloor}o$ fellow travels $gl_\phi$ along the subsegment connecting $x_n:=\pi^g(f_{\lfloor a'n\rfloor}o)$ to $y_n:=\pi^g(f_{\lfloor a'n\rfloor}o)$, because the projections are very far apart by property (iii). In particular, $x_n$ and $y_n$ are uniformly close to points $p_n$ and $q_n$ on $[f_{\lfloor a'n\rfloor}o,f_{\lfloor b'n\rfloor}o]$ respectively. By property (iv), the projections $x_n$ and $y_n$ are also logarithmically close to $\pi^g(o)$ and $\pi^g(f_no)$. Therefore, we have 
\begin{align*}
d_\T(o,\pi^g(o))=d_\T(o,p_n)+O(\log(n)),\\
d_\T(o,\pi^g(f_no))=d_\T(o,q_n)+O(\log(n)).
\end{align*}
Hence, we can focus on estimating $d_\T(o,p_n)$ and $d_\T(o,q_n)$. In fact, we will provide an estimate on $d_\T(o,p)$ for any point $p\in[f_{\lfloor a'n\rfloor}o,f_{\lfloor b'n\rfloor}o]$.

By the triangle inequality, for any point $p\in[f_{\lfloor a'n\rfloor}o,f_{\lfloor b'n\rfloor}o]$, we have
\begin{align*}
d_\T(o,f_{\lfloor b'n\rfloor}o)-d_\T(f_{\lfloor a'n\rfloor}o,f_{\lfloor b'n\rfloor}o)\le d_\T(o,p),\\
d_\T(o,p)\le d_\T(o,f_{\lfloor a'n\rfloor}o)+d_\T(f_{\lfloor a'n\rfloor}o,f_{\lfloor b'n\rfloor}o).
\end{align*}
We now estimate distances using properties (i) and (ii). 

In view of (ii) and the inequalities above, for our purposes it suffices to show that
\[
d_\T(f_{\lfloor a'n\rfloor}o,f_{\lfloor b'n\rfloor}o)\leq d_\T(o,f_{\lfloor b'n\rfloor}o)-d_\T(o,f_{\lfloor a'n\rfloor}o)+O(\log(n)).
\]
We obtain this inequality as follows: Let $r_n,s_n\in\tau_n$ be provided by property (i), so that $d_\T(r_n,f_{\lfloor a'n\rfloor}o),d_\T(s_n,f_{\lfloor b'n\rfloor}o)=O(\log(n))$. Using that $r_n$ and $s_n$ lie on a geodesic originating at $o$, we have
\begin{align*}
d_\T(f_{\lfloor a'n\rfloor}o,f_{\lfloor b'n\rfloor}o) &\le d_\T(r_n,s_n)+O(\log(n))\\
 &=|d_\T(o,s_n)-d_\T(o,r_n)|+O(\log(n))\\
 &\le |d_\T(o,f_{\lfloor b'n\rfloor}o)-d_\T(o,f_{\lfloor a'n\rfloor}o)|+O(\log(n)).
\end{align*}
Since $\ep<(b-a)/10$, in view of (ii) we can remove the absolute value, and obtain the required estimate.
\end{proof}

Now we provide references for the properties (i)-(iv). 

Property (i) is a corollary of Theorem 10.7~\cite{MS:clt} obtained summing the probabilities that each step of the walk is logarithmically far from $\tau_n$.

Property (ii) follows from positivity of the Teichmüller drift, which implies that for any $\epsilon>0$, with probability going to $1$ as $k$ tends to infinity we have $d_{\mathcal T}(o,f_ko)\in [(L-\epsilon) k,(L+\epsilon) k]$ (see the argument for \cite[Theorem 4.3$(i)$]{KM:Poisson}).

This easily allows us to choose appropriate $a',b'$.

For later purposes, we also note that (i) and the aforementioned property imply the following proposition, which is a version of a theorem of Tiozzo~\cite{T15} and could also be deduced from said theorem.

\begin{thm}
\label{tracking}
 In the setting of the theorem, for any $\epsilon>0$ with asymptotic probability $1$ we have $d_{\mathcal T}(f_mo,\tau_n(L m))\leq \epsilon m$ for all $\epsilon n\leq m \leq (1-\epsilon)n$.
\end{thm}

Properties (iii) and (iv) follow from Theorem 2.3 and Proposition 3.2 of \cite{ST:large_proj}, where a general framework is provided to show that random walk create logarithmically large projections. We explain how: In the terminology of \cite{ST:large_proj} we want to show that 
\[
\left(\mc{S}:=\{gl_\phi\}_{g\in\text{\rm Mod}(\Sigma)},Y_0:=l_\phi,\{\pi^g\}_{g\in\text{\rm Mod}(\Sigma)},\pitchfork\right),
\]
where we define the projections on the group $\pi^g:\text{\rm Mod}(\Sigma)\to gl_\phi$ to be
\[
\pi^g(h):=\pi^g(ho),
\]
forms a {\em projection system} (as in \cite[Definition 2.1]{ST:large_proj}), where $\pitchfork$ is the relation on the translates of having bounded projection to each other, and that the probability measure $\mu$ is {\em admissible} (as in \cite[Definition 2.2]{ST:large_proj}). 

The fact that the 4-tuple is a projection system follows from the contraction property of the projections $\pi^g$ and well-known arguments. More specifically, referring to the requirements (1)-(5) of \cite[Definition 2.1]{ST:large_proj}, we have that: Properties (1)-(3) are straightforward. Property (4) follows from the contraction property and, e.g., \cite[Lemma 2.5]{Si:contr}. Property (5) follows instead from the fact that there are finitely many cosets of $\langle g \rangle$ such that if the projection of $hl_\phi$ on $l_\phi$ is unbounded, then $h$ belongs to one of these cosets. This follows from, e.g., \cite[Corollary 4.4]{Si:contr}.

The fact that $\mu$ is admissible is also not difficult to be checked: Among the requirements perhaps only property (4) is not immediate. This property says, in our context, that the probability that the random walk ends up in one of the cosets of $\langle g \rangle$ for which the projection of $hl_\phi$ on $l_\phi$ is unbounded  is exponentially small in the length of the walk. This holds because after $n$ steps the random walk can only possibly visit linearly many of the elements of those cosets (they are undistorted), while the probability of ending up at any one of them is exponentially small, just because $\text{\rm Mod}(\Sigma)$ is non-amenable.
Now that we explained why~\cite{ST:large_proj} applies, the third item follows from the \cite[Theorem 2.3]{ST:large_proj}, while the fourth one follows from \cite[Proposition 3.2]{ST:large_proj}, with $R=0$.
\end{proof}

As a different application of the same projection systems framework used above, we have the following statement whose proof is  rather similar to the previous one. It will be used in the proof of Theorem~\ref{main2} via Theorem~\ref{thm:main} and in the application to the decay rate of the shortest geodesics for random 3-manifolds. Recall that the projection $\Upsilon:\T\to\mc{C}$ sends $f_jo$ to $f_j\delta$ with $\delta\in\mc{D}$.

Denote by $L_{\mc{C}}=\lim{d_{\mc{C}}(\delta,f_n\delta)/n}>0$ the drift of the random walk on the curve graph, which is positive by a result of Maher~\cite{Ma}.

\begin{thm}\label{thm:log_proj2}
 Denote by $L_{\mc{C}}$ the curve graph drift of the random walk, and let $0<a<b<1$. Then, there exist $\ep_0>0$ and $C>0$ such that with asymptotic probability $1$ the following holds. There exists a non-separating simple closed curve $\gamma_n\subset\Sigma$ such that
\begin{itemize}
 \item{$d_{\gamma_n}(\delta,f_n\delta)\ge\epsilon_0\log(n)$.}
 \item{$d_Y(\delta,f_n\delta)\le C$ for every proper subsurface $Y\subset\Sigma-\gamma_n$.}
 \item{$d_{\mc{C}}(\delta,\gamma_n)\in [aL_{\mathcal C}n,bL_{\mathcal C}n].$}
\end{itemize}
\end{thm}

\begin{proof}
If we exclude the requirement about the location of the curve $\gamma_n$ with respect to $[\delta,f_n\delta]$, that is, the third requirement in the list, then the statement of the theorem is exactly the content of Proposition 7.1 of \cite{ST:large_proj}. Here we want to control simultaneously the presence of a curve $\gamma_n$ with large annular projection and bounded projections to the subsurfaces disjoint from it together with the position of $\gamma_n$ on the segment $[\delta, f_n\delta]$ in order to make sure that it lies far away from the disk sets $\mc{D}$ and $f_n\mc{D}$. 
 
As for the case of Teichmüller space, we have that the following properties hold with asymptotic probability 1:
\begin{enumerate}[(i)']
\item{$d_{\mc{C}}(f_j\delta,[\delta,f_n\delta])\le C\log(n)$.} 
\item{$d_{\mc{C}} (f_{\lfloor a'n\rfloor}\delta,\delta)\in[(a+\ep)L_{\mc{C}}n,(a+2\ep)L_{\mc{C}}n)]$ and $d_{\mathcal{C}} (f_{\lfloor b'n\rfloor}\delta,\delta)\in[(b-2\ep)L_{\mc{C}}n,(b-\ep)L_{\mc{C}}n)]$.}
\end{enumerate} 
 
Observe that, if there is a large annular subsurface projection between $f_{\lfloor a'n\rfloor}\delta$ and $f_{\lfloor b'n\rfloor}\delta$ on some $\gamma_n$, then the curve $\gamma_n$ lies on the 1-neighborhood of $[f_{\lfloor a'n\rfloor}\delta,f_{\lfloor b'n\rfloor}\delta]$ by the Bounded Geodesic Image Theorem~\cite{MasurMinsky:I}. Just like in the proof of the claim in Theorem \ref{randergodic1}, properties (i)' and (ii)' ensure then that $\gamma_n$ is at the appropriate distance from $\delta$. That is, it satisfies the last requirement of Theorem~\ref{thm:log_proj2}. 
 
Regarding the size of the subsurfaces projections we proceed as follows. We need three ``buffer projections'', as in Proposition 7.1 \cite{ST:large_proj}, whose proof yields the following: There are $\ep_1>0$ and $C_1>0$ such that the following holds with asymptotic probability 1. There exist non-separating curves $\gamma^n_1,\gamma^n_2,\gamma^n_3$ such that 
\begin{enumerate}[(i)']
\setcounter{enumi}{2}
 \item{$d_{\gamma^n_j}(f_{\lfloor a'n\rfloor}\delta,f_{\lfloor b'n\rfloor}\delta)\ge\epsilon_1\log(n)$.}
 \item{$d_Y(f_{\lfloor a'n\rfloor}\delta,f_{\lfloor b'n\rfloor}\delta)\leq C_1$ for all subsurfaces $Y\subset\Sigma-\gamma^n_2$.}
 \item{$d_{\gamma^n_1}(f_{\lfloor b'n\rfloor}\delta,\gamma^n_2), d_{\gamma^n_3}(f_{\lfloor a'n\rfloor}\delta,\gamma^n_2), d_{\gamma^n_2}(f_{\lfloor a'n\rfloor}\delta,\gamma^n_1), d_{\gamma^n_2}(f_{\lfloor b'n\rfloor}\delta,\gamma^n_3)\le C_1$.}
\end{enumerate}

Similar to property (iv) used in the proof of Theorem \ref{randergodic1}, we have the following replacement which holds with asymptotic probability 1 and follows again from Proposition 3.2 of \cite{ST:large_proj}
\begin{enumerate}[(i)']
\setcounter{enumi}{5}
 \item{$d_{\gamma^n_1}(f_{\lfloor a'n\rfloor}\delta, \delta)\le\epsilon_1 \log(n)/3$ and $d_{\gamma^n_3}(f_{\lfloor b'n\rfloor}\delta,\delta)\le\ep\log(n)/3$.}
\end{enumerate}

A consequence of these properties is that the annular projections $\pi_{\gamma^n_2}(\delta)$ and $\pi_{\gamma^n_2}(f_{\lfloor a'n\rfloor}\delta)$ coarsely coincide. This is a routine application of the Behrstock Inequality~\cite{Beh06}, which states that there is a constant $B$ such that for all curves $\alpha, \beta,\gamma$ we have $\min\{d_\gamma(\alpha,\beta),d_\alpha(\gamma,\beta)\}\leq B$ (provided that the quantities are well-defined). Here is the argument: By properties (iii)' and (v)', $d_{\gamma^n_1}(f_{\lfloor a'n\rfloor}\delta,\gamma^n_2)$ is large. Hence, by property (vi)', $d_{\gamma^n_1}(\delta,\gamma^n_2)$ is also large. Therefore, by the Behrstock Inequality, $d_{\gamma^n_2}(\delta,\gamma^n_1)$ is bounded and, by property (v)', the same holds for $d_{\gamma^n_2}(\delta,f_{\lfloor a'n\rfloor}\delta)$ as required. 

The same argument also applies to projections $\pi_Y(\delta)$ and $\pi_Y(f_{\lfloor a'n\rfloor}\delta)$ for all subsurfaces $Y$ in the complement of $\gamma_2^n$. In fact, we have the following: Since $d_{\gamma^n_1}(\delta,\gamma^n_2)$ and $d_{\gamma^n_1}(f_{\lfloor a'n\rfloor}\delta,\gamma^n_2)$ are both large and $Y$ is a subsurface of the complement of $\gamma_2^n$, also $d_{\gamma^n_1}(\delta,\partial Y)$ and $d_{\gamma^n_1}(f_{\lfloor a'n\rfloor}\delta,\partial Y)$ are large. Hence, by the Behrstock Inequality, $d_Y(\delta,\gamma_n^1)$ and $d_Y(f_{\lfloor a'n\rfloor}\delta,\gamma_n^1)$ are both uniformly bounded.

Changing the roles of $\delta$ and $f_{\lfloor a'n\rfloor}\delta$ with $f_n\delta$ and $f_{\lfloor b'n\rfloor}\delta$ concludes the proof.
\end{proof}

\subsection{The proof of Theorem~\ref{main1}}
Consider the Teichmüller segment $[o,f_no]$. Fix $\delta>0$ large enough. We need to find two pseudo Anosov mapping classes $\psi$ and $\psi'$ with short pants decompositions and four surfaces $S_0=o<S_1<S_2<S_3<S_4<f_no=S_5$ such that 
\begin{enumerate}[\rm (i)]
\item{$S_j<S_{j+1}$ and their distance is at least $h$, a very large constant.}
\item{$[S_1,S_2]$, $[S_3,S_4]$ $\delta$-fellow travel $l_\psi,l_{\psi'}$.}
\item{$d_{\mc{C}}(\Upsilon[S_1,S_4],\mc{D})\ge h$ and $d_{\mc{C}}(\Upsilon[S_1,S_4],f_n\mc{D})\ge h$.}
\end{enumerate}
%\todo{Added first item. A}

% We first prove the second property: Recall that we chose $o\in\T$ such that $\delta:=\Upsilon(o)\in\mc{D}$ (and hence $f_n\delta=\Upsilon(f_no)\in f_n\mc{D}$).
% 
% We show that for every fixed $\rho>0$, the probability that these properties hold is at least $1-\rho$ for every $n$ sufficiently large.

We start with with a claim that will help us ensure that the last property holds.

%\todo{changed here, old commented out. A}

{\bf Claim}: For every $h>0$ and $\ep>0$, with asymptotic probability 1 we have that both $d_{\mc{C}}(\Upsilon\tau_n[\ep Ln,(1-\ep)Ln],\mc{D})$ and $d_{\mc{C}}(\Upsilon\tau_n[\ep Ln,(1-\ep)Ln],f_n\mc{D})$ are greater than $h$.

{\em Proof of the claim}. The claim is a consequence of 

\begin{thm}[Maher~\cite{Maher:Heegaard}]
\label{hempeldist}
For every $\ep>0$ we have
\[
\mb{P}_n\qa{\text{\rm $d_{\mc{C}}(\mc{D},f\mc{D})\in[(L_\mc{C}-\ep)n,(L_\mc{C}+\ep)n]$}}\stackrel{n\to\infty}{\longrightarrow}1.
\]
\end{thm}

Choose $\ep_2>\ep_1>0$ much smaller than $\ep$. By Theorem~\ref{hempeldist}, for any given $\rho$ we have $d_{\mc{C}}(\mc{D},f_n\delta)\ge(L_\mc{C}-\ep_1)n$ with probability at least $1-\rho$ for $n$ large.

By Theorem~\ref{tracking} we can assume $d_\T(f_mo,\tau_n(Lm))\le\ep_1m$ for every $\ep_1 n<m<(1-\ep_1)n$ with probability $\ge 1-\rho$ for $n$ large. We also assume $L_\mc{C}-\ep_1<d_{\mc{C}}(\delta,f_n\delta)/n<L_\mc{C}+\ep_1$ with probability $\ge 1-\rho$ for every $n$ large.

Consider $m\in[\ep_1n,(1-\ep_2)n]$. 

We have the following estimate on the distance from $\mc{D}$: Let $B>0$ be the Lipschitz constant of $\Upsilon:\T\to\mc{C}$. Recall that $f_m\delta=\Upsilon(f_mo)$
\begin{align*}
d_{\mc{C}}(\Upsilon\tau_n(Lm),\mc{D}) &\ge d_{\mc{C}}(f_m\delta,\mc{D})-d_{\mc{C}}(f_m\delta,\Upsilon\tau_n(Lm))\\
 &\ge(L_\mc{C}-\ep_1)m-B\ep_1m\\
 &\ge(L_\mc{C}-\ep_1-B\ep_1)\ep_1n.
\end{align*}
Notice that if $\ep_1$ is small enough, the right hand side increases linearly in $n$ with uniform constants. 

As for the other disk set $f_n\mc{D}$, we also get
\begin{align*}
d_{\mc{C}}(\Upsilon\tau_n(Lm),f_n\mc{D}) &\ge d_{\mc{C}}(\delta,f_n\mc{D})-d_{\mc{C}}(\delta,f_m\delta)-d_{\mc{C}}(f_m\delta,\Upsilon\tau_n(Lm))\\
 &\ge(L_\mc{C}-\ep_1)n-(L_\mc{C}+\ep_1)m-B\ep_1 m\\
 &\ge[(L_\mc{C}-\ep_1)-(L_\mc{C}+\ep_1)(1-\ep_2)-B\ep_1(1-\ep_2)]n.
\end{align*}
As before, if $\ep_1$ is very small compared to $\ep_2$, the right hand side increases linearly in $n$ with uniform constants. In conclusion, if $\ep_1$ is small enough and $n$ is large enough, the claim holds as $[\ep Ln,(1-\ep)Ln]\subset[\ep_1n,(1-\ep_2)n]$.\qed

%This settles the proof of property (iii) for the segment $\tau_n[\ep Ln,(1-\ep)Ln]$. Observe that any subsegment $[S^n_1,S^n_4]$ will enjoy the same property.

%%\todo{Commented out imprecise stuff. A}

The next claim will easily allow us to find surfaces $S_i$ satisfying (i) and (ii).

{\bf Claim}: Let $\phi$ be a pseudo-Anosov element with a short pants decomposition. Let $l_\phi:\mb{R}\to\T$ be its axis. For every $\ep>0$, for every $h>0$, with asymptotic probability 1, the Teichmüller segments $\tau_n[\ep Ln,2\ep Ln]$ and $\tau_n[(1-2\ep)Ln,(1-\ep)Ln]$ $\xi$-fellow travel (with $\xi$ only depending on $\phi$), along subsegments $\tau_n[t^n_1,t^n_2]$ and $\tau_n[t^n_3,t^n_4]$ of length at least $h$, some translates $\psi_n=g_nl_\phi$ and $\psi_n'=g'_nl_\phi$ of the axis $l_\phi$.

{\em Proof of the claim}. We just need to apply Theorem~\ref{randergodic1} with parameters $0<a<b$ given by $0<L\ep<2\ep L$ and $0<(1-2\ep)L<(1-\ep)L$ respectively.\qed

{\bf Conclusion of the proof}: For a fixed $\ep>0$ we define $o<S^n_1<S^n_2<S^n_3<S^n_4<f_no$ to be the four surfaces $\tau_n(t^n_1)<\tau_n(t^n_2)<\tau_n(t^n_3)<\tau_n(t^n_4)$ as given by the second claim. By construction, for every $n$ large enough, these satisfy properties (i), (ii) of Proposition~\ref{pa examples} and by the first Claim they also satisfy property (iii). Hence, by Proposition \ref{pa examples} there are pants decompositions $P_n^1,P_n^2,P_n^3,P_n^4$ that satisfy the filling criterion with parameter smaller than $\eta=\eta(K)$ as given by Proposition~\ref{model} where $K>1$ is the bilipschitz constant that we fixed at the beginning. They also satisfy the property of part (b) of Theorem \ref{main1} by our choice of $\tau(t^n_j)$ and Proposition \ref{pa examples}. Therefore, by Proposition \ref{model}, we have that $M_{f_n}$ is hyperbolic and, moreover, away from the cusps, it is $K$-bilipschitz to the metric $\mb{M}_n$ as described by Lemma \ref{five pieces}. As $\mb{M}_n$ contains an isometric copy of $\mc{CC}(Q(P_n^2,P_n^3))$, the bilipschitz map identifies the non-cuspidal part of such piece with a tubular neighborhood of the Heegaard surface (because of the control on the isotopy class of the map provided by Proposition \ref{model}). This concludes the proof of Theorem~\ref{main1}.\qed

\subsection{The proof of Theorem~\ref{main2} via short curves}\label{subsec:thm1viathm3}
We sketch now a proof of Theorem~\ref{main2} that uses the construction of Theorem~\ref{thm:main}.

The argument also gives that $M_{f_n}$ contains a curve of length $\le 1/\log(n)$ (later on, we improve this estimate using model metric, see Theorem~\ref{main6}).

There is yet another version of Theorems~\ref{randergodic1} and~\ref{thm:log_proj2} which says that, with high probability, there is a curve $\gamma_n$ such that $d_{\Sigma-\gamma_n}(\Upsilon(o),\Upsilon(f_no))$ has size at least $\log(n)$, and $\gamma_n$ lies close to the middle of a geodesic in $\mathcal C$ from $\Upsilon(o)$ to $\Upsilon(f_no)$. Similarly to the first Claim in the proof of Theorem~\ref{main1}, we have that $\gamma_n$ also lies on a shortest geodesic connecting $\mathcal D$ and $f_n\mathcal D$, and far from the endpoints of said geodesic. Using that $\mathcal D$ is quasiconvex and the Bounded Geodesic Image Theorem~\cite{MasurMinsky:I}, we have that the subsurface projection to $\Sigma-\gamma_n$ of $\mathcal D$ is bounded and coarsely coincides with that of $\Upsilon(o)$. A similar statement holds for $f_n\mathcal D$. Hence, we get that $d_{\Sigma-\gamma_n}(\mc{D},f_n\mc{D})$ is logarithmically large, and also that $(H_g,\gamma_n)$ and $(H_g,f_n(\gamma_n))$ are both pared acylindrical. We can therefore use Theorem~\ref{thm:main}.

\section{Four applications}
\label{applications}

We describe four applications of Theorem~\ref{main1}. 

We recall that the model metric decomposition consists of five pieces
\[
\mb{M}_n=H_1^n\cup\Omega_1^n\cup Q_n\cup\Omega_2^n\cup H_2^n,
\]
but, for our applications, we will mainly focus on the maximally cusped structure $Q_n=Q(P_n,R_n)$, as given by Theorem \ref{main1precise}. We recall that $P_n,R_n$ are uniformly close to $\Upsilon(X_n),\Upsilon(Y_n)$ where $X_n\in\tau_n[\ep Ln,2\ep Ln]$ and $Y_n\in\tau_n[(1-2\ep)Ln,(1-\ep)Ln]$. We also recall that it bilipschitz embeds, away from its cusps, into $M_{f_n}$ with bilipschitz constant arbitrarily close to 1 as $n$ goes to $\infty$.

\subsection{Diameter growth}
As a first geometric application, we compute the coarse growth rate for the diameter of random 3-manifolds.

\begin{mythm}[\ref{main4}]
There exists $c>0$ such that
\[
\mb{P}_n[\text{\rm diam}(M_f)\in[n/c,cn]]\stackrel{n\to\infty}{\longrightarrow}1.
\]
\end{mythm}

The proof of Theorem~\ref{main4} has two different arguments, one for the coarse upper bound and one for the coarse lower bound. The upper bound comes from a result by White~\cite{W01} that relates the diameter to the {\em presentation length} of the fundamental group, a topological and algebraic invariant. Of a different nature is the coarse lower bound where we heavily use the $\ep$-model metric structure of Theorem~\ref{main1} and the relation with the model manifold. 

We start with the upper bound. We need the following definition

\begin{dfn}[Presentation Length]
Let $G$ be a finitely presented group. The length of a finite presentation $G=\langle F\left|R\right.\rangle$ is given by
\[
l(F,R)=\sum_{r\in R}{|r|_F-2}
\]
where $|r|_F$ denotes the word length of the relator $r\in R$ with respect to the generating set $F$. The presentation length of $G$ is defined to be
\[
l(G):=\min\ga{l(F,R)\left|\text{ \rm $G=\langle F\left| R\right.\rangle$ finite presentation}\right.}.
\]
We also recall that a relator $r\in R$ is triangular if $|r|_F\le 3$. 
\end{dfn}

\begin{thm}[White~\cite{W01}]
There exists $c>0$ such that for every closed hyperbolic 3-manifold $M$ we have
\[
\text{\rm diam}(M)\le c\cdot l(\pi_1M).
\]
\end{thm}

Let $S\subset\text{\rm Mod}(\Sigma)$ be the finite support of the probability measure $\mu$.

\begin{lem}
There exists $C(S)>0$ such that for every $f\in\text{\rm Mod}(\Sigma)$ we have
\[
l(\pi_1(M_f))\le C|f|_S.
\]
In particular $\text{\rm diam}(M_f)\le K|f|_S$ where $K=c\cdot C$.
\end{lem}

\begin{proof}
The 3-manifold $M_f$ admits a triangulation $T$ with a number of simplices uniformly proportional, depending on $S$, to the word length $|f|_S$. We have $\pi_1(M_f)=\pi_1(T_2)$ where $T_2$ denotes the 2-skeleton of $T$. By van Kampen, the fundamental group of a 2-dimensional connected simplicial complex $X$ admits a presentation $\pi_1(X)=\langle F\left| R\right.\rangle$ where every relation is triangular and the number of relations $|R|$ is roughly the number of 2-simplices.   
\end{proof}

As a corollary, we get
\[
\text{\rm diam}(M_{f_n})\le K|f_n=s_1\dots s_n|_S\le Kn
\]
thus proving the upper bound in Theorem~\ref{main4}.

The coarse lower bound follows from the structure of the model metric and the estimate of Theorem~\ref{blocksep} that comes from the model manifold technology of Minsky~\cite{M10}.

In particular, by Theorem~\ref{blocksep}, if $Q_n=Q(P_n,R_n)$ is a maximal cusp then the distance between the boundary components of its non-cuspidal part $Q^{\text{\rm nc}}$ is at least $Ad_{\mc{C}}(P_n,R_n)-A$. In the case of random 3-manifolds we have
\begin{align*}
d_{\mc{C}}(P_n,R_n) &\asymp d_{\mc{C}}(\Upsilon(X_n),\Upsilon(Y_n))\\
 &\ge d_\mc{C}(\Upsilon(o),\Upsilon(f_no))-d_{\mc{C}}(\Upsilon(o),\Upsilon(X_n))-d_{\mc{C}}(\Upsilon(Y_n),f_n\Upsilon(o))\\
 &\asymp L_\mc{C} n-o(n).
\end{align*}

\subsection{Injectivity radius decay}
As a second geometric application, we give a coarse upper bound to the decay rate of the length of the shortest geodesic of random 3-manifolds.
 
\begin{mythm}[\ref{main6}]
There exists $c>0$ such that
\[
\mb{P}_n\left[\text{\rm inj}(M_f)\le c/\log(n)^2\right]\stackrel{n\to\infty}{\longrightarrow}1.
\]
\end{mythm}

\begin{proof}
By Theorem~\ref{main1precise}, it is enough to show that $Q(P_n,R_n)$ satisfies 
\[
\text{\rm systole}(Q(P_n,R_n))\le c/\log(n)^2.
\]
We recall that $P_n$ and $R_n$ are uniformly close to $\Upsilon(X_n),\Upsilon(Y_n)$ with $X_n\in\tau_n[\ep L n,2\ep L n]$ and $Y_n\in\tau_n[(1-2\ep)Ln,(1-\ep)Ln]$ where $\tau_n$ is the parametrized Teichmüller segment $[o,f_no]$, for a fixed basepoint $o\in\T$.

By Minsky~\cite{M00}, a curve $\gamma\in\mc{C}$ not contained in $P_n$ or $R_n$ is short in $Q(P_n,R_n)$ if and only if there is a large subsurface projection $d_Y(P_n,R_n)$ on a proper subsurface $Y\subset\Sigma$ with $\gamma\subset\partial Y$. Furthermore, by the Length Bound Theorem~\cite{BrockCanaryMinsky:ELC2}, its length will be bounded by
\[
\ell_{Q(P_n,R_n)}(\gamma)\le D\frac{S_\gamma(P_n,R_n)}{d_\gamma(P_n,R_n)^2+S_\gamma(P_n,R_n)^2}
%\frac{1}{D}\cdot\frac{S_\gamma(P_2^n,P_3^n)}{d_\gamma(P_2^n,P_3^n)^2+S_\gamma(P_2^n,P_3^n)^2}\le
\]
for some uniform constant $D>0$ and where $d_\gamma(P_n,R_n)$ is the annular projection corresponding to $\gamma$ and 
\[
S_\gamma(P_n,R_n)=1+\sum_{Y\in\mc{Y}_\gamma}{\{\{d_Y(P_n,R_n)\}\}_K}.
\]
Here $\mc{Y}_\gamma$ denotes the collection of essential subsurfaces of $\Sigma-\gamma$, $K>0$ is, again, some uniform constant, and $\{\{\bullet\}\}_K$ is the function defined by $\{\{x\}\}_K=x$ if $x>K$ and 0 otherwise.

Putting things together, it is enough to show that there exists a curve $\gamma\subset\Sigma$ for which $d_\gamma(P_n,R_n)\ge\ep_0\log(n)$ and $S_\gamma(P_n,R_n)\le C$ for some uniform $\ep_0>0$ and $C$. For the purposes of the argument below, we note that the latter condition is equivalent to having uniformly bounded projections on all $Y$ contained in $\Sigma-\gamma$, in view of the distance formula of \cite{MasurMinsky:II}.

Replacing $P_n$ with $\delta=\Upsilon(o)$ and $R_n$ with $f_n\delta=\Upsilon(f_no)$, the aforementioned property is contained the statement of Theorem~\ref{thm:log_proj2}. To conclude, we only have to argue that $\delta$ and $P_n$ have coarsely the same subsurface projections to the annulus corresponding to $\gamma$ and to all subsurfaces contained in $\Sigma-\gamma$ (and similarly for $f_n\delta$ and $R_n$). But this holds provided that we choose $\ep$ small enough in Theorem~\ref{main1}, and $a,b$ sufficiently close to $1/2$ in Theorem~\ref{thm:log_proj2}. In fact, in this case we have that geodesics in $\mathcal C$ from $\delta$ to $P_n$ cannot pass 2-close to $\gamma_n$, just because they are much shorter than the distance from $\delta$ to $\gamma_n$. We can therefore apply the Bounded Geodesic Image Theorem~\cite{MasurMinsky:I} and conclude (since a similar argument also applies to $f_n\delta$ and $R_n$).
\end{proof}

We conclude the discussion with a couple of remarks on the lower bound for the injectivity radius. If we consider only the three middle pieces $\Omega_n^1\cup Q_n\cup\Omega_n^2$ of the $\ep$-model metric, the rate of $1/\log(n)^2$ is exactly the coarse decay rate of the systole. This is again an adaptation of the arguments of \cite{ST:large_proj}. Hence, in order to get a precise lower bound, we have to understand the systole of the handlebody pieces $H_n^1$ and $H_2^n$. Such computation would be possible, for example, in the presence of a model manifold technology for handlebodies analogue to the one of Minsky \cite{M10} and Brock, Canary and Minsky \cite{BrockCanaryMinsky:ELC2} for hyperbolic manifolds diffeomorphic to $\Sigma\times\mb{R}$.

\subsection{Geometric limits of random 3-manifolds}
We now exploit the model metric structure to establish the existence of certain geometric limits (see Chapter E.1 of~\cite{BP92} for the definition of the pointed geometric topology) for families of random 3-manifolds. 

These limits will be used in our last application concerning the arithmeticity and the commensurability class of random 3-manifolds.

\begin{pro}
\label{main5}
Let $\phi\in\text{\rm Mod}(\Sigma)$ be a pseudo-Anosov mapping class. Consider a sequence $A_n\subset\text{\rm Mod}(\Sigma)$ such that $\limsup\mb{P}_n[A_n]>0$. Then, we can find a sequence $n_j\uparrow\infty$ and elements $f_{n_j}\in A_{n_j}$ such that $M_{f_{n_j}}$ are hyperbolic 3-manifolds and the sequence $M_{f_{n_j}}$ converges to the infinite cyclic covering of hyperbolic mapping torus $T_\phi$ in the pointed geometric topology for a suitable choice of base points $x_{n_j}\in M_{f_{n_j}}$.
\end{pro}

\begin{proof}
By assumption, there exist $\delta>0$ and a sequence $m_j\uparrow\infty$ such that $\mb{P}_{m_j}[A_{m_j}]\ge\delta$. We choose $(n_j)_{j\in\mb{N}}$ by inductively refining $(m_j)_{j\in\mb{N}}$.

By Theorem~\ref{randergodic1} and Theorem~\ref{main1}, for every $k\in\mb{N}$ we have that the event
\[
G_{n,k}:=\left\{
\begin{array}{c}
\text{\rm $M_{f_n}$ satisfies Theorem~\ref{main1} with parameters $K:=1+1/k$ and $\ep$}\\
\text{\rm $\tau_n=[o,f_no]$ satisfies Theorem~\ref{randergodic1} with parameters $\ep$ and $\phi$}
\end{array}
\right\}
\]
has probability at least $1-\delta/10$ for every sufficiently large $n$, say for $n\ge N_k$. In particular, if $m_i\ge N_k$, we have $A_{m_i}\cap G_{m_i,k}\neq\emptyset$. 

We recall that Theorem \ref{main1} provides a maximally cusped structure $Q(P_n,R_n)$ whose convex core $K$-bilipschitz embeds away from the cusp into $M_{f_n}$. Furthermore, the pants decomposition $P_n$ is located uniformly close to $\Upsilon\tau_n[\ep Ln,2\ep Ln]$ while the pants decomposition $R_n$ is uniformly close to $\Upsilon\tau_n[(1-\ep)Ln,(1-2\ep)Ln]$. 

We define now inductively the sequence $(n_j)_{j\in\mb{N}}$. Suppose that we have already chosen $n_1,\cdots,n_{j-1}$. The next element will be 
\[
n_j:=\min\{m_i\left|m_i>\max\{n_{j-1},N_j\}\right.\}.
\] 
As $n_j>N_j$, we have $A_{n_j}\cap G_{n_j,j}\neq\emptyset$, so we can choose $f_{n_j}\in A_{n_j}\cap G_{n_j,j}$.
 
We recall that $\tau_n=[o,f_no]$ satisfies Theorem~\ref{randergodic1} with parameters $\phi$ and $\ep$, so it has a subsegment $\tau_n[3\ep Ln,(1-3\ep)Ln]$ that uniformly travels a translate $g_nl_\phi$ along a subsegment of length $\ep_0\log(n)$.

Up to remarking $\tau_n$, an operation that does not change the isometry type of $Q(P_{n_j},R_{n_j})$, we can assume that $\tau_{n_j}$ uniformly fellow travels $l_\phi$ along the subsegment $l_\phi[-a_{n_j},a_{n_j}]$ with $a_{n_j}=\ep_0\log(n_j)\uparrow\infty$. Hence, the sequence of remarked Teichmüller segments $\tau_{n_j}[3\ep Ln_j,(1-3\ep)Ln_j]$ is converging uniformly on compact subsets to a geodesic $l$ that uniformly fellow travels the axis $l_\phi$. 

Notice that $l_\phi$ converges in the forward and backward directions to the projective classes of the invariant laminations $\lambda^+$ and $\lambda^-$ of the pseudo-Anosov mapping class $\phi$. Since $l$ fellow travels $l_\phi$ and $\lambda^+$ and $\lambda^-$ are are minimal, filling and uniquely ergodic (see Exposé 9 and Exposé 12 of~\cite{FLP}), we conclude that also $l$ converges in the forward and backward direction to the same laminations and so do the sequences of remarked segments $\tau_{n_j}[0,2\ep Ln_j]$ and $\tau_{n_j}[(1-2\ep)Ln_j,(1-\ep)Ln_j]$ (see, for example, Lemma 1.4.2 of~\cite{KM:Poisson}).  

As $P_{n_j},R_{n_j}$ lie uniformly close to the $\Upsilon$-projections of such segments, this immediately implies, by work of Klarreich \cite{Kla99}, that $P_{n_j}\to\lambda^-$ and $R_{n_j}\to\lambda^+$.

By work of Brock, Bromberg, Canary, and Minsky~\cite{BBCM13} and the solution of the Ending Lamination Conjecture~\cite{M10,BrockCanaryMinsky:ELC2}, we conclude that the sequence of maximally cusped structures $Q(P_{n_j},R_{n_j})$ converges to the infinite cyclic covering of $T_\phi$. As $Q(P_{n_j},R_{n_j})$ becomes geometrically arbitrarily close to $M_{f_{n_j}}$, the claim follows.

\end{proof}

\subsection{Commensurability and arithmeticity}
Dunfield and Thurston, using a simple homology computation, have shown in~\cite{DT06} that their notion of random 3-manifold is not biased towards a certain fixed set of 3-manifolds. This means that for every fixed 3-manifold $M$, with asymptotic probability 1, $M_f$ is not diffeomorphic to $M$. 

Using geometric tools it is possible to strengthen this conclusions and show that Dunfield and Thurston's notion of random 3-manifolds is also transverse, in a sense made precise in the theorem below, to the class of arithmetic hyperbolic 3-manifolds and to the class of 3-manifolds which are commensurable to a fixed 3-manifold $M$.  

\begin{mythm}[\ref{main3}]
With asymptotic probability $1$ the following holds
\begin{enumerate}
\item{$M_f$ is not arithmetic.}
\item{$M_f$ is not in a fixed commensurability class $\mc{R}$.}
\end{enumerate}
\end{mythm}

\begin{proof}
The argument is mostly borrowed from Biringer-Souto~\cite{BS11}.

The proof of both points starts from the following observation: Each $M_{f_n}$ finitely covers a maximal orbifold $M_{f_n}\to\mc{O}_n$. 

We first prove the non-arithmeticity. We argue by contradiction: Suppose that $\mb{P}_n[\text{\rm $M_f$ is arithmetic}]$ does not go to 0. Combining with Theorem~\ref{main6}, we also have 
\[
\limsup\mb{P}_n[\text{$M_f$ is arithmetic and $\text{\rm inj}(M_f)\le c/\log(n)^2$}]>0.
\]
By Proposition~\ref{main5}, up to passing to a subsequence, say the whole sequence for simplicity, we can pick $M_{f_n}$ such that $M_{f_n}$ is arithmetic, has injectivity radius $\text{\rm inj}(M_{f_n})\le c/\log(n)^2$ and there are base points $x_n\in M_{f_n}$ such that the sequence $(M_{f_n},x_n)$ converges geometrically to $(Q_\infty,x_\infty)$ where $Q_\infty$ is a doubly degenerate structure on $\Sigma\times\mb{R}$ with $\text{\rm inj}(Q_\infty)>0$.

Since $M_{f_n}$ are arithmetic, the orbifolds $\mc{O}_n$ are congruence and have $\lambda_1(\mc{O}_n)\ge 3/4$ (see~\cite{BS91} or Theorem 7.1 in~\cite{BS11}). By Proposition 4.3 of~\cite{BS11}, the orbifolds $\mc{O}_n$ cannot be all different, hence we can assume that they are fixed all the time $\mc{O}_n=\mc{O}$. We get a contradiction by observing that $\mc{O}$ is covered by closed 3-manifolds $M_{f_n}$ with arbitrarily small injectivity radius. 

We now discuss commensurability. Proceed again by contradiction and assume that $\mb{P}_n[\text{\rm $M_f$ is in the commensurability class $\mc{R}$}]$ does not go to 0. By the arithmetic part we know that we also have
\[
\limsup\mb{P}_n[\text{\rm $M_f\in\mc{R}$, $M_f$ not arithmetic, and $\text{\rm inj}(M_f)\le c/\log(n)^2$}]>0.
\]
As before, using Proposition~\ref{main5}, choose a geometrically convergent sequence $(M_{f_n},x_n)\to(Q_\infty,x_\infty)$ of non-arithmetic, commensurable hyperbolic 3-manifolds with $\text{\rm inj}(M_{f_n})\downarrow 0$. Commensurability and non-arithmeticity imply together that $\mc{O}_n=\mc{O}$ is fixed all the time: It is the orbifold corresponding to the commensurator $\text{\rm Comm}(\pi_1(M_{f_n}))$, which is a discrete subgroup of $\text{\rm PSL}_2\mb{C}$ by Margulis (see Theorem 10.3.5 in~\cite{MR}) and is an invariant of the commensurability class. We conclude with the same argument as before.
%\todo{Add remark on the Lehmer Conjecture as suggested by the referee, as this might be considered evidence towards Lehmers conjecture. We decide not to do ths, since it would be very weak evidence}
  
\end{proof}

\appendix
\section{Curve graph and pared acylindrical manifolds}
\label{appendix c}

\begin{lemma}[2.4]
Let $\gamma,\gamma'\subset\Sigma$ be essential multicurves. We have
\begin{enumerate}[{\rm (i)}]
\item{If $d_{\mc{C}}(\gamma,\mc{D})\ge 2$ then $(H_g,\gamma)$ is pared.}
\item{If $d_{\mc{C}}(\gamma,\mc{D})\ge 3$ then $(H_g,\gamma)$ is pared acylindrical.}
\item{If $d_{\mc{C}}(\gamma,\gamma')\ge 1$ then $(\Sigma\times[0,1],\gamma\times\{0\}\sqcup\gamma'\times\{1\})$ is pared.}
\item{If $d_{\mc{C}}(\gamma,\gamma')\ge 3$ then $(\Sigma\times[0,1],\gamma\times\{0\}\sqcup\gamma'\times\{1\})$ is pared acylindrical.}
\end{enumerate}
\end{lemma}

\begin{proof}
Let us prove (i). By Dehn's Lemma, if $\Sigma-\gamma$ is not $\pi_1$-injective, then $\Sigma-\gamma$ admits a compression which implies $d_{\mc{C}}(\gamma,\mc{D})\le 1$. Therefore, if  $d_{\mc{C}}(\gamma,\mc{D})\ge 2$, the inclusion $\Sigma-\gamma\subset H_g$ satisfies property (1). 

As for property (2), we proceed as follows: Consider an essential annulus $f:(A,\partial A)\to(H_g,N(\gamma))$ which cannot be properly homotoped into $N(\gamma)$. The boundary curves $\alpha_j=f(\partial_jA)$, for $j=1,2$, are homotopic to a power of the simple core curve $\gamma\subset N(\gamma)$, say $\alpha_j\simeq \gamma^{n_j}$ with $n_j\neq 0$ because $f$ is essential. We freely homotope $f$ such that $\alpha_j=\gamma^{n_j}$ and a simple arc $\eta\subset A$ joining the two boundary components of $\partial A$ maps to a loop $\beta=f(\eta)$ in $H_g$. 

By assumption $\alpha_1$ and $\alpha_2$ are homotopic in $H_g$. At the level of fundamental groups, we have $\gamma^{n_1}=\beta\gamma^{n_2}\beta^{-1}$ in $\pi_1(H_g)$. Since $\pi_1(H_g)=\mb{F}_g$ is a free group, necessarily $\beta=\gamma^k$ for some $k\in\mb{Z}$. Hence we can homotope $f$ such that it maps a regular neighborhood $U$ of $\eta\cup\partial A$ to $N(\gamma)$. The complement of $U$ in $A$ is a disk with boundary $\delta$ mapped to $N(\gamma)$. Since $N(\gamma)$ maps $\pi_1$-injectively to $H_g$, the loop $f(\delta)$ is also homotopically trivial in $N(\gamma)$. Since $H_g$ is aspherical we can homotope $f$ restricted to the complement of $U$ to a nullhomotopy of $f(\delta)$ in $N(\gamma)$. Hence $\Sigma-\gamma\subset H_g$ satisfies (2). 

We now prove (ii). We need to check property (3). Consider an essential annulus $f:(A,\partial A)\to(H_g,\Sigma-\gamma)$ which cannot be properly homotoped into $\Sigma-\gamma$. By the Annulus Theorem, we can assume that $f$ is an embedding. We conclude using the following.

{\bf Claim}. $H_g-\gamma$ does not contain any properly embedded essential annulus $(A,\partial A)\subset (H_g,\Sigma-\gamma)$.

{\em Proof of the claim}. Since handlebodies do not contain incompressible and $\partial$-incompressible surfaces, the annulus $A$ admits a boundary compression. This means that we find an embedded disk $D^2\subset H_g$ whose boundary is divided into two segments $\partial D^2=\alpha\cup\beta$ with $\alpha\subset A$ and $\beta\subset\Sigma$, both joining the two components of the boundary $\partial A$. The boundary $\delta$ of a tubular neighborhood of $\partial A\cup\beta$ is a disk bounding curve. By construction it has distance at most 2 from $\gamma$. This concludes the proof.\qed

The proofs of (iii) and (iv) are analogous. 

Consider (iii) first. Conditions (1) is clear as $\Sigma\times[1,2]$ deformation retracts to $\Sigma\times\{i\}$ for $i=0,1$. Also conditions (2) is straightforward: Let $P:=\gamma\times\{1\}\sqcup\gamma'\times\{2\}$ be the pared locus. Consider a map $f:(A,\partial A)\to (\Sigma\times[1,2],U)$ where $U$ is a regular neighborhood of $P$. Consider the homotopic loops $\alpha_1=f(\partial_1A),\alpha_2=f(\partial_2A)$. Since $d_{\mc{C}}(\gamma,\gamma')\ge 1$ and $\Sigma\times[0,1]$ deformation retracts to $\Sigma=\Sigma\times\{0\}$, different components of $P$ are not homotopic in $\Sigma\times[0,1]$. Thus $\alpha_1,\alpha_2$ are contained in a single component $U(\alpha)$ of $U$ with core curve $\alpha\subset P$, either on $\Sigma\times\{1\}$ or on $\Sigma\times\{2\}$ (say the first). Since $\Sigma\times[0,1]$ deformation retracts to $\Sigma$, we have that $\alpha_1\simeq\alpha_2\simeq\alpha^n$ for some $n\neq 0$. We freely homotope $f$ so that $\alpha_j=\alpha^n$. Let $\eta\subset A$ be a horizontal arc. Notice that the curve $\beta=f(\eta)$ is a loop and satisfies $\alpha^n\simeq\beta\alpha^n\beta^{-1}$ in $\pi_1(\Sigma)$. This implies that $\beta\simeq\alpha^k$ for some $k\in\mb{Z}$. We conclude as in the proof of (i).  

Consider now (iv). Let $f:(A,\partial A)\to(\Sigma\times[0,1],\Sigma_0\sqcup\Sigma_1-P)$ be an essential annulus that cannot be properly homotoped in $\Sigma_0\sqcup\Sigma_1-P$. Again, by the Annulus Theorem, we can replace $f$ with an embedding. In $\Sigma\times[0,1]$ properly embedded annuli are either isotopic to vertical ones $\alpha\times[0,1]$ or boundary parallel. Thus it is enough to show that $f(A)$ is not vertical: As before, consider the homotopic simple closed curves $\alpha_j=f(\partial_jA)$ with $j=1,2$. On $\Sigma$, they represent the same isotopy class and $\alpha_1$ is disjoint from $\gamma$ while $\alpha_2$ is disjoint from $\gamma'$. This means that $d_{\mc{C}}(\gamma,\gamma')\le 2$ contradicting our initial assumption. Thus $f(A)$ is not a vertical annulus and, hence, can be homotoped relative into the boundary.
\end{proof}

\section{Double incompressibility for pared handlebodies}
\label{appendix b}

We give a proof of Proposition~\ref{doubleincomp} whose statement we recall

\begin{proposition}[5.6]
If $(H_g,\gamma)$ is pared acylindrical, then the inclusion $\Sigma-\gamma\subset H_g$ is doubly incompressible.
\end{proposition}

We have to prove that the conditions (a)-(e) of the definition of double incompressibility hold. We proceed step by step by checking one condition at a time. Note that conditions (a) and (c) both follow immediately from the defining properties of pared acylindrical handlebodies. Hence, we only focus on (b) and (e).

\subsection{Homotopy classes of arcs}
We check condition (b). 

\begin{lem*}
Essential relative homotopy classes of arcs $(I,\partial I)\to (\Sigma-\gamma,N(\gamma))$ map injectively into relative homotopy classes of arcs $(I,\partial I)\to (H_g,U(\gamma))$.
\end{lem*}

\begin{proof}
For simplicity denote $A:=N(\gamma)$ and $U:=U(\gamma)$. Consider two arcs $\alpha,\beta$ with endpoints in $\text{\rm int}(A)$, each intersecting $\partial A$ transversely in exactly two points. 

Suppose that they are homotopic as maps into $(H_g,A)$. Then, we can find arcs $\xi,\delta$ in $\text{\rm int}(A)$, each joining an endpoint of $\alpha$ and an endpoint of $\beta$, such that the concatenation $\kappa=\xi*\alpha*\delta^{-1}*\beta^{-1}\subset\Sigma$ is nullhomotopic in $H_g$.

Either $\kappa$ is nullhomotopic in $\Sigma$, in which case $\alpha$ and $\beta$ represent the same homotopy class $(I,\partial I)\to (\Sigma,A)$, or $\kappa$ is essential in $\Sigma$.

Suppose we are in the second case. Up to a little perturbation we can assume that $\kappa$ has only transverse self intersections and intersects $\partial A$ exactly in $(\alpha\cap\partial A)\cup(\beta\cap\partial A)$. By the Loop Theorem there is a diskbounding curve $\eta$ in $\kappa\cup U$ where $U$ is a tiny neighborhood of the singular set of transverse self intersections of $\kappa$. Such a curve $\eta$ has geometric intersection at most 2 with $\partial A$ and hence with $\gamma$.

{\bf Claim}. If $i(\eta,\gamma)\le 2$, then $\Sigma-\gamma$ has either an essential disk or an essential annulus.  

In particular, the existence of $\eta$ contradicts the assumption on $(H_g,\gamma)$ being pared acylindrical.

{\em Proof of the claim}. The curve $\eta$ bounds an essential disk $\eta=\partial D^2$ in $H_g$.

If $i(\eta,\gamma)=0$, then $D^2$ is an essential disk disjoint from $\gamma$. 

If $i(\eta,\gamma)=1$, then the boundary of a regular neighborhood of $D^2\cup\gamma$ in $H_g$ is an essential disk disjoint from $\gamma$.

If $i(\eta,\gamma)=2$, then the boundary of a regular neighborhood of $\gamma\cup D^2$ in $H_g$ contains an essential annulus disjoint from $\gamma$.
\end{proof}

\subsection{Maximal abelian subgroups}
We check condition (e). 

\begin{lem*}
Maximal cyclic subgroups of $\pi_1(\Sigma-\gamma)$ are mapped to maximal cyclic subgroups of $\pi_1(H_g)$.
\end{lem*}

\begin{proof}
We need to check that every primitive element of $\pi_1(\Sigma-\gamma)$ is also primitive in $\pi_1(H_g)$. We proceed as in Canary-McCullogh (see Lemma 5.1.1 in~\cite{CMcC}). Suppose this is not the case, then there exists an essential map $f:A=S^1\times I\to H_g$ such that $f(\partial_1A)=\alpha$, a loop representing a primitive element in $\pi_1(\Sigma-\gamma)$, and $f(\partial_2A)=\beta^k$ for some $k\ge 2$ and $\beta\not\in\pi_1(\Sigma-\gamma)$. 

The map $f:A\to H_g$ factors through $f_0:A_0\to H_g$ where $A_0$ is the quotient space obtained by identifying points on $\partial_2A$ that differ by a $2\pi/k$-rotation. We have $f_0(\partial_1A_0)=\alpha$ and $f_0(\partial_2A_0)=\beta$. Notice that $A_0$ embeds in a solid torus $\mb{T}=D^2\times S^1$ in such a way that $\partial_1A_0$ is a simple closed curve on $T:=\partial\mb{T}$ and $\partial_2A_0$ is the core curve $0\times S^1$ and moreover $\mb{T}$ deformation retracts to $A_0$. By the last property we can extend $f_0$ to a map $F_0:\mb{T}\to H_g$. 

We show that $F_0$ can be homotoped relative to $\partial_1A_0$ such that $F_0(\mb{T})\subset\Sigma-N(\gamma)$. 
This implies that $\alpha=F_0(\partial_1A_0)$ is homotopic in $\Sigma-\gamma$ to $F_0(\partial_2A_0)^k$ and, hence, it could not have been primitive.

The boundary $T=\partial\mb{T}$ is divided into two annuli $T=U\cup V$: A tubular neighborhood $U$ of $\partial_1A_0$ and the complement $V$. Up to a small homotopy we can assume $F_0(U)\subset\Sigma-N(\gamma)$. Consider the restriction of $F_0$ to the annulus $V$. We claim that we can homotope it into $\Sigma-N(\gamma)$. In fact, if this were not the case, then by the Annulus Theorem we would find an essential embedded annulus $(A,\partial A)\subset (H_g,\Sigma-N(\gamma))$ contradicting the fact that $(H_g,\gamma)$ is pared acylindrical. Therefore we can homotope $F_0$ relative to $U$ in such a way that $F_0(T)\subset\Sigma-N(\gamma)$.

We finally show that we can homotope $F_0$ such that $F_0(\mb{T})\subset\Sigma-\gamma$. The meridian $\mu=\partial D^2\times\{\star\}$ of the solid torus $\mb{T}$ is now mapped to a loop in $\Sigma-N(\gamma)$ which is nullhomotopic in $H_g$. Since $\Sigma-\gamma$ is $\pi_1$-injective, the loop $F_0(\mu)$ is also trivial in $\Sigma-\gamma$. As $H_g$ is aspherical, we can homotope the restriction of $F_0$ to $D^2\times\{\star\}$ to a nullhomotopy that takes place in $\Sigma-N(\gamma)$. Finally, as the the complement of $T\cup D^2\times\{\star\}$ is a 3-ball $B$, using again the fact that $H_g$ is aspherical we can homotope $F_0$ restricted to $B$ such that the image of the entire solid torus $\mb{T}$ lies in $\Sigma-N(\gamma)$.        
\end{proof}

\section{Isotopies of Margulis tubes}
\label{appendix}

We prove the following elementary lemma:

\begin{lemma}[4.10]
There exists $\eta<\eta_3/2$ such that the following holds. Let $\mb{T}_{\eta_3}$ be a Margulis tube with core geodesic $\alpha$ of length $l(\alpha)\le\eta$. Suppose that there exists a $2$-bilipschitz embedding of the tube in a hyperbolic 3-manifold $f:\mb{T}_{\eta_3}\to M$. Then $f(\alpha)$ is homotopically non-trivial and it is isotopic to its geodesic representative within $f(\mb{T}_{\eta_3})$.
\end{lemma}

\begin{proof}
We proceed by small steps.

Let $T_\alpha:=\partial\mb{T}_{\eta_3}$ be the boundary of the Margulis tube.  

{\bf Claim}. We have $d(f(\mb{T}_\eta),f(T_\alpha))\ge (c\log(\eta_3/\eta)-c)/2$ for some universal constant $c>0$.

{\em Proof of the claim}. By Brooks-Matelski \cite{BM82}, the analogous claim is true in $\mb{T}_{\eta_3}$, that is $d(\mb{T}_\eta,T_\alpha)\ge c\log(\eta_3/\eta)-c$ for some universal constant $c>0$. If $\gamma$ is a path of minimal length joining $f(\mb{T}_\eta)$ to $f(T_\alpha)$ in $M$ then it must be contained in $f(\mb{T}_{\eta_3})$ because $f(T_\alpha)$ separates $M$. Thus $\ell(\gamma)\ge\ell(f^{-1}\gamma)/2$, as $f$ is $2$-bilipschitz. The claim follows.\qed 

We use the fact that $f(\mb{T}_\eta)$ lies deeply inside $f(\mb{T}_{\eta_3})$ to control how $f$ can distort distances:  

{\bf Claim}. For every (large) $L>0$, there exists an $\eta>0$ such that the following holds. Consider a geodesic arc $\gamma$ contained in $\mb{T}_\eta$ with length $\ell(\gamma)\le L$. Let $\kappa$ be the geodesic representative of $f(\gamma)$ in $M$ (i.e.~the unique geodesic homotopic to $f\circ\gamma$ relative endpoints). Then $\kappa$ is contained in $f(\mb{T}_{\eta_3})$ and homotopic within it to $f(\gamma)$. Furthermore, we have $\ell(\gamma)/2\le\ell(\kappa)\le 2\ell(\gamma)$.

Before proving the claim, notice that $\mb{T}_\eta$ is convex, so it contains all geodesics between any two of its points. 

{\em Proof of the claim}. Observe that $\kappa$ must be contained in $f(\mb{T}_{\eta_3})$ for sufficiently small $\eta$: In fact we have $\ell(\kappa)\le 2\ell(\gamma)\le 2L$ and if $\kappa$ leaves $f(\mb{T}_{\eta_3})$ then its length is at least $(c\log(\eta_3/\eta)-c)/2$ (and we can suppose that $\eta$ is small enough so that $(c\log(\eta_3/\eta)-c)/2>2L$). Also notice that the homotopy between $f\circ\gamma$ and $\kappa$ can be chosen to be length non-increasing so that the whole homotopy is contained in $f(\mb{T}_{\eta_3})$. As a consequence we can consider $f^{-1}\circ\kappa$ which is a path in $\mb{T}_{\eta_3}$ homotopic relative to the endpoints to $\gamma$. In particular $2\ell(\kappa)\ge\ell(f^{-1}\kappa)\ge\ell(\gamma)$. \qed

%If $f(\alpha)$ is compressible then it is compressible in a neighborhood of $f(\alpha)$ of size comparable to $\ell(f(\alpha))$ (lift $f\alpha$ to the universal cover, we can compress such loop to a point by a straight line homotopy). In particular, it would be compressible within $f(\mb{T}(\alpha))$ which is not possible (the distance between $f(\alpha)$ and $f(T_\alpha)$ is coarsely at least $\log(1/\eta)/K$ which is much larger than $\eta$). Therefore $f$ is $\pi_1$-injective. As a consequence $f(\mb{T}(\alpha))$ lies in the thin part of $M$, that is $f(\mb{T}(\alpha))\subset\mb{T}(\beta)$. 

The universal cover of $\mb{T}_{\eta_3}(\alpha)$ is a $a$-neighborhood $N_a(\ell)$ of a geodesic $\ell\subset\mb{H}^3$. Denote by $F:N_a(\ell)\to\mb{H}^3$ the lift of $f$ to the universal coverings. Consider the restriction of $F$ to $\ell$. We can subdivide $\ell$ into segments of length $L/2$. By the claim, the restriction of $F$ to each of these segments is a $2$-bilipschitz embedding in the metric sense.

{\bf Claim}. The restriction of $F$ to $\ell$ is a uniform quasi geodesic (with universal constants).

{\em Proof of the claim}. This follows from the local to global property of quasi geodesics in hyperbolic spaces.\qed

As a consequence, $f(\alpha)$ is not nullhomotopic and, by stability of quasi-geodesics, it lies $r$-close (where $r$ is a universal constant) to its geodesic representative $\beta$ in $M$ (a very short geodesic of length $\ell(\beta)\le 2\ell(\alpha)<2\eta$). If $\eta$ is chosen small enough, then $\beta$ is contained in $f(\mb{T}_{\eta_3})$ and $f(\alpha)\subset N_r(\beta)\subset f(\mb{T}_{\eta_3})$ where $N_r(\beta)$ is the $r$-neighborhood of $\beta$ in $M$.

We want to show that $f(\alpha)$ is actually isotopic to $\beta$. 

Notice, however, that, a priori, $\beta$ can be non-primitive, so we will consider instead $\beta=\gamma^k$ with $\gamma$ primitive. The simple closed geodesic $\gamma$ is the core of the Margulis tube.

The proof can now be concluded using topological tools.

Up to a very small isotopy we can assume that $f(\alpha)$ is disjoint from $\gamma$ and still contained in $N_r(\gamma)=N_r(\beta)$. For safety, we assume that an entire metric tubular neighborhood of $f(\alpha)$ of the form $f(N_\delta(\alpha))$ for some tiny $\delta$ is disjoint from $\gamma$ and contained in $N_r(\gamma)$.

%Since the radius of the tube $f(\mb{T}_{\eta_3}(\alpha))$ is large, we can assume that a metric tubular neighborhood of $\beta$ of the form $N_r(\beta)$ with $r>\ep$ is contained in $f(\mb{T}_{\eta_3}(\alpha))$. 

Denote by $T_\gamma=\partial N_r(\gamma)$ the boundary of the metric neighborhood of $\gamma$ and observe that $T_\gamma\subset f(\mb{T}_{\eta_3})-f(N_\delta(\alpha))$. The complementary region $f(\mb{T}_{\eta_3})-f(N_\delta(\alpha))$ is diffeomorphic to $T_\alpha\times[0,1]$.

{\bf Claim}. $T_\gamma$ is incompressible in $T_\alpha\times[0,1]$. 

{\em Proof of the claim}. In fact, the only possible compressible curve on $T_\gamma$ is the boundary $\partial D_\gamma$ of the compressing disk $D_\gamma$ of the tubular neighborhood of $N_r(\gamma)$. Every other simple closed curve is homotopic in $f(\mb{T}_{\eta_3})$ to a multiple of $\gamma$ and hence it is not trivial (recall $\gamma^k\simeq f(\alpha)\neq0$). However, the curve $\partial D_\gamma$ cannot be compressible in $T_\alpha\times[0,1]$ otherwise it would bound a disk $D'_\gamma$ with interior disjoint from $D_\gamma$ and together they would give a 2-sphere $S^2\cong D_\gamma\cup D'_\gamma$ intersecting once $\gamma$. Such a sphere is homologically non trivial in $f(\mb{T}_{\eta_3})$, but a solid torus does not contain such an object.\qed 

By standard 3-dimensional topology (see \cite[Proposition~3.1 and Corollary~3.2]{W69}), incompressibility implies that $T_\gamma$ is parallel to $T_\alpha\times\{1\}=f(T_\alpha)$. Therefore, $\gamma$ is the core curve $\gamma\cong 0\times S^1$ for another product structure $f(\mb{T}_{\eta_3})\cong D^2\times S^1$ or, in other words, there exists an orientation preserving self diffeomorphism of $f(\mb{T}_{\eta_3})$ that sends $f(\alpha)$ to $\gamma$. Such a diffeomorphism is isotopic to a power of the Dehn twist along the meridian disk of the solid torus, hence it does not change the isotopy class of the core curve. Thus $f(\alpha)$ is isotopic to $\gamma$ and $\gamma=\beta$.

This concludes the proof.   
\end{proof}

\bibliographystyle{alpha}
\bibliography{bibliography}

\end{document}